\documentclass[12pt, reqno, oneside]{amsart}
\usepackage{parskip}
\usepackage[]{babel}
\usepackage[utf8]{inputenc} 
\usepackage{amsmath,amsthm,amsbsy}
\usepackage{amsfonts}
\usepackage{mathrsfs}
\usepackage{bbm}
\usepackage{amssymb}
\usepackage[colorlinks=true, menucolor=blue, linkcolor=blue, citecolor=blue]{hyperref}
\usepackage{enumitem}
\usepackage{multicol}
\usepackage{float}
\usepackage{fancyhdr}
\usepackage{layout}
\usepackage{esint}
\usepackage{array}
\usepackage{makecell}
\allowdisplaybreaks

\usepackage[colorinlistoftodos]{todonotes}
\usepackage{orcidlink} 

\usepackage{mathtools}
\usepackage[foot]{amsaddr}

\theoremstyle{plain}
\newtheorem{theorem}{Theorem}[section]                                          
\newtheorem{proposition}[theorem]{Proposition}                          
\newtheorem{lemma}[theorem]{Lemma}
\newtheorem{corollary}[theorem]{Corollary}

\theoremstyle{definition}
\newtheorem{definition}[theorem]{Definition}
\theoremstyle{remark}
\newtheorem{remark}[theorem]{Remark}
\newtheorem{example}[theorem]{Example}

\newtheorem{assumption}[theorem]{Assumption}

\makeatletter \@addtoreset{equation}{section} \makeatother

\newcommand{\caract}{\mathbbm{1}}
\newcommand{\calF}{\mathcal{F}}
\newcommand{\calA}{\mathcal{A}}

\newcommand{\calB}{\mathcal{B}}

\newcommand{\calP}{\mathcal{P}}

\newcommand{\defeq}{:=}

\newcommand{\N}{\mathbb{N}}     
\newcommand{\Q}{\mathbb{Q}}     
\newcommand{\R}{\mathbb{R}}     
\newcommand{\Prob}{\mathbb{P}}  
\newcommand{\Exp}{\mathbb{E}}   
\newcommand{\goth}[1]{\mathfrak{#1}} 
\newcommand{\ind}[2]{\mathbbm{1}_{#1}\left( #2 \right)}          
\newcommand{\inner}[2]{\left\langle #1 \, , \, #2 \right\rangle} 
\newcommand{\innerp}[2]{\langle #1 \, , \, #2 \rangle} 
\newcommand{\norm}[1]{\left\|#1\right\|}              
\newcommand{\triplet}[3]{\left( #1, #2, #3 \right) }             
\newcommand{\ProbSpace}{\triplet{\Omega}{\mathcal{F}}{\Prob}}    
\newcommand{\abs}[1]{\left| #1 \right|}                          

\newcommand{\quadraVari}[1]{\left\langle  #1  \right\rangle } 

\newcommand{\fle}{\rightarrow}

\newcommand{\operQuadraVari}[1]{\left\langle \!\left\langle #1  \right\rangle \!\right\rangle} 

\newcommand\restr[2]{{
  \left.\kern-\nulldelimiterspace 
  #1 
  \vphantom{\big|} 
  \right|_{#2} 
  }}

\title{Cylindrical Martingale-Valued Measures, Stochastic Integration and SPDEs}

\author{S. Cambronero\,\orcidlink{0000-0001-6758-4942}$^1$}
\author{D. Campos\,\orcidlink{0000-0002-3608-1151}$^2$}
\author{C. A. Fonseca-Mora\,\orcidlink{0000-0002-9280-8212}$^3$}
\author{D. Mena\,\orcidlink{0000-0002-9443-391X}$^4$}
\address{Centro de Investigaci\'{o}n en Matem\'{a}tica Pura y Aplicada \\ Escuela de Matem\'{a}tica, Universidad de Costa Rica}
\email{$^1$ santiago.cambronero@ucr.ac.cr} 
\email{$^2$ josedavid.campos@ucr.ac.cr}
\email{$^3$ christianandres.fonseca@ucr.ac.cr}
\email{$^4$ dario.menaarias@ucr.ac.cr}

\begin{document}
\emergencystretch 3em


\begin{abstract}
We develop a theory of Hilbert-space valued stochastic integration with respect to cylindrical martingale-valued measures. As part of our construction, we expand the concept of quadratic variation, introduced by Veraar and Yaroslavtsev (2016), to the  case of cylindrical martingale-valued measures that are allowed to have discontinuous paths (this is carried out within the context of separable Banach spaces). Our theory of stochastic integration is applied to address the existence and uniqueness of solutions to stochastic partial differential equations in Hilbert spaces. \end{abstract}

\maketitle



\tableofcontents

\section{Introduction}

In recent years, there has been an increasing interest in the usage of cylindrical stochastic processes as models for the perturbation of infinite dimensional systems, in particular, of stochastic partial differential equations (e.g.  \cite{FerrarioOlivera:2019, KumarRiedle:2020, LiuZhai:2016, PriolaZabczyk:2011, Riedle:2015, Riedle:2018, SunXieXie:2020, VeraarYaroslavtsev:2016}). In most of these works, the cylindrical process is a cylindrical martingale, or more  generally a cylindrical semimartingale. A cylindrical martingale $M$ on a Banach space $X$ is a linear operator such that, for each $x^* \in X^*$, $M(x^*)$ is a (real-valued) martingale. Usually, suitable continuity conditions are required for $M$. 

Another popular alternative for the modeling of the noise of a stochastic partial differential equation is a martingale-valued measure (e.g. \cite{Applebaum:2006, ChongKevei:2019, ConusDalang:2008, Dalang:1999, DalangMueller:2003, FoondunKhoshnevisan:2009}). This concept was introduced by Walsh in \cite{Walsh:1986} and was motivated by the space-time Gaussian white noise (see Example \ref{ex:whitenoisemeasure}). Roughly speaking, a martingale-valued measure is a family $(M(t,A):t \geq 0, A \in \mathcal{A})$ such that $(M(t,A):t \geq 0)$ is a real-valued square integrable martingale for each $A \in \mathcal{A}$ and $M(t, \cdot)$ is an $L^{2}$-valued finitely additive measure on $\mathcal{A}$ for each $t \geq 0$. Here $\mathcal{A}$ is a ring of Borel subsets of a topological space $U$. See Definition \ref{defiHValuedlOrthoMartinValuMeasure} for further details. 

Motivated by the above, in this paper we introduce a theory of stochastic integration with respect to  cylindrical martingale-valued measures and study some of its properties. We apply our theory to study existence and uniqueness for some classes of stochastic partial differential equations (SPDEs). One distinguishing feature of our work is that our cylindrical martingale-valued measures are allowed to have discontinuous paths.

The  concept of a cylindrical martingale-valued measure is a hybrid between the definitions of cylindrical martingale and of martingale-valued measure. Indeed, in this case we have a family $(M(t,A):t \geq 0, A \in \mathcal{A})$ such that, for each $A \in \mathcal{A}$, $(M(t,A):t \geq 0)$ is a cylindrical square integrable martingale on a Banach space $X$ and $M(t, \cdot)$ is an $L^{2}$-valued finitely additive measure on $\mathcal{A}$ for each $t \geq 0$. The definition of cylindrical martingale-valued measures is introduced in Section \ref{sectCylinMartingValued}. We explore the usefulness of this concept by giving some examples.   
 
A particular family of cylindrical martingale-valued measures with independent increments was introduced in earlier works \cite{AlvaradoFonseca:2021, FonsecaMora:2018}. In these definitions the authors assumed the existence of a  particularly simple covariance structure, which was suitable to study L\'{e}vy processes (see  Examples \ref{exampleCMVMWithFamilySeminorms} and \ref{exampleLevyHValuedMVM}). 
The existence of the quadratic variation of a (general) cylindrical martingale-valued measure has not been addressed, and this is precisely the first main objective of our work. 

The definition of a quadratic variation for a cylindrical martingale-valued measure requires the introduction of new concepts. Indeed, since these objects are not stochastic processes (cylindrical or classical), we cannot
directly apply the theory of  quadratic variation for Banach space-valued martingales nor that of  cylindrical martingales which is currently available in the literature, for example, in the works in \cite{DiGirolamiFabbriRusso:2014, Dinculeanu:2000, MetivierPellaumail, MikuleviciusRozovskii:1998, Ondrejat:2005, VeraarYaroslavtsev:2016}.  

The construction and study of properties of the quadratic variation for a cylindrical martingale-valued measure is carried out in Section \ref{sectQuadraticVariation}. We start in Section \ref{subsectDefiQuadraVariation} with the definition of a quadratic variation. 
Since our definition of cylindrical martingale-valued measure encloses both the concept of cylindrical martingale and of martingale-valued measure, we found it convenient to formulate our definition of quadratic variation as a supremum of a family of measures (see Section \ref{sectPreliminars}). This follows what is done by Veraar and Yaroslavtsev in \cite{VeraarYaroslavtsev:2016} (there for cylindrical square-integrable martingales with continuous paths), but taking this family of measures as the intensity measures (defined in Section \ref{sectHilbertMartValuedMeasures}) for the real-valued martingale-valued measures $(M(t,A)(x^*): t \geq 0, A \in \mathcal{A})$, where $x^{*} \in X^{*}$ (see Definition \ref{defiQuadraticVariation} for further details). 

Our definition allows the existence of several quadratic variation processes. In order to establish sufficient conditions for uniqueness, we introduce the sequential boundedness property on the family of intensity measures $(\nu_{x^*}: x^* \in X^*)$ for $M(t,A)(x^*)$ (Definition \ref{defiQuadraticVariation}). This property helps us to obtain  a unique simple expression for our quadratic variation $\operQuadraVari{M}$, as a supremum of measures on a countable and dense collection of elements in the unit sphere of $X^*$ (see Theorems \ref{theoSuffiCondiExistQuadraVariat} and \ref{theoUniquenessQuadraticVariation}). The sequential boundedness property is obtained in a natural way in the case of cylindrical martingales with continuous paths (see Example \ref{examCylinMartinContPaths}), thus, generalizing the work in \cite{VeraarYaroslavtsev:2016} for this case.
Moreover, the sequential boundedness property is a weaker condition than the uniform continuity of the family of intensity measures (see
Proposition \ref{propUniformUCPImpliesWeakBoundedness}).

The characterizations obtained in the previous results allow us to explicitly construct the quadratic variation in some cases, for example, in the case of a cylindrical L\'evy process (see 
Example \ref{examCMVMCylindriLevyProcesses}) or for Hilbert-space valued martingale-valued measures determined by a L\'evy process (see  Examples \ref{exampleCMVMWithFamilySeminorms} and \ref{exampleLevyHValuedMVM}).  We also include some examples in which some cylindrical martingale measures do not have a quadratic variation (see Example \ref{examCounterExampleCMVMStochasticIntegral} and Remark \ref{remaNotExistQuaVariaCyliPoissonIntegral}).

In Section \ref{subsectQuadraVariationoperator},  the main purpose is to obtain a Radon-Nikodym representation $ d\alpha_M = Q_M\, d\operQuadraVari{M} $ (see Theorem \ref{theoExistenCovariaOperatorQ}), where $Q_M$ takes values in $\mathcal{L}(X^*, X^{**})$ and the vector-valued measure $\alpha_M$  corresponds to a covariation operator \eqref{eqDefiAlphaMRandomMeasure}.   
Our approach is similar in nature to the one in \cite{VeraarYaroslavtsev:2016}, however, we use a (random) vector-valued measure instead of vector-valued stochastic process, thus implying substantial differences in the proof. Examples of our result to  cylindrical and Hilbert space valued L\'{e}vy process are given (see Examples
\ref{exampleCylinMartingaleDeltaLocBounCova}, \ref{exampleCMVMWithFamilySeminorms} and \ref{exampleLevyHValuedMVM}). 

Section \ref{subSectCylindricalPRM} is devoted to the study of cylindrical martingale-valued measures defined by Poisson random measures. We show that the concept of `cylindrical Poisson random measure' cannot be defined (under reasonable assumptions) but that this  does not prevent us from defining a cylindrical martingale valued measure which uses a Poisson random
measure as a building block. To illustrate this, we give two examples of constructions for which the quadratic variation $\operQuadraVari{M}$ exists and can be calculated explicitly (see Examples \ref{exampleLinearCPoissonRM} and \ref{exampleNonLinearCylinPRM}). 

In Section \ref{sectStochasticIntegration} we develop a theory of Hilbert-space valued stochastic integration for  operator-valued processes  with respect to a cylindrical martingale-valued measure which possesses a quadratic variation. 

The construction of the integral is carried out in Section \ref{subsecConstrIntegral} and for simple integrands follows a global radonification approach. To be more specific, we consider Hilbert spaces $H$ and $G$, and a simple integrand of the form $\Phi(r,\omega,u)= \mathbbm{1}_{]a,b]}(r) \mathbbm{1}_{F}(\omega) \mathbbm{1}_{A}(u) S$,  
where $0 \leq a< b \leq T$, $F \in \mathcal{F}_{a}$, $A \in \mathcal{A}$ and $S$ is a Hilbert-Schmidt operator from $H$ into $G$. Then we define
\begin{equation} \label{eqIntroDescriInteg}
\left(\int_{0}^{t}\int_{U} \Phi(r,u) M(dr,du),g \right)_{G} = \mathbbm{1}_{F} (M(b \wedge t,A)-M(a \wedge t,A))(S^{*}g),
\end{equation}
for every $g \in G$ and $t \in [0,T]$. Observe that because $S$ is a Hilbert-Schmidt operator from $H$ into $G$, the equality above defines a $G$-valued square integrable martingale with second moments. With the help of the operator $Q_{M}$ one can extend the stochastic integral to a large class of integrands (see Definition \ref{defiNewIntegrand}) that satisfy an It\^{o} isometry (see \eqref{theoItoIsometry}). Some explicit calculations for our stochastic integral are given in Examples \ref{exampleStochIntegHilbertLevyProcess} and \ref{exampleStochIntegLinearCylinPoissonRM}.

Some of the basic properties of the stochastic integral are given in Section \ref{subsecPropeIntegral}, and a further extension to a larger class of integrands is given in Section \ref{subsecExtensIntegral}. Furthermore, we prove a stochastic Fubini theorem in Section \ref{subsecStochFubini}. The theory developed in this article extends the corresponding theory in \cite{AlvaradoFonseca:2021}, therefore generalizes some other theories of stochastic integration in Hilbert spaces currently available in the literature (see the discussion in Section 4.2 in \cite{AlvaradoFonseca:2021}). Since our theory of quadratic variation was developed under the context of Banach spaces, we hope we could extend the definition of the stochastic integral to some classes of Banach spaces, for example to UMD spaces as in \cite{VeraarYaroslavtsev:2016}. 

As an application of our theory of stochastic integration, in Section \ref{SectSPDEs} we study the existence and uniqueness of weak and mild solutions to the following class of stochastic partial differential equations:
\begin{equation*}
  dX_t = \left[ AX_t + B(t,X_t)\right]dt + \int_U F(t,u,X_t) M(dt,du). 
\end{equation*}
Here $A$ is the infinitesimal  generator for a $G$-valued $C_{0}$-semigroup, $M$ is a cylindrical  martingale-valued measure on $H$, and $B$ and $F$ are coefficients satisfying suitable conditions.  In Example \ref{Ex:LevynoiseSPDE}, we analyze the case of a stochastic heat equation with linear Lévy noise and Dirichlet boundary conditions, which provides examples of general discontinuous martingales, for which our theory applies.

Sufficient conditions for equivalence between weak and mild solutions are given in Section \ref{subsectWeakMildSoluc}.  Finally, in Section \ref{subsecExistUniqueSPDEs} we introduce some growth and Lipschitz type conditions for the coefficients $B$ and $F$ in order to show the existence and uniqueness of weak and mild solutions to stochastic partial differential equations with multiplicative cylindrical martingale-valued measure noise (see Theorem \ref{theoExisteUniqueness} and Example \ref{examplExisteUniqSPDE}).

\section{Preliminaries and Notation}
\label{sectPreliminars}

\subsection{Stochastic and cylindrical processes}

For a Banach space $X$ we denote by $X^*$ its (strong) dual space. In this work, the norm of the underlying Banach space will be often denoted by $\norm{\cdot}$, but when it is necessary to emphasize the space, we use the notation $\norm{\cdot}_{X}$.  We will use the notation $H$ for a Hilbert space with inner product $(\cdot,\cdot)_{H}$.  We identify the dual of a Hilbert space with the space itself. For a Hausdorff topological space $U$, its Borel $\sigma$-algebra will be denoted by $\calB(U)$.

For any two Banach spaces $X$ and $Y$, the Banach space of bounded linear operators from $X$ into $Y$ will be denoted by  $\mathcal{L}(X,Y)$. We will denote the space of real-valued bounded bilinear forms on $X \times Y$ by $\goth{Bil}(X,Y)$. Trace class operators on a Hilbert space $H$ will be denoted by $\mathcal{L}_{1}(H)$.

Throughout this work we assume that $(\Omega, \mathcal{F}, \Prob)$ is a complete probability space equipped with a filtration $(\mathcal{F}_{t})_{t \in \R_{+}}$ that satisfies the \emph{usual conditions}, i.e. it is right continuous and $\mathcal{F}_{0}$ contains all $\Prob$-null sets.
We denote by $L^{0} (\Omega, \mathcal{F}, \Prob)$ the space of equivalence classes of real-valued random variables defined on $(\Omega, \mathcal{F}, \Prob)$. The space $L^{0} (\Omega, \mathcal{F}, \Prob)$ will be always equipped with the topology of convergence in probability and in this case it is a complete, metrizable, topological vector space.

Let $X$ be a Banach space and let $(S,\Sigma)$ be a measurable space. A function $\alpha: \Sigma \to X$ is called a {\it vector measure}, if whenever $E_1$ and $E_2$ are disjoint members of $\Sigma$ then $\alpha(E_1 \cup E_2)=\alpha(E_1)+\alpha(E_2)$. If, in addition, $\displaystyle \alpha \Bigl( \bigcup_{n=1}^\infty E_n\Bigr)=\sum_{n=1}^\infty \alpha(E_n)$ in the norm topology of $X$ for all collection $(E_n)$ of pairwise disjoint members of $\Sigma$, then $\alpha$ is called a {\it countably additive vector measure} or simply, $\alpha$ is countably additive. In any case, the {\it variation} of $\alpha$ is the extended non-negative function $|\alpha|$ whose value on a set $E \in \Sigma$ is defined by
$$|\alpha|(E)=\sup_{\Pi}\sum_{A \in \Pi}\|\alpha(A)\|,$$
where the supremum is taken over all finite partitions $\Pi$ of $E$ by members of $\Sigma$. If $|\alpha|(S)<\infty$, then $\alpha$ will be called a {\it vector measure of bounded variation}. (See \cite{DiestelUhl:1977} for further details).

A mapping $\mu: \Omega \times \Sigma \rightarrow [0,\infty] $ will be called a \emph{random measure} if for all $E \in \Sigma$, $\omega \mapsto \mu(\omega,E)$ is measurable and for (almost) all $\omega \in \Omega$, $\mu(\omega, \cdot) $ is a measure on $(S,\Sigma)$.

The predictable $\sigma$-algebra on $\Omega \times [0,\infty)$ (see Chapter 10 in \cite{Kallenberg:2021}) is denoted by $\mathcal{P}$ and for any $T>0$ we denote by $\mathcal{P}_{T}$ the restriction of $\mathcal{P}$ to $\Omega \times [0,T]$. Let $(S,\Sigma)$ be a measurable space and let $(\mu_{t}: t \geq 0)$ be a family of random measures on $(S,\Sigma)$. We say that $(\mu_{t}: t \geq 0)$ is \emph{predictable} if for any given $A \in S$ the mapping $(\omega, t) \mapsto \mu_{t}(\omega)(A)$ is predictable.  

 Let $I_{T}=[0,T]$ for $0<T<\infty$ or $I_{T}=[0,\infty)$ for $T=\infty$. In either case, denote by $\mathcal{M}_{T}^{2}$ the linear space of all the real-valued, c\`{a}dl\`{a}g, mean-zero, square integrable martingales on the time interval $I_{T}$. The space $\mathcal{M}_{T}^{2}$ is Banach when equipped with the norm 
$$
\norm{m}_{\mathcal{M}_{T}^{2}}= \left( \sup_{t \in I_{T}} \Exp \left[  \abs{m(t)}^{2} \right] \right)^{1/2}.
$$

Let $X$ be a Banach space. A \emph{cylindrical random variable} on $X$ is a linear and continuous operator $Z:X^* \rightarrow L^{0} (\Omega, \mathcal{F}, \Prob)$. A family of cylindrical random variables $Z=(Z_{t}: t \in I_{T})$ on $X$ is called a \emph{cylindrical stochastic process} on $X$. A cylindrical stochastic process  $M=(M_{t}: t \in I_{T})$ is called a \emph{cylindrical mean-zero square integrable martingale} on $X$ is for every $x^* \in X^*$ we have $M(x^*) \in \mathcal{M}_{T}^{2}$. 

Below we review some basic properties of the quadratic (co)variation. See Section 20 in \cite{Metivier} for proofs and further details.   

Let $(H, (\cdot, \cdot)_{H})$ be a separable Hilbert space and let $ {Y}=( {Y}_{t}: t \geq 0)$ and $ {Z}=( {Z}_{t}: t \geq 0)$ denote two $H$-valued square integrable martingales. There exists a unique (up to indistinguishability) predictable, real-valued process $\innerp{ {Y}}{ {Z}}$ with paths of finite variation such that $\innerp{ {Y}}{ {Z}}_{0}=0$ and $( {Y}, {Z})_{H} - \innerp{ {Y}}{ {Z}}$ is a martingale. We call $\inner{ {Y}}{ {Z}}$ the \emph{(predictable) quadratic covariation of $ {Y}$ and $ {Z}$}. We denote $\innerp{ {Y}}{ {Y}}$ by $\langle  {Y} \rangle$ and call it the \emph{(predictable) quadratic variation of $ {Y}$}; this process is increasing. One can show that for any given $t \geq 0$
\begin{equation}\label{eqPredQuadVariaAsLimitOverPartition}
\sum_{j=1}^{\infty} \Exp \left[ \left(   {Y}_{t \wedge t^{n}_{j+1}}-  {Y}_{t \wedge t^{n}_{j}},   {Z}_{t \wedge t^{n}_{j+1}}-  {Z}_{t \wedge t^{n}_{j}} \right)_{H} \, \vline \, \mathcal{F}_{t^{n}_{j}} \right] \overset{\Prob}{\rightarrow} \innerp{  {Y}}{  {Z}}_{t},     
\end{equation}
where $\{0=t^{n}_{0}< t^{n}_{1} < \cdots < t^{n}_{k} < \cdots \}_{n\geq 1}$ is a sequence of partitions of $[0, \infty)$ such that $t^{n}_{j} \rightarrow \infty$ as $j \rightarrow \infty$ and $\delta_{n}=\sup_{j} \abs{ t^{n}_{j+1}- t^{n}_{j}} \rightarrow 0 $ as $n \rightarrow \infty$. 
For any $h \in H$, one can easily check that $Y(h)\defeq ((  {Y}_{t},h)_{H}: t \geq 0)$ is a square integrable real-valued martingale, hence has (predictable) quadratic variation $\quadraVari{Y(h)}$. We can relate $\langle   {Y} \rangle$ and $\quadraVari{Y(h)}$ via \eqref{eqPredQuadVariaAsLimitOverPartition} as follows.  

Given $h,g \in H$, for any sequence of partitions $\{0=t^{n}_{0}< t^{n}_{1} < \cdots < t^{n}_{k} < \cdots \}_{n\geq 1}$  of $[0, \infty)$ as described above, we have
\begin{multline*}
\sum_{j=1}^{\infty} \Exp \left[ \left( Y(h)_{t \wedge t^{n}_{j+1}}-Y(h)_{t \wedge t^{n}_{j}} \right) \left(  Y(g)_{t \wedge t^{n}_{j+1}}-Y(g)_{t \wedge t^{n}_{j}} \right)\, \vline \, \mathcal{F}_{t^{n}_{j}} \right] \\
\leq 
\norm{h} \norm{g} \sum_{j=1}^{\infty} \Exp \left[ \norm{   {Y}_{t \wedge t^{n}_{j+1}}-  {Y}_{t \wedge t^{n}_{j}}}^{2} \, \vline \, \mathcal{F}_{t^{n}_{j}} \right].     
\end{multline*} 
Then taking limits as $n \rightarrow \infty$ and by the uniqueness of limits in probability we conclude from 
\eqref{eqPredQuadVariaAsLimitOverPartition} that $\Prob$-a.e.
\begin{equation}\label{eqReqlVsHilbertPrediQuadrVaria}
\inner{Y(h)}{Y(g)}_{t} \leq \norm{h} \norm{g} \langle   {Y} \rangle_{t}.     
\end{equation}

\subsection{Supremum of measures}
\label{sectsupmeasures}
In this section we summarize some results on the supremum of a family of measures introduced in \cite{VeraarYaroslavtsev:2016}, and try to write them in a more general context to make them useful when dealing with quadratic variations. For measures $\mu$ and $\nu$ defined on a measurable space $(S,\Sigma)$, we write $\nu \leq \mu$ if
$$
\forall A \in \Sigma \quad \nu(A) \leq \mu(A).
$$
This defines a partial ordering on the class $\mathcal{M}_{+}(S,\Sigma)$ of all measures on $(S,\Sigma)$. Given a family $(\mu_\alpha)_{\alpha\in\Lambda}$ of measures on $(S,\Sigma)$, there is always a supremum, that is, there exists $\check{\mu} \in \mathcal{M}_{+}(S,\Sigma)$ that satisfies
\begin{enumerate}
\item For each $\alpha\in \Lambda$, $\mu_\alpha \leq \check{\mu}$
\item If $\mu \in \mathcal{M}_{+}(S,\Sigma)$ and $\mu_\alpha \leq \mu$ for each $\alpha\in\Lambda$, then $\check{\mu} \leq \mu$.
\end{enumerate}
The following lemma establishes the existence we have just mentioned. We write it in the generality we need. We omit the proof, since the one in \cite[Lemma 2.6]{VeraarYaroslavtsev:2016}  works perfectly well if we only start with a family of sub-additive set functions. We refer to \cite{CCFM:SUP} for a more general treatment of the concept. 

\begin{lemma}
\label{lemmaSupExplicit}
Let $(\mu_\alpha)_{\alpha\in\Lambda}$ be a family of non-negative, finitely sub-additive set functions on a measurable space $(S,\Sigma)$. For $A\in\Sigma$, let $\mathscr{P}(A)$ be the collection of finite partitions of $A$,  by elements of $\Sigma$. Define the set function 
$\check{\mu}$ by 
\begin{equation}\label{supremumofmeasures}
\check{\mu}(A) = \sup_{\Pi\in \mathscr{P}(A)}\,\, \sum_{C\in\Pi}\,\, \sup_{\alpha\in\Lambda} \mu_\alpha(C), \qquad A \in \Sigma.
\end{equation}
Then $\check{\mu}$ is the smallest (in particular, unique) finitely additive measure on $(S,\sigma)$ that satisfies $\check{\mu}\geq \mu_\alpha$ for each $\alpha\in \Lambda$. If each $\mu_\alpha$ is sub-additive, then $\check{\mu}$ is a measure on $(S, \Sigma)$.
\end{lemma}

The functions $\mu_\alpha$ might be all finite and still have $\check{\mu}(A) =\infty$ at each nonempty set; an easy example can be obtained with $S$ finite and $\mu_n = n \mu_0$, $\mu_0$ being the counting measure. But if there is a finite measure dominating each $\mu_\alpha$, it follows that $\check{\mu}$ is a finite measure. \smallskip

We call $\check{\mu}$ the supremum measure of the family $(\mu_\alpha)$ and denote it by $\displaystyle
\sup_{\alpha\in\Lambda} \mu_\alpha.$
Notice that
$$
\left( \sup_{\alpha\in\Lambda} \mu_\alpha \right)(A) \geq \sup_{\alpha\in \Lambda} \mu_\alpha(A).
$$
The right hand side is computed, for each $A$, as the classical supremum of a set of real numbers, which in general does not define a measure, not even for a family of two measures.

The following lemma is a generalization of Lemma 2.8 in \cite{VeraarYaroslavtsev:2016} for the case of a family of random measures. 
\begin{lemma}
\label{lemmIntegralOfSupremum}
Let $(S,\Sigma,\nu)$ be a measure space, $(\Omega,\calF,\mathbb{P})$ a probability space. Let $\mathscr{F}$ be a family of measurable functions from $\Omega\times  S$ into $[0,\infty]$ and $(f_j)_{j\in\N}$ a sequence in $\mathscr{F}$. Define $\overline{f} = \sup_{j\geq 1} f_j $ and assume that $\sup_{f\in \mathscr{F}} f = \overline{f}$. For each $f\in \mathscr{F}$ let $\mu_f$ be the random measure defined by
$$
\mu_f(\omega,B) =\int_B f(\omega,\cdot) d\nu, \quad B\in \Sigma.
$$
If we define $\check{\mu}(\omega,\cdot) := \sup_{f\in \mathscr{F}} \mu_f(\omega,\cdot)$, then $\check{\mu}(\omega,\cdot) = \sup_{j\geq 1} \mu_{f_j}(\omega,\cdot)$ and
\begin{equation}
\label{supremum of random measures}
\check{\mu}(\omega,B) = \int_B \overline{f}(\omega,\cdot) d\nu.
\end{equation}
In particular $\check{\mu}$ is a random measure on $(S,\Sigma, \nu)$.
\end{lemma}
\begin{proof}
For fixed $\omega$ we apply Lemma $2.8$ in \cite{VeraarYaroslavtsev:2016} to obtain both identities. Tonelli's theorem allows us to conclude that each $\mu_f$ is a random measure, so the same is true for $\check{\mu}$ because of
(\ref{supremum of random measures}). 
\end{proof}

In the case of a countable number of measures, Lemma $2.9$ of \cite{VeraarYaroslavtsev:2016} gives us an expression for $\check{\mu}$ as the limit
$$
\check{\mu}(A) = \lim_{N\fle \infty} \left(\sup_{1\leq n\leq N} \mu_n \right)(A).
$$

Remark 2.10 in \cite{VeraarYaroslavtsev:2016} also shows that, in the case $S=\R$ and $\Sigma = \mathcal{B}_{\R}$ one has
$$
\check{\mu}(a,b] = \sup_{\Pi \in \mathscr{R}} \sum_{C \in \Pi} \sup_{\alpha} \mu_{\alpha}(C),
$$
where $\mathscr{R}$ is the family of  all Riemann-type partitions of $(a,b]$, and the endpoints of the sub-intervals in such partitions can be taken rational (with the possible exception of $a$ and $b$). These ideas have the following useful corollary, which we present in a bit more general setting, as will be needed afterwards. 

\begin{corollary}
\label{CoroSupPredRandMeas}
The supremum of a countable family of random measures over $(\R,\mathcal{B}_{\R})$ is itself a  random measure. Moreover, if $\mu_n(\omega,\cdot) \leq \mu_{n+1}(\omega,\cdot)$ for each $n$, then for a.e. $\omega$
$$
\left( \sup_{n\in\N} \mu_n(\omega,\cdot) \right)(A) = \lim_{n\fle\infty} \mu_n(\omega,A), \quad A\in \calB(\R).
$$
\end{corollary}

The following equivalence for an infinite sum of measures is not difficult to establish.
\begin{corollary}
\label{CoroSumRandMeas}
The infinite sum of random measures is a random measure. Moreover
$$
\left( \sum_{n=1}^\infty \mu_n\right)(A) =  \sum_{n=1}^\infty \mu_n(A) = \left(\sup_{F} \sum_{n\in F} \mu_n \right)(A)
$$
where $F$ runs over the finite subsets of $\mathbb{N}$.
\end{corollary}

\section{Hilbert space-valued martingale-valued measures} \label{sectHilbertMartValuedMeasures}

Let $H$ be a separable Hilbert space. In this section we recall the definition of $H$-valued martingale-valued measures, originally introduced by Walsh in \cite{Walsh:1986} in the real-valued setting. 

Assume  that $U$ is a Hausdorff topological space which is Lusin in the sense that it is homeomorphic to a Borel subset of the line. 
Let $ \mathcal{A}$ be a ring of Borel subsets of $U$. We also consider an underlying probability space $(\Omega,\calF,\Prob)$.

\begin{definition}
\label{definitionSigmafiniteMeasure}
A set function $N:\Omega\times\calA\fle H$ is called a $\sigma$-finite $L^2 (\Omega, \mathcal{F}, \Prob; H)$-valued measure if it satisfies:
\begin{enumerate}
    \item $N$ is $L^2$-valued: \label{H-valuedCondition}$\Exp[\norm{N(A)}^2] <\infty$ for every $A\in \calA$. 
    \item $N$ is finitely additive: \label{finiteAdditivity} if $A,B\in\calA$ are disjoint then $\Prob$-a.e.
    $$N(A\cup B) = N(A) + N(B).$$
    \item \label{sigmaFiniteCondition} $N$ is $\sigma$-finite: there exists a sequence $(U_{n})_{n \in \N}$ in $\mathcal{B}(U)$ such that $U_n \uparrow U$, $\calB(U_n) \subseteq \calA$ for each $n$ and
    $$
    \sup_{A \in \calB(U_n)} \Exp [ \norm{N(A)}^{2}] < \infty.
    $$
    \item \label{countableAditivity} $N$ is \emph{countably additive} on each $U_n$ (as an $L^2 (\Omega, \mathcal{F}, \Prob; H)$-valued function): For each sequence $(A_{j})_{j \in \N}$ in $\mathcal{B}(U_{n})$ decreasing to $\emptyset$, we have $\Exp[\norm{N(A_{j})}^{2}] \fle 0$.
    \item \label{extensionCondition} For each $A\in \calA$, $N(A)$ is the $L^2$-limit of $N(A\cap U_n)$.
\end{enumerate}
\end{definition}

As we mentioned before, this definition is based on that of \cite{Walsh:1986}. We have adapted it to the Hilbert space-valued case and implemented some minor modifications. For instance, property \ref{extensionCondition}, which we made part of the definition, is treated in \cite{Walsh:1986} as an extension procedure. Such an extension could actually result in a modification of $N$ on some sets $A\in \calA$, or might even reduce the original ring\footnote{Consider $U=\R$ and 
$\calA = \{ A\in \calB(\R):A\text{ is bounded or } A^c \text{ is bounded} \}$.
Define the deterministic set function $N(A) = 0$ for $A$ bounded and $N(A)=1$ for $A^c$ bounded. This clearly satisfies properties (i)-(iv) in Definition \ref{definitionSigmafiniteMeasure}, with $U_n=[-n,n]$. The extension suggested in \cite{Walsh:1986} implies redefining $N(A) =0$ for each $A\in \calB(\R)$. In the same context, define $N(A) = |A|$ for $A$ bounded and $N(A)=1-|A^c|$ when $A^c$ is bounded. In this case, the $L^2$-limit of $N(A\cap U_n)$ exists if, and only if, $|A|<\infty$.}. We do observe that this condition would be superfluous if, for instance, $N$ is countably additive on $\mathcal{A}$.

\begin{remark}\label{remarkBorelSigmaAlgCountablyGenerated}
Since our space is Lusin, there exists a sequence of sets that generates $\mathcal{B}(U)$.  Given a $\sigma$-finite mapping $N$, by taking intersections with the generating countable collection if necessary, we can assume the sequence $(U_n)$ generates $\mathcal{B}(U)$, instead of being increasing.
\end{remark}

\begin{definition}\label{defiHValuedlOrthoMartinValuMeasure}
An $H$-valued \emph{orthogonal martingale-valued measure} is a collection 
$$
M=(M(t,A): t \geq 0, A \in \mathcal{A})
$$ 
of $H$-valued random variables satisfying:
\begin{enumerate} 
\item For any $A \in \mathcal{A}$, $M(0,A)= 0$ $\Prob$-a.e.
\item For any $A \in \mathcal{A}$, $M(A) = (M(t,A))_{t \geq 0}$, is a mean-zero, square integrable $H$-valued martingale.
\item If $t>0$, $M(t,\cdot): \mathcal{A} \rightarrow L^{2}(\Omega, \mathcal{F}, \Prob; H)$ is a $\sigma$-finite $L^{2}(\Omega, \mathcal{F}, \Prob; H)$-valued measure. 
\item For $t > 0$, $\inner{M(A)}{M(B)}_{t}=0$ whenever $A$ and $B$ in $\mathcal{A}$ are disjoint.
\end{enumerate}
\end{definition}

By the classical definition, $\quadraVari{M(A)}_t$ is increasing in $t$ for fixed $A$. It is also increasing in $A$ for fixed $t$, according to the following.

\begin{lemma}
\label{lemmaAdditiveVariation}
If $M$ is an $H$-valued orthogonal martingale-valued measure, 
$\langle M(\cdot) \rangle_t$ is additive on $\calA$. More precisely, given 
$A,B \in \calA$ disjoint we have $\langle M(A\cup B) \rangle_t = \langle M(A) \rangle_t + \langle M(B) \rangle_t$ $\Prob$-a.e.
Moreover, for $A \subseteq B$ and $t\geq 0$ we have $\langle M(A) \rangle_t \leq  \langle M(B) \rangle_t$, $\Prob$-a.e.
\end{lemma}

\begin{proof}
If $A,B \in \calA$ are disjoint, we have $\Prob$-a.e.
$$
\langle M(A\cup B) \rangle_t = \langle M(A) \rangle_t + \langle M(B) \rangle_t + 2 \langle M(A),M(B) \rangle_t = \langle M(A) \rangle_t + \langle M(B) \rangle_t 
$$
because of the orthogonality. The final assertion follows easily from this.
\end{proof}

Let us consider an $H$-valued orthogonal martingale-valued measure $M$. We fix an interval $[0,T]$ and define, following \cite{Walsh:1986}:
\begin{equation}
\label{eqDefMuWalsh}
  \mu(A) = \Exp [ \norm{M(T,A)}^2] = \Exp [ \langle M(A) \rangle_T ], \quad \forall A \in \mathcal{A}.  
\end{equation}
Let $(U_n)$ be the sequence corresponding to $M(T,\cdot)$ in \ref{sigmaFiniteCondition} of Definition \ref{definitionSigmafiniteMeasure}. Being the map $t\mapsto \Exp[\norm{M_t(A)}^2]$ increasing, the same family $(U_{n})$ works for all $M(t,\cdot)$, $0 \leq t \leq T$. In fact, by Lemma \ref{lemmaAdditiveVariation} $\mu$ is finitely additive and
$$
  \sup_{A\in\calB(U_n)} \Exp [\norm{M(t,A)}^{2}] = \Exp [\norm{M(t,U_n)}^{2}] \leq \Exp [\norm{M(T,U_n)}^{2}] = \mu(U_n) < \infty.
$$
We summarize it in the following lemma.

\begin{lemma}\label{lemmaHValuedMeasMuFiniteAdditive}
Let $M$ be an $H$-valued orthogonal martingale-valued measure. The finitely additive measure $\mu$ defined on $\calA$ by \eqref{eqDefMuWalsh}, is $\sigma$-finite. Moreover, $\mu$ is a (countably additive) measure on each $(U_n,\calB(U_n))$.
\end{lemma}

The following theorem is essentially an $H$-valued version of Proposition $2.7$ of \cite{Walsh:1986}. Its proof can be carried out by following the steps on that reference, with minor changes. We have decided to include a self-contained proof in the Appendix \ref{appendix} by two reasons. First, there are details in the original proof that we consider important to emphasize and, second, the ideas of the proof will be crucial to other results in this paper.

\begin{theorem}\label{theoExistMeasMuHValuedMVM}
Let $(M_t(A): 0\leq t\leq T,A\in {\mathcal A} )$ be an orthogonal $H$-valued martingale-valued measure, where $\calA$ is a sub-ring of $\calB(U)$. Then there exists a family $(\nu_t,0\leq t\leq T)$ of random $\sigma$-finite measures on $(U,\calB(U))$ such that:
\begin{enumerate}
\item[(i)] $(\nu_t,0\leq t\leq T)$ is predictable.
\item[(ii)] For all $A\in \calA$, $t\mapsto \nu_{t}(A)$ is right-continuous and increasing.
\item[(iii)] For all $0 \leq t\leq T$ and $A\in {\mathcal A}$ we have $\Prob(\nu_t(A)=\langle M(A) \rangle_t) =1$.
\end{enumerate}
\end{theorem}

Next corollary constructs a random measure $\nu$ on $\mathcal{B}(\R_{+}) \otimes \mathcal{B}(U)$, based on the family $(\nu_t: t\geq 0)$ of random random measures on $\mathcal{B}(U)$. We will say that such a random measure $\nu$ is {\em predictable} if the family $(\nu_t)$ is.

\begin{corollary} \label{coroExistMeasMuHValuedMVM}
Given an $H$-valued orthogonal martingale-valued measure $M=(M(t,A): t \geq 0, A \in \mathcal{A})$, there exists a predictable $\sigma$-finite random measure $\nu$ on $\mathcal{B}(\R_{+}) \otimes \mathcal{B}(U)$ such that for every $t \geq 0$ and $A \in \calA$,
\begin{equation*}
\Prob \left( \{ \omega \in \Omega : \nu(\omega)\left( [0,t] \times A \right)  = \quadraVari{ M(A) }_{t}(\omega)  \} \right)=1.  \end{equation*}
\end{corollary}
\begin{proof} First, by a standard argument we can construct a family  $(\nu_t: t\geq 0 )$ of random $\sigma$-finite measures on $(U,{\mathcal B}(U))$ satisfying the statements in Theorem \ref{theoExistMeasMuHValuedMVM} for each $T>0$. We can therefore define $\nu$ by  first setting $\nu([0,t] \times A)=\nu_{t}(A)$ for $t \geq 0$ and $A \in \mathcal{A}$ and then extending it to $\mathcal{B}(\R_{+}) \otimes \mathcal{B}(U)$, using the fact that the collection of all rectangles $[0,t] \times A$ generates the product $\sigma$-algebra.
\end{proof}

\begin{remark}
\label{nuregularizesquadvar}
The previous corollary provides us with a regularization for the random mapping $(t,A) \mapsto \quadraVari{M(A)}_t$, which in general, does not have to behave like a random measure.    If such mapping is verified to be already predictable and countably additive on the measurable rectangles of $\mathcal{B}([0,T]) \otimes \mathcal{B}(U)$, then the previous construction is not needed.
\end{remark}

The following theorem takes advantage of ideas on the proof of Theorem \ref{theoExistMeasMuHValuedMVM}. It receives an inequality between two random measures that is satisfied almost everywhere for each individual rectangle $(s,t]\times A$, $A \in \calA$, and shows that it actually holds almost everywhere all over $\calB(\R_{+}) \otimes \calB(U)$. The result is not classic, because it is not clear at the beginning that we can pick a countable number of elements of $\calA$ that generate $\calB(U)$. The fact that $\calA$ contains each $\calB(U_n)$ fills the gap.

\begin{theorem}\label{theoremLemmaRemark}
Suppose $\nu_1$ and $\nu_2$ are random measures on $\calB(\R_+)\otimes \calB(U)$ such that, for each $A\in \calA$ and $0\leq s < t$
$$
\Prob(\nu_1((s,t]\times A) \leq \nu_2((s,t]\times A) ) = 1.
$$
Then $\nu_1 \leq \nu_2$. More precisely
$$
\Prob \left(\nu_1(C) \leq \nu_2(C), \forall C \in \calB(\R_+)\otimes \calB(U) \right) = 1.
$$
\end{theorem}
\begin{proof}
First, for fixed $s,t$ define $\mu_j(B) = \nu_j((s,t]\times B)$, for $j = 1,2$, $B\in \calB(U)$. As in the proof of Theorem \ref{theoExistMeasMuHValuedMVM}, we can assume that $U$ is a Borel set on the line. We consider the measures $\mu_{j,n}$ obtained by restricting $\mu_j$ to $(U_n,\calB(U_n))$ and see them as defined in $(\R,\calB(\R))$. To be precise, $\mu_{j,n}(B) :=\mu_j(B\cap U_n)$. Since $\calB(U_n) \subseteq \calA$, for $a < b$ we have $\Prob$-a.e.
$$ \mu_{1,n}((a,b]) \leq  \mu_{2,n}((a,b]).$$
It follows that
$$
\Prob \left( \mu_{1,n}((a,b]) \leq \mu_{2,n}((a,b]),\, \forall a,b\in\Q \right) =1
$$
and this is enough to conclude (working with fixed $\omega$) that
$$
\Prob \left( \mu_{1,n}(B) \leq \mu_{2,n}(B), \, \forall B\in \calB(\R) \right) =1.
$$
By the continuity of both measures we obtain $\Prob$-a.e.
$$
\mu_{1}(B) = \sup_n \mu_{1,n}(B) \leq \sup_n \mu_{2,n}(B) = \mu_2(B),\quad \forall B\in \calB(U).
$$
We have shown that, for fixed $s,t$
$$
\Prob(\nu_1((s,t]\times B)\leq \nu_2((s,t]\times B),\forall B\in\calB(U))=1.
$$
The remaining argument is classic using the right continuity of $\nu_{\cdot}$, working with fixed $\omega$. In fact, the rectangles $(s,t]\times B$ with $s,t$ rational and $B\in \calB(U)$ generate $\calB(\R_+)\otimes\calB(U)$.
\end{proof}

\begin{remark}\label{RemarkUniquenessIntensityMeasure}
The previous theorem clearly holds with equality. More precisely, if for given $s,t$ and $A\in \calA$, we have $\nu_1((s,t]\times A) = \nu_2((s,t]\times A)$ $\Prob$-a.e., then $\nu_1 = \nu_2$. This, in particular, implies  the uniqueness of $\nu$ in Corollary \ref{coroExistMeasMuHValuedMVM} and justifies  Definition \ref{defIntensityMeasure} below. It is important to recall that two random measures $\mu$ and $\nu$  are considered the same when $\Prob (\mu(C)=\nu(C) \text{ for all measurable } C ) =1 $.
\end{remark}

The following theorem extends the result to compare two random signed measures, or the absolute value of a random signed measure and a random measure. This result will be needed in the proofs of Lemmas \ref{lemmaAlphaMBilinearFormOnQSpanF} and \ref{lemmaContiAlphaMBilinearForm}.

\begin{theorem}\label{theoremComparisonSignedMeasuresOnRing}
    Suppose $\alpha$ and $\beta$ are random signed measures defined on $\calB(\R_+)\otimes \calB(U)$. 
    \begin{enumerate}
        \item If for each $A\in \calA$ and $0\leq s < t$
    $$
     \Prob(\alpha((s,t]\times A) \leq \beta((s,t]\times A) ) = 1
    $$
    then $\alpha \leq \beta$. 
    \item If for each $A\in \calA$ and $0\leq s < t$
    $$
     \Prob(\alpha((s,t]\times A) = \beta((s,t]\times A) ) = 1
     $$
     then $\alpha = \beta$. 
     \item If for each $A\in \calA$ and $0\leq s < t$
    $$
     \Prob( |\alpha((s,t]\times A)| \leq \beta((s,t]\times A) ) = 1
    $$
then $|\alpha| \leq \beta$. 
    \end{enumerate}
\end{theorem}
\begin{proof}$\,$\\
\begin{enumerate}
    \item[$(i)$] If $\alpha^{-}$ is finite, it follows by applying Theorem \ref{theoremLemmaRemark} to the inequality
    $$\alpha^{+} \leq \alpha^{-}+\beta.$$
    Otherwise, $\alpha^{+}$ must be finite and we apply Theorem \ref{theoremLemmaRemark} to the inequality
    $$(\alpha^{+} - \beta)^{+} \leq \alpha^{-}.$$
    \item[$(ii)$] Apply $(i)$ to the inequalities $\alpha\leq\beta$ and $\beta\leq\alpha$.
    \item[$(iii)$] Apply $(i)$ to the inequalities $- \beta \leq \alpha$ and $\alpha \leq \beta$.
\end{enumerate}
\end{proof}

\begin{definition}
    \label{defIntensityMeasure}
    The unique random predictable $\sigma$-finite measure $\nu$ given in Corollary \ref{coroExistMeasMuHValuedMVM} will be called  the intensity measure of $M$.
\end{definition}

\begin{example}\label{ex:whitenoisemeasure}
Let $(U,\calB(U),\lambda)$ be a $\sigma-$finite space. A (Gaussian) \emph{white noise measure} based on $\lambda$ is a random set function $W$ on the sets $A\in {\mathcal U}$ of finite $\lambda-$measure, such that
\begin{enumerate}
\item $W(A)$ is a ${\mathcal N}(0,\lambda(A))$ random variable.
\item If $A\cap B=\emptyset$, then $W(A)$ and $W(B)$ are independent and
$$
W(A\cup B)=W(A)+W(B).
$$
Equivalently $\Exp[ W(A)W(B)] = \Exp[ |W(A\cap B)|^2 ]= \lambda(A\cap B)$. 
\end{enumerate}

Consider a white noise measure $W$ on $(\R_+\times U,\calB(\R_+)\otimes \calB(U),Leb\otimes \lambda)$. This means, $W$ is defined at least on the cylinder sets $(s,t]\times A$, with $0\leq s \leq t$ and $A\in \calA$. Here, $\calA$ is the ring of finite $\lambda$-measure sets on $\calB(U)$. 

We define $M_t(A) = W([0,t] \times A )$ and observe that $M_t(A) \sim \mathcal{N}(0,t\lambda(A))$. Thus $(M_t(A): t\geq 0, A \in \calA)$ is clearly a martingale measure with respect to its natural filtration. Besides
$$
A\cap B =\emptyset \Rightarrow M_t(A) \text{ and } M_t(B) \text{ are independent },
$$
 hence orthogonal. We also have $\quadraVari{M(A)}_t =t\lambda(A)$. Given $U_n \uparrow U$ with $\lambda(U_n) < \infty$, it is clear that $\calB(U_n) \subseteq \calA$  and 
$$
\sup_{A\in \calB(U_n)} \Exp|M_t(A)|^2 = t\lambda(U_n)  < \infty.
$$
The martingale measure $M$ is also countably additive on each $U_n$. In fact, if $(A_j)$ is a sequence on $\calB(U_n)$ and $A_j \downarrow \emptyset$, then $\lambda(A_j) \fle 0$ and consequently $\Exp|M_t(A_j)|^2 \leq T\lambda(A_j) \fle 0$. Finally, given $A\in\calA$ we have
$$
\Exp|M_t(A) - M_t(A\cap U_n)|^2 = \Exp|M_t(A\setminus U_n)|^2 = t\lambda(A\setminus U_n) \fle 0.
$$
We have shown that each  $M_{t}(\cdot)$ satisfies Definition \ref{definitionSigmafiniteMeasure} and therefore, $M$ satisfies condition (iii) in Definition \ref{defiHValuedlOrthoMartinValuMeasure}. Conditions (i), (ii) and (iv) are immediate. We conclude that $M$ is a (real-valued) orthogonal martingale-valued measure.

    \begin{enumerate}[label=(\alph*)]
        \item Notice the additivity of $\quadraVari{M(\cdot)}_t$ as mentioned on Lemma \ref{lemmaAdditiveVariation}.
        \item Notice that the measure $\mu$ of Lemma \ref{lemmaHValuedMeasMuFiniteAdditive} in this case is $\mu(A) = T\lambda(A)$.
        \item The family of measures constructed in Theorem \ref{theoExistMeasMuHValuedMVM} is given by $\nu_t(A) = t\lambda(A)$. Also, in Corollary \ref{coroExistMeasMuHValuedMVM}, the intensity measure is
        $\nu((s,t]\times A) = (t-s)\lambda(A)$, that is $\nu = Leb\otimes \lambda$.
    \end{enumerate}

\end{example}

\section{Cylindrical orthogonal martingale-valued Measures}
\label{sectCylinMartingValued}

Let $X$ be a Banach space with separable dual $X^*$. The following definition extends that of orthogonal martingale-valued measure to the cylindrical context.

\begin{definition}\label{defiCylindricalOrthoMartinValuMeasure}
A \emph{cylindrical orthogonal martingale-valued measure} on $X^*$ is a collection $M=(M(t,A): t \geq 0, A \in \mathcal{A})$ of cylindrical random variables on $X$ such that:
\begin{enumerate} 
\item For each $A \in \mathcal{A}$, $M(0,A)(x^*)= 0$ $\Prob$-a.e. for all $x^* \in X^*$. \label{timezerocomvmdef}
\item For each $A \in \mathcal{A}$, $M(A) = (M(t,A): t \geq 0)$, is a cylindrical mean-zero square integrable martingale, and for each $t > 0$ and $A \in \calA$, the map 
$$ 
M(t,A): X^* \rightarrow L^{0} \ProbSpace
$$ 
is continuous. \label{cylindricalMartingale}
\item If $t>0$ and $x^* \in X^*$, $M(t,\cdot)(x^*): \mathcal{A} \rightarrow L^{2} \ProbSpace$ is a $\sigma$-finite $L^{2}$-valued measure. \label{fixedtimeismeasure}
\item If $t>0$ and $x^* \in X^*$, $\inner{M(A)(x^*)}{M(B)(x^*)}_{t}=0$ whenever $A,B\in \mathcal{A}$ are disjoint. \label{orthogonality}
\end{enumerate}
\end{definition}

\begin{remark}\label{remarkL2ContinuityCMVM}
Let $t>0$ and $A \in \mathcal{A}$. Our assumption \ref{cylindricalMartingale} imply that the linear mapping $M(t,A): X^* \rightarrow L^2 \ProbSpace$ is continuous. In fact, assume that $x^{*}_{n} \rightarrow x^*$ in $X^*$ and $M(t,A)(x_n^*) \rightarrow Y$ in $L^2 \ProbSpace$. By \ref{cylindricalMartingale} we have $M(t,A)(x^{*}_{n}) \rightarrow M(t,A)(x^{*})$ in probability. On the other hand, since $L^2$ convergence implies convergence in probability we have $M(t,A)(x_n^*) \rightarrow Y$ in probability. By uniqueness of limits $Y = M(t,A)(x^{*})$ $\Prob$-a.e. and the closed graph theorem finishes the work.  
\end{remark}

\begin{remark}\label{remarkNotationRealQuadraVariM}
For any $x^* \in X^*$ and $A \in \mathcal{A}$, by \ref{cylindricalMartingale} the process 
$$
M(A)(x^*) = (M(t,A)(x^*) : t \geq 0)
$$
is a real-valued square-integrable martingale, so the brackets in \ref{orthogonality} correspond to the (real) covariation of two real-valued processes. We will also use the notation $\langle M(A)(x^*) \rangle_t$ for the (real) quadratic variation, when it exists.
\end{remark}

\begin{example}\label{exampleCylinMartingaleDelta}
Consider a finite set $U=\{a_{1}, \ldots, a_{n}\}$. In this case we can take $\mathcal{A}=2^U$, which corresponds to the discrete topology. For each $k=1, \cdots, n$, let $Z^k =(Z^{k}_{t}: t \geq 0)$ be a cylindrical c\`{a}dl\`{a}g zero-mean square integrable martingale such that for each $t \geq 0$ the mapping $Z^{k}_{t}: X^* \rightarrow L^{0} \ProbSpace$ is continuous. Assume moreover that for each $x^* \in X^*$, the real-valued martingales $(Z^{k} (x^{*}))_{k=1}^{n}$ are orthogonal. 

Define a family $M=(M(t,A): t \geq 0, A \in \mathcal{A})$ by the prescription:
$$M(t,A) (x^*) = \sum_{k=1}^{n} Z^{k}_{t}(x^*) \delta_{a_{k}}(A), \quad \forall \, x \in X^*, \, t \geq 0, \, A \in \mathcal{A}. $$
Since every $A \in \mathcal{A}$ is either the empty set or has a finite number of elements, it is not difficult to check that (i)-(iv) in 
Definition \ref{defiCylindricalOrthoMartinValuMeasure} are satisfied. Hence $M$ defined above is a  cylindrical orthogonal martingale-valued measure.
\end{example}

\begin{example}\label{exampleCylinMartingaleIntegral}
 Let $Z: X^* \rightarrow \mathcal{M}^{2}_{\infty}$ be a continuous linear operator, in particular a cylindrical  c\`{a}dl\`{a}g zero-mean square integrable martingale on $X$. Let $g: \R_{+} \times \Omega \rightarrow U$ be a predictable process and let $\mathcal{A}=\mathcal{B}(U).$    

 Define a family $M=(M(t,A): t \geq 0, A \in \mathcal{A})$ by the prescription:
 \begin{equation*}
M(t,A)(x^*) =\int_{0}^{t} \, \mathbbm{1}_{A}(g(s)) \, dZ(x^*)_s,  \quad \forall x^* \in X^*, \, t \geq 0, \, A \in \mathcal{A}.      
 \end{equation*}
 It is clear that for all $A \in \mathcal{A}$ and $x^* \in X^*$ we have $M(0,A)(x^*)=0$ $\Prob$-a.e., so Definition \ref{defiCylindricalOrthoMartinValuMeasure}(i) is satisfied.  
 
 The linearity of $Z$ shows that each $M(t,A)$ defines a cylindrical random variable on $X$. Since $\mathbbm{1}_{A}(g(s,\omega))$ is a real-valued bounded predictable process, then $M(t,A) (x^*)$ is a real-valued zero-mean square integrable c\`{a}d\`{a}g martingale. Furthermore, by the It\^{o} isometry we have
 \begin{eqnarray*}
  \Exp \left[ \abs{M(t,A) (x^*)}^{2} \right]
  & = & \Exp \left[ \abs{ \int_{0}^{t} \, \mathbbm{1}_{A}(g(s)) \, dZ(x^*)_s }^{2} \right] \\
 & = &  \Exp \int_{0}^{t} \, \mathbbm{1}_{A}(g(s)) \, d\quadraVari{Z(x^*)}_s \\
 & \leq & \Exp \left[ \quadraVari{Z(x^*)}_t \right] = \Exp \left[ \abs{Z(x^*)_{t}}^2 \right].    
 \end{eqnarray*}
Hence showing that for each $t >0$ and $A \in \mathcal{A}$ the mapping $M(t,A): X^* \rightarrow L^{2}\ProbSpace$ is continuous. 
This shows Definition \ref{defiCylindricalOrthoMartinValuMeasure}(ii) is satisfied. 

Now we check Definition \ref{defiCylindricalOrthoMartinValuMeasure}(iii). In fact, let $t>0$ and $x^* \in X^*$.  We have proved above that $\Exp \left[ \abs{M(t,A)( x^*)}^{2} \right] < \infty$ for each $A \in \mathcal{A}$. Moreover if $A \cap B = \emptyset$, $A , B \in \mathcal{A}$, then 
\begin{eqnarray*}
M(t,A \cup B)(x^*)
& = & \int_{0}^{t} \, \mathbbm{1}_{A \cup B}(g(s)) \, dZ(x^*)_s \\ 
& = & \int_{0}^{t} \, \mathbbm{1}_{A}(g(s)) \, dZ(x^*)_s + \int_{0}^{t} \, \mathbbm{1}_{B}(g(s)) \, dZ(x^*)_s \\
& = & M(t,A)(x^*) + M(t,B)(x^*).
\end{eqnarray*}
Furthermore, each $M(t,\cdot)( x^*)$ is easily seem to be $\sigma$-finite $L^2$-valued function by taking $U_{n}=U$ for all $n \in \N$. 

Finally, given $t>0$ and $x^* \in X^*$, for $A , B \in \mathcal{A}$ satisfying $A \cap B = \emptyset$ we have, by the theory of stochastic integration, that  
\begin{eqnarray*}
\Exp \left[ \inner{M(A)(x^*)}{M(B)(x^*)}_{t} \right]
& = & \Exp \left[ M(t,A )(x^*) \cdot M(t,B)(x^*) \right] \\
& = & \Exp \int_{0}^{t} \mathbbm{1}_{A}(g(s)) \cdot \mathbbm{1}_{B}(g(s)) \, d\quadraVari{Z(x^*)}_s =0.
\end{eqnarray*}
Hence Definition \ref{defiCylindricalOrthoMartinValuMeasure}(iv) is satisfied.  
\end{example}

We finish this section with a result that allows us to look at a Hilbert space-valued martingale-valued measure as a cylindrical orthogonal martingale-valued measure.  
 
\begin{proposition}\label{propHValuedMVMDefinesCylinMVM}
Let $H$ be a separable Hilbert space and let $\tilde{M}=(\tilde{M}(t,A): t \geq 0, A \in \mathcal{A})$ be an $H$-valued orthogonal martingale-valued measure. Assume further that for every $A, B \in \mathcal{A}$ disjoint and $h \in H$, the real-valued martingales $(\tilde{M}(A),h)_{H}$ and $(\tilde{M}(B),h)_{H}$ are orthogonal. Then $\tilde{M}$ induces a  cylindrical orthogonal martingale-valued measure $M=(M(t,A): t \geq 0, A \in \mathcal{A})$ on $H$ by means of the prescription 
\begin{equation}\label{eqInducedCMVMDefinedByHValuedMVM}
M(t,A)(\omega)(h)\defeq (\tilde{M}(t,A)(\omega),h)_{H}, \quad \forall \omega \in \Omega, t \geq 0, A \in \mathcal{A}, h \in H.    
\end{equation} 
\end{proposition}

\begin{proof}
It is easy to verify each of the statements of Definition \ref{defiCylindricalOrthoMartinValuMeasure}. Note, in particular, that for each $t>0$ and $A \in \mathcal{A}$, the map $M(t,A): H \to L^0(\Omega, \mathcal{F}, \mathbbm{P})$ is continuous  by using a Chebyshev-type argument. Indeed for any $\epsilon>0$, one can see that
$$\mathbbm{P} ( |(\tilde{M}(t,A),h_n-h)| \geq \epsilon ) \leq \frac{1}{\epsilon^2} \cdot \| h_n-h\|^2_H \, \mathbbm{E}[\|\tilde{M}(t,A) \|^2] \to 0,$$
whenever $h_n \to h$ in $H$.
\end{proof}

\begin{remark} Let $H$ be a separable Hilbert space and let $\tilde{M}=(\tilde{M}(t,A): t \geq 0, A \in \mathcal{A})$ be an $H$-valued martingale-valued measure. Let $(h_{n})_{n \geq 1}$ be a orthonormal basis in $H$ and for every $n \geq 1$ let ${\tilde{M}}^{n}(t,A)=(\tilde{M}(t,A),h_{n})_{H}$. Using the result of Lemma 2.2 in \cite{RozovskyLototsky:2018} one can show that for every $A, B \in \mathcal{A}$, $t \geq 0$, we have
$$ \quadraVari{\tilde{M}(A),\tilde{M}(B)}_{t}=\sum_{n \geq 1} \quadraVari{\tilde{M}^{n}(A),\tilde{M}^{n}(B)}_{t}.$$
Therefore, if we assume that for every $A, B \in \mathcal{A}$ disjoint and $h \in H$, the real-valued martingales $(\tilde{M}(A),h)_{H}$ and $(\tilde{M}(B),h)_{H}$ are orthogonal, we conclude that $\tilde{M}(A)$ and $\tilde{M}(B)$ are orthogonal (as $H$-valued martingales). Thus $\tilde{M}$ is orthogonal. 
\end{remark}

\section{Construction of the quadratic variation} \label{sectQuadraticVariation}

\subsection{Definition and properties of the quadratic variation}\label{subsectDefiQuadraVariation}

In this section we define the (predictable) quadratic variation of a cylindrical martingale-valued measure. Our definition is based on a extension of the definition of the quadratic variation as a supremum of measures introduced by Veraar and Yaroslavtsev in \cite{VeraarYaroslavtsev:2016} in the case of cylindrical continuous local martingales. In our case, we found convenient to formulate our definition in terms of the family of intensity measures defined by the family of real-valued martingale-valued measures $(M(t,A)(x^*): t \geq 0, A \in \mathcal{A})$. The existence of such a family of measures is guaranteed by the following result. 

\begin{lemma}\label{lemmaExistMeasMuRealQuadraVaria}
For every $x^* \in X^*$, there exists a random predictable $\sigma$-finite measure $\nu_{x^*}$ on $\mathcal{B}(\R_{+}) \otimes \mathcal{B}(U)$ such that for every $t \geq 0$ and $A \in \calA$,
\begin{equation}
\label{eqAlmosSureSetMeasureNuhA}
\Prob \left( \{ \omega \in \Omega : \nu_{x^*}(\omega)\left( [0,t] \times A \right)  = \quadraVari{M(A)(x^*)}_{t}(\omega)  \} \right)=1. 
\end{equation}
\end{lemma}

\begin{proof}
    It is enough to apply Corollary \ref{coroExistMeasMuHValuedMVM} to the $\R$-valued orthogonal martingale-valued measure $M(A)(x^*)$, for each $x^*\in X^*$. 
\end{proof}

Basic properties of $\quadraVari{M(\cdot)(x^*)}$ are inherited by $\nu_{x^*}$. For instance, given a real number $c$, we have $\nu_{cx^*} = c^2\nu_{x^*}$ (as random measures). In fact, given $t$ and $A$ we have, $\Prob$-a.e.
$$
c^2 \nu_{x^*}([0,t]\times A)= c^2 \quadraVari{M(A)(x^*)}_t = \quadraVari{cM(A)(x^*)}_t = \quadraVari{M(A)(cx^*)}_t.
$$
Uniqueness of the intensity measure gives the identity (Remark \ref{RemarkUniquenessIntensityMeasure}).

\begin{example}\label{exampleContinuityConditionCylinMartingaleDelta}
Let $M$ denote the cylindrical martingale-valued measure defined in Example \ref{exampleCylinMartingaleDelta}. For every $x^{*} \in X^{*}$ our assumption that the real-valued martingales $(Z^{k} (x^{*}))_{k=1}^{n}$ are orthogonal implies that $\quadraVari{M(A)(x^{*})}_{t} = \sum_{k=1}^{n} \quadraVari{Z^{k}(x^{*})}_{t} \delta_{a_{k}}(A)$ for every $t>0 $ and $A \in \mathcal{A}=\mathcal{B}(U)$. Hence $$\nu_{x^{*}}(ds,du)=\sum_{k=1}^{n} \lambda_{\quadraVari{Z^{k}(x^{*})}}(ds) \delta_{a_{k}}(du) ,$$ 
where $\lambda_{\quadraVari{Z^{k}(x^{*})}}$ denotes the \emph{Lebesgue-Stieltjes} measure associated to $\quadraVari{Z^{k}(x^{*})}$. 
\end{example}

We are ready to introduce our definition of quadratic variation for $X$, which is an extension of Definition 3.4 in \cite{VeraarYaroslavtsev:2016}. First, we will need the following terminology: let $(S,\Sigma)$ be a measurable space and let $\mathcal{M}_{+}(S,\Sigma)$ be the set of all positive measures on $(S,\Sigma)$. For $\eta, \zeta: \Omega \rightarrow \mathcal{M}_{+}(S,\Sigma)$ we say that $\eta \leq \zeta $ if $\eta(\omega) \leq \zeta (\omega)$ for $\Prob$-a.e. $\omega \in \Omega$. The relation ``$\leq $'' defines a partial order in $\mathcal{M}_{+}(S,\Sigma)$. 

\begin{definition}\label{defiQuadraticVariation}
We say that $M$   \emph{has a quadratic variation} if there exists a random measure $\eta: \Omega \rightarrow \mathcal{M}_{+}(\R_{+} \times U, \mathcal{B}(\R_{+}) \otimes \mathcal{B}(U) )$ such that 
\begin{enumerate}    
    \item Given $t \geq 0$ and $A \in \mathcal{A}$, for $\Prob$-a.e. $\omega \in \Omega$ we have $\eta(\omega)([0,t] \times A) < \infty$.
    \item \label{propertyUpperBoundNuX} $\eta$ is a minimal element (in the partial order ``$\leq $'') for the collection of all the random measures $\zeta: \Omega \rightarrow \mathcal{M}_{+}(\R_{+} \times U, \mathcal{B}(\R_{+}) \otimes \mathcal{B}(U) )$ with the property:   $\forall\, x^* \in X^*$ with $\| x^* \| = 1$,  $\nu_{x^*} \leq \zeta$.
\end{enumerate}
We say that $\eta$ is a \emph{quadratic variation} for $M$. 
\end{definition}

The reader might notice that in the definition above,  the possibility that M has more than one quadratic variation is left open. Later in this section we introduce a sufficient condition for the existence of a unique quadratic variation. In the following result we provide a necessary condition for the existence of a quadratic variation for $M$.

\begin{proposition}\label{propQuadraticVariaMaximalOverCountableSupremum}
Assume that $M$ has a quadratic variation $\eta$. Let $(x^{*}_{n})_{n \geq 1}$ be a subset  of the unit sphere in $X^{*}$ and let $\mu \defeq \sup_{n \geq 1} \nu_{x^{*}_{n}}$. Then  $\mu \leq \eta$.    
\end{proposition}
\begin{proof}
Since $\eta$ is a quadratic variation for $M$, for every $n \geq 1$ there exists $\Omega_{n} \subseteq \Omega$, with $\Prob(\Omega_{n})=1$ and  $\nu_{x^{*}_{n}}(\omega) \leq \eta(\omega)$ for every $\omega \in \Omega_{n}$. Let $\Omega_{0}= \bigcap_{n \geq 1} \Omega_{n}$. Then $\Prob(\Omega_{0})=1$ and by the definition of supremum of measures we have $\mu(\omega) \leq \eta(\omega)$ for every $\omega \in \Omega_{0}$.     
\end{proof}

We now tackle the issue of existence and uniqueness of a quadratic variation for $M$. Our key property is described as follows:

\begin{definition}
\label{defiSeqBoundProperty}
 We say that the family of intensity measures $(\nu_{x^{*}}: x^{*} \in X^{*})$ satisfies the \emph{sequential boundedness property} if whenever  $(x^{*}_{n})$ is dense in the unit sphere, $\norm{x^*}=1$ and $t>0$, there exists $\Omega_{x^*} \subseteq \Omega$ with $\Prob(\Omega_{x^*})=1$, such that for $0 \leq s <t$, $A \in \mathcal{A}$ and $\omega \in \Omega_{x^*}$,
 
\begin{equation}
  \label{eqsequentialBoundednessProper}
  \nu_{x^{*}}(\omega)((s,t] \times A) \leq  \sup_{n \geq 1} \nu_{x^{*}_{n}}(\omega) ((s,t] \times A). 
\end{equation}
 \end{definition}

Note that the supremum in (\ref{eqsequentialBoundednessProper}) is a classical supremum of real numbers, which can be infinite in some cases.

\begin{remark}
\label{remaSufCondSeqBound}
 It is enough to verify \eqref{eqsequentialBoundednessProper} for sequences that converge to $x^*$ in the unit sphere. In fact, if that is true and $(x_n^*)$ is dense in the unit sphere, there is a subsequence $(x_{n_k})$ converging to $x^*$. Then $\Prob$-a.e. we have, for $0 \leq s < t$ and $A\in \calA$,
 $$
 \nu_{x^*}((s,t]\times A) \leq \sup_{k} \nu_{x_{n_k}^*}((s,t]\times A) \leq \sup_{n} \nu_{x_{n}^*}((s,t]\times A).
 $$
 \end{remark}

The following result provides a sufficient condition for the sequential boundedness property to hold.  

\begin{proposition}\label{propUniformUCPImpliesWeakBoundedness}
Assume the family of intensity measures $(\nu_{x^*}: x^* \in X^*)$ satisfies the following condition: given $x^{*}_{n} \rightarrow x^*$ in the unit sphere of $X^{*} $, for every $t>0$ we have
\begin{equation}\label{eqLocalUniformConverInProbabilityMeasuresNu}
\sup_{A \in \mathcal{A}} \sup_{0 \leq s \leq t} \abs{ \nu_{x^{*}_{n}}((s,t] \times A) - \nu_{x^{*}}((s,t] \times A)} \overset{\Prob}{\rightarrow} 0, \quad \mbox{as } n \rightarrow \infty.       
\end{equation}    
Then $(\nu_{x^{*}}: x^{*} \in X^{*})$ satisfies the sequential boundedness property.
\end{proposition}

\begin{proof}
Let $X_n = \sup_{A \in \mathcal{A}} \sup_{0 \leq s \leq t} \abs{ \nu_{x^{*}_{n}}((s,t] \times A) - \nu_{x^{*}}((s,t] \times A)}$. Since $X_n \overset{\Prob}{\rightarrow} 0$, there exists a sequence of positive integers $n_k \uparrow\infty$ such that $X_{n_k}\fle 0,\, \Prob$-a.e. In particular, there exists $\Omega_{x^*} \subseteq \Omega$ with $\Prob(\Omega_{x^*})=1$ such that, for $0\leq s<t$, $A\in \calA$ and $\omega\in \Omega_{x^*}$
$$
\nu_{x^{*}}(\omega)((s,t] \times A) = 
\lim_{k \rightarrow \infty} \nu_{x^{*}_{n_k}}(\omega)((s,t] \times A) \leq \sup_{n \geq 1} \nu_{x^{*}_{n}}(\omega)((s,t] \times A).
$$
The last remark finishes the work.
\end{proof}

Now we provide a sufficient condition for the existence and uniqueness of a quadratic variation for $M$.

\begin{theorem}\label{theoSuffiCondiExistQuadraVariat}
Let $(x^{*}_{n})_{n \geq 1}$ be a dense subset of the unit sphere in $X^{*}$ and let $\mu=\sup_{n \geq 1} \nu_{x^{*}_{n}}$. Assume 
\begin{enumerate}
    \item Given $t \geq 0$ and $A \in \mathcal{A}$, $\Prob$-a.e. we have $\mu([0,t] \times A)< \infty$. \label{FiniteonRectangles}
    \item $(\nu_{x^{*}}: x^{*} \in X^{*})$ satisfies the sequential boundedness property.
\end{enumerate}
Then $\mu$ is a quadratic variation for $M$. In particular, the quadratic variation is unique (any quadratic variation equals $\mu$ $\Prob$-a.e.)
\end{theorem}
\begin{proof} 
We must check that $\mu$ satisfies Definition \ref{defiQuadraticVariation}. In fact, by the sequential boundedness property, for all $x^* \in X^*$ with $\norm{x^*}=1$ we have that  $\Prob$-a.e., for $0 \leq s <t$ and $A \in \mathcal{A}$
\begin{equation*}
 \nu_{x^{*}}((s,t] \times A) \leq   \sup_{n\geq 1} \nu_{x_{n}^{*}}((s,t] \times A)  \leq   \mu((s,t] \times A).    
\end{equation*}
By Theorem \ref{theoremLemmaRemark} this inequality extends to $\mathcal{B}(\R_{+}) \otimes \mathcal{B}(U)$, then we have that for every $x^*$ in the unit ball of $X^*$,  $\nu_{x^*} \leq \mu$  $\Prob$-a.e., and by \ref{FiniteonRectangles}, $\mu$ is finite on the required rectangles.  

Let $\zeta$ be a random measure on $\mathcal{B}(\R_{+}) \otimes \mathcal{B}(U)$ such that $\nu_{x^*} \leq \zeta $ whenever $\norm{x^*}=1$. It follows that $\Prob$-a.e. $\nu_{x^*_n} \leq \zeta$ for every $n \geq 1$; by definition of supremum of measures, $\mu \leq \zeta$.  Therefore, $\mu$ is a quadratic variation  for $M$. Finally, given any quadratic variation $\eta$ we have $\mu \leq \eta$. Since $\eta$ is minimal, this implies $\eta=\mu$ $\Prob$-a.e. 
\end{proof}

\begin{definition}
If $M$ has a unique quadratic variation, we will denote it by $\operQuadraVari{M}$ and refer to it as \emph{the quadratic variation of $M$}.     
\end{definition}

As the next result shows, in the presence of the sequential boundedness property for the family of intensity measures, if a quadratic variation exists, it is unique. Other properties of the quadratic variation are given. 
       
\begin{theorem}\label{theoUniquenessQuadraticVariation}
Assume $M$ has a quadratic variation and $(\nu_{x^{*}}: x^{*} \in X^{*})$ satisfies the sequential boundedness property. Then $M$ has a unique quadratic variation $\operQuadraVari{M}$. Moreover, for any dense subset $(x^{*}_{n})_{n \geq 1}$ of the unit sphere in $X^{*}$ we have $\operQuadraVari{M}=\sup_{n \geq 1} \nu_{x^{*}_{n}}$ $\Prob$-a.e.  
 In particular $\operQuadraVari{M}$ is a predictable random measure and $\Prob$-a.e.,  
 \begin{equation}\label{eqRepreQuadraVariaAsSupremum}
 \operQuadraVari{M}([0,t] \times A)=\sup_{\Pi \in \mathscr{R}([0,t]\times A)} \sum_{C \in \Pi} \sup_{n\geq 1} \nu_{x_{n}^{*}}(\omega) \left( C \right),   \quad \forall t \geq 0, A \in \mathcal{B}(U),
 \end{equation}
 where $\mathscr{R}([0,t]\times A)$ is the family of partitions of $[0,t]\times A$ of the form 
 \begin{equation}\label{Rationalpartition}
   \Pi = \{ (t_{i-1}, t_i] \times A_j:  1 \leq i \leq k, 1 \leq j \leq m, \ k,m \in \N \},   
 \end{equation}
 where $0=t_{0} < t_{1} < \cdots < t_{k}=t$ are rational (with the possible exception of $t$), the sets $A_1,\ldots,A_{m}$ form a partition of $A \in \mathcal{B}(U)$, and $(x^{*}_{n})_{n \geq 1}$ is a dense subset of the unit sphere in $X^{*}$.
\end{theorem}
\begin{proof}
Let $(x^{*}_{n})_{n \geq 1}$ be a dense subset of the unit sphere and let $\mu =\sup_{n \geq 1} \nu_{x^{*}_{n}}$. 

Let $\eta$ be a quadratic variation for $M$. 
Then, by  Proposition \ref{propQuadraticVariaMaximalOverCountableSupremum}, $\mu = \sup_{n \geq 1} \nu_{x^{*}_{n}}$ satisfies for all $t \geq 0$, $A \in \mathcal{A}$, $\Prob$-a.e. 
$\mu([0,t] \times A) \leq \eta([0,t] \times A)< \infty$. Hence by Theorem \ref{theoSuffiCondiExistQuadraVariat} $\mu$ is a quadratic variation and we have $\eta=\mu$ $\Prob$-a.e. This shows uniqueness of the quadratic variation and by definition $\operQuadraVari{M} = \mu$.

Moreover, by Lemma \ref{lemmaSupExplicit}, $\operQuadraVari{M}$ takes the form \eqref{eqRepreQuadraVariaAsSupremum}  and it is a predictable random measure by Corollary \ref{CoroSupPredRandMeas}. 
\end{proof}

\begin{remark}\label{partitionssameindex}
Note that any partition descibed as in \eqref{Rationalpartition} can be written in the form $\{ (s_j, t_j] \times A_j : j = 1, \ldots, m \}$, where the numbers $s_j$ and $t_j$ are rational, $A_1, \ldots, A_m$ form a partition of $A$, and also the following is satisfied: 
\begin{equation} \label{separationPartitionIntervalsSimpleForm}
\mbox{for } k \neq j, \    \mbox{if} \  (s_{k},t_{k}] \cap   (s_{j},t_{j}] \neq \emptyset, \  \mbox{then}  \  (s_{k},t_{k}]= (s_{j},t_{j}] \mbox{ and } A_{k} \cap A_{j}= \emptyset.
\end{equation}  
\end{remark}

\begin{remark}
Under the assumptions in Theorem \ref{theoUniquenessQuadraticVariation}, notice that by Lemma \ref{lemmaExistMeasMuRealQuadraVaria} for any given $t \geq 0$ and  $A \in \mathcal{A}$ we have $\Prob$-a.e. 
 \begin{equation}\label{eqRepreQuadraVariaAsSupremQuadraVariations}
 \operQuadraVari{M}([0,t] \times A)=\sup_{\Pi}\sum_{j=1}^{m_{\Pi}} \sup_{n \geq 1} \left( \quadraVari{M(A_{j})(x^{*}_{n})}_{t_{j}}- \quadraVari{M(A_{j})(x^{*}_{n})}_{s_{j}} \right), 
 \end{equation}
 where the supremum is taken over all partitions $\Pi$ in the form described in Remark \ref{partitionssameindex}.    
\end{remark}

\begin{example}\label{examCylinMartinContPaths}
Let $Z=(Z_{t}: t \geq 0)$ be a cylindrical  continuous zero-mean square integrable martingale. Consider a one-point set $U=\{a\}$ and let $M$ be the corresponding cylindrical orthogonal martingale-valued measure of Example \ref{exampleCylinMartingaleDelta}. We shall verify that the family of intensity measures $(\nu_{x^{*}}: x^* \in X^*)$ of $M$ satisfies the sequential boundedness property. 

Let $x^{*} \in X^{*} $. By Example \ref{exampleContinuityConditionCylinMartingaleDelta},for all $0 \leq s < t$, $A \in \mathcal{A}= 2^{U}$ we have 
\begin{equation}\label{eqMeasureNuXContinuousMartingale}
 \nu_{x^{*}}((s,t],A)= \lambda_{\quadraVari{Z(x^{*})}}((s,t])\delta_{a}(A)=(\quadraVari{Z(x^{*})}_{t}-\quadraVari{Z(x^{*})}_{s}) \delta_{a}(A). 
\end{equation}
Assume $x^{*}_{n} \rightarrow x^*$ and let $t>0$. Since $Z$ defines a continuous  linear operator from $X^*$ into the space $\mathcal{M}^{2,c}_{t}$ of real-valued continuous square integrable martingales on $[0,t]$, then $\quadraVari{Z(x_{n}^{*})} \rightarrow \quadraVari{Z(x^{*})}$ in probability uniformly on $[0,t]$ (see Proposition 18.6 in \cite{Kallenberg:2021}).  Hence 
\begin{flalign*}
&\sup_{A \in \mathcal{A}} \sup_{0 \leq s \leq t} \abs{ \nu_{x^{*}_{n}}((s,t] \times A) - \nu_{x^{*}}((s,t] \times A)} \\
& = \sup_{0 \leq s \leq t} \abs{ (\quadraVari{Z(x_{n}^{*})}_{t}-\quadraVari{Z(x_{n}^{*})}_{s}) - (\quadraVari{Z(x^{*})}_{t}-\quadraVari{Z(x^{*})}_{s})}  \overset{\Prob}{\rightarrow} 0,    
\end{flalign*}
as $n \rightarrow \infty$. By Proposition \ref{propUniformUCPImpliesWeakBoundedness} the family $(\nu_{x^{*}}: x^* \in X^*)$  satisfies the sequential boundedness property. 

Notice that by \eqref{eqMeasureNuXContinuousMartingale}, Theorem  \ref{theoSuffiCondiExistQuadraVariat} and Theorem \ref{theoUniquenessQuadraticVariation} we have  that $M$ has (a necessarily unique) quadratic variation  if and only if for some (equivalently for any) dense subset $(x^{*}_{n})_{n \geq 1}$ of the unit sphere in $X^{*}$ we have $\Prob$-a.e. 
\begin{equation}\label{eqFiniteSupRemaContinuousMarting}
 \left(\sup_{n \geq 1} \lambda_{\quadraVari{Z(x_{n}^{*})}}\otimes \delta_{a}\right)([0,t] \times A)< \infty, \quad  \forall t \geq 0, \, A \in 2^{U}.   
\end{equation}  
In such a case we have $\operQuadraVari{M}=\sup_{n \geq 1} \lambda_{\quadraVari{Z(x_{n}^{*})}}\otimes \delta_{a}$ $\Prob$-a.e. 

Observe that since $2^{U}=\{\emptyset, \{a\}\}$, then given $(x^{*}_{n})_{n \geq 1}$ as above, \eqref{eqFiniteSupRemaContinuousMarting} holds true if and only if $\Prob$-a.e. $\left(\sup_{n \geq 1} \lambda_{\quadraVari{Z(x_{n}^{*})}} \right)([0,t])< \infty$, $\forall t \geq 0$. By Remark 2.10 in \cite{VeraarYaroslavtsev:2016} this is equivalent to the existence of a non-decreasing right-continuous process $F: \R_{+} \times \Omega \rightarrow \R_{+}$ such that $\Prob$-a.e. $\lambda_{F}= \sup_{n \geq 1} \lambda_{\quadraVari{Z(x_{n}^{*})}}$. 

The above observation can be thought as a generalization of Proposition 3.7 in \cite{VeraarYaroslavtsev:2016} in the case of a cylindrical continuous square integrable martingale. 
\end{example}

\begin{example}\label{examCMVMCylindriLevyProcesses}
Let $Z=(Z_{t}: t \geq 0)$ be a cylindrical zero-mean square integrable L\'{e}vy process in $X$, i.e. for every $d \in \N$, $x_{1}^{*}, \cdots, x_{d}^{*} \in X^*$ the $\R^{d}$-valued stochastic process $(Z_{t}(x_{1}^{*}), \cdots, Z_{t}(x_{d}^{*}): t \geq 0)$ is a  L\'{e}vy process in $\R^{d}$, and $\Exp [Z_{t}(x^{*}) ] = 0 $ and $\Exp [ \abs{Z_{t}(x^{*})}^{2} ]< \infty$ for every $t \geq 0$, $x^{*} \in X^{*}$. We will always assume for such a $Z$ that the mapping $Z_{t}: X^* \rightarrow L^{0} \ProbSpace$ is continuous. 

With the above hypothesis, it follows by Theorem 4.7 in \cite{ApplebaumRiedle:2010} that there exists a positive symmetric operator $Q:X^{*} \rightarrow X^{**}$, called the the \emph{covariance operator} of $Z$, defined by $(Qx^{*})y^*=\Exp \left[ Z_{1}(x^{*}) Z_{1}(y^{*}) \right] $ $\forall x^{*}, y^{*} \in X^{*}$. The operator $Q$ is linear and continuous by Proposition III.1.1 in \cite{VakhaniaTarieladzeChobanyan}. For every $x^{*} \in X^{*}$, observe that $\quadraVari{Z(x^{*})}_{t} = t \Exp \left[\abs{Z_{1}(x^{*})}^{2} \right]= t ( Q(x^{*} ) x^{*})$. 

Now let $n \in \N$, $U=\{a_{1}, \cdots, a_{n}\}$, and  $Z^{1}, \cdots, Z^{n}$ be cylindrical zero-mean square integrable L\'{e}vy processes in $X$ with corresponding covariance operators $Q^{1}, \cdots, Q^{n}$. Assume moreover that for each $x^* \in X^*$, the real-valued martingales $(Z^{k} (x^{*}))_{k=1}^{n}$ are orthogonal. Let $M$ be as  defined in Example \ref{exampleCylinMartingaleDelta}. We will show that $M$ has a quadratic variation. 

By Example \ref{exampleContinuityConditionCylinMartingaleDelta} we  have for every $x^{*} \in X^{*} $  
$$ \nu_{x^{*}}(ds,du)=\sum_{k=1}^{n} \lambda_{\quadraVari{Z^{k}(x^{*})}}(ds) \delta_{a_{k}}(du) = \sum_{k=1}^{n}  ( Q^{k} (x^{*} ) x^{*}) ds\delta_{a_{k}}(du) . $$
From the above it is immediate that for any $x^*, y^* \in X^*$, 
$$ \sup_{A \in \mathcal{A}} \sup_{0 \leq s \leq t} \abs{ \nu_{x^{*}}((s,t] \times A) - \nu_{y^{*}}((s,t] \times A)}  \leq t \left(\sum_{k=1}^{n} \norm{Q^{k}} \right) (\norm{x^{*}}+\norm{y^*}) \norm{x^{*}-y^{*}}$$
and hence the family of measures $\nu_{x^{*}}$ satisfies \eqref{eqLocalUniformConverInProbabilityMeasuresNu}, therefore it has the sequential boundedness property by Proposition \ref{propUniformUCPImpliesWeakBoundedness}. 

Let $(x_{m}^{*})_{m \geq 1} $ be a dense subset of the unit sphere in $X^{*}$. According to Theorem \ref{theoSuffiCondiExistQuadraVariat} we must show that $\sup_{m \geq 1} \nu_{x_{m}^{*}}$ defines a random positive measure which is finite on the rectangles $[0,t] \times A$. 

In fact, let $t \geq 0$ and $A \in \mathcal{A}=2^{U}$.
Take a partition $\{ (s_j, t _j] \times A_j \}$ of $[0,t] \times A$, that satisfies \eqref{separationPartitionIntervalsSimpleForm}, furthermore, we can take each $A_j$ as a singleton.  Then we have
\begin{eqnarray*}
 \sum_{j=1}^{N} \sup_{m \geq 1} \nu_{x_{m}^{*}} \left( (s_j, t_{j}] \times A_{j} \right)
 & = &  \sum_{j=1}^{N} \sup_{m \geq 1}  \sum_{k=1}^{n} (t_{j}-s_j) ( Q^{k} (x_{m}^{*} ) x_{m}^{*}) \delta_{a_{k}}(A_{j}) \\
 & = & \sum_{j=1}^{N} (t_{j}-s_j) \sum_{k=1}^{n} \sup_{m \geq 1} ( Q^{k} (x_{m}^{*} ) x_{m}^{*}) \delta_{a_{k}}(A_{j}) \\
 & = & \sum_{j=1}^{N} (t_{j}-s_j) \sum_{k=1}^{n} \norm{Q^{k}} \delta_{a_{k}}(A_{j}) \\
  & = & \sum_{k=1}^{n} \norm{Q^{k}} \sum_{j=1}^{N} (t_{j}-s_j)  \delta_{a_{k}}(A_{j}) 
\end{eqnarray*}
Let $\mu$ be the measure on $(\R_{+} \times U, \mathcal{B}(\R_{+}) \otimes \mathcal{B}(U))$ given by $\mu(ds,du)=\sum_{k=1}^{n} \norm{Q^{k}} ds \delta_{a_{k}}(du)$. Then it is clear that $\mu([0,t] \times A)< \infty $ for every $t \geq 0$ and $A \in \mathcal{B}(U)$ and from the calculations above we have $\mu=\sup_{m \geq 1} \nu_{x^{*}_{m}}$. Hence by Theorem \ref{theoSuffiCondiExistQuadraVariat} we conclude that $M$ has quadratic variation and $\operQuadraVari{M}(ds,du)= \sum_{k=1}^{n} \norm{Q^{k}} ds \delta_{a_{k}}(du)$. 
\end{example}

As the next example shows, it is not true in general that every  $M$ defined as in Example \ref{exampleCylinMartingaleDelta}  has quadratic variation.

\begin{example}\label{examCounterExampleCMVMStochasticIntegral}
For $M$ as in Example \ref{exampleCylinMartingaleIntegral} we have 
$$
|M(t,A)(x^*)|^2-\int_0^t \caract_A(g(u))d\langle Z(x^*)\rangle_u
$$
is a martingale, so 
$$
\langle M(A)(x^*)\rangle_t= \int_0^t \caract_A(g(u))d\langle Z(x^*)\rangle_u.
$$
As this process is predictable and $\sigma$-additive, we have 
$$
\nu_{x^*}( (s,t]\times A) = \int_s^t \caract_A(g(u))d\langle Z(x^*)\rangle_u.
$$
We now describe a particular $Z$ for which $M$ does not have a quadratic variation. 

Suppose $X=L^2[0,1]$. Let $\mathcal{A} =\mathcal{B}(X)$ and
$$
Z(h)_t = \int_0^t h(s)dW_s
$$
for some (real) standard Brownian motion $W$. We thus have 
$$
d\langle Z(h) \rangle_t = h^2(t) dt.
$$
Notice that (\ref{eqLocalUniformConverInProbabilityMeasuresNu}) is satisfied in this case. In fact
\begin{eqnarray*}
\sup_{A \in \mathcal{A}} \sup_{0\leq s\leq t} \left| \nu_{h_n}((s,t]\times A) - \nu_{h}((s,t]\times A) \right|
 & = & \sup_{A \in \mathcal{A}} \sup_{0\leq s\leq t} \left| \int_s^t \caract_A(g(u)) (h_n^2(u)-h^2(u)) du \right| \\
 & \leq & \int_0^t |h_n^2(u)-h^2(u)|du \fle 0 
\end{eqnarray*}
whenever $h_n\fle h$. 

If $M$ had a quadratic variation it should satisfy \eqref{eqRepreQuadraVariaAsSupremum}, however, the later does not hold in general. For example, we can fix $A=U$ and consider a dyadic partition $\mathcal{D}_k$ for $[0,1]$ of level $k$, that is, by $2^k$ subintervals of length $2^{-k}$.  Let $\{ h_n : n \geq 1 \}$ be the $L^2$-normalized Haar basis; this family is dense on the unit ball of $L^2[0,1]$, and since each $h_n$ is a norm one function supported on some dyadic interval $I_n$, we have
$$
\nu_{h_n}(I_n \times U) \geq \int_{I_n} h_n^2(u)du = 1, \qquad \forall\, n \geq 1.
$$
Therefore, 
$$
\sum_{I \in \mathcal{D}_k} \sup_{n \geq 1} \nu_{h_n}( I \times U) \geq 2^k.
$$
Note that this sum is over one of the partitions considered in \eqref{eqRepreQuadraVariaAsSupremum}, therefore, the right hand side of that equation should be infinite. 
\end{example}

\subsection{The quadratic variation operator measure}
\label{subsectQuadraVariationoperator}

Throughout this section we assume that $M$ has a  quadratic variation and the family of intensity measures satisfies the sequential boundedness property. By Theorem \ref{theoUniquenessQuadraticVariation} the quadratic variation of $M$ is unique. We use the following notation, for the covariation of two real-valued processes $X$ and $Y$: 
$$\inner{X}{Y}_{s}^{t} \defeq \inner{X}{Y}_{t}-\inner{X}{Y}_{s}.$$ 

Set $T>0$. For any given $0 \leq s \leq t \leq T$ and $A \in \calA$, the mapping 
$ (x^* ,y^*) \mapsto  \quadraVari{M(A)(x^*),M(A)(y^*)}_{s}^{t}$ defines a bilinear form on $X^* \times X^*$  taking values in the space of real-valued random variables. Conversely, for any given $x^*, y^* \in X^*$,  $(s,t] \times A \mapsto \quadraVari{M(A)(x^*),M(A)(y^*)}_{s}^{t}$ defines a finitely additive random signed measure on the ring of subsets of $[0,T] \times U$ generated by the sets of the form $\{ 0 \} \times A$, and $(s,t] \times A$ for $ 0 \leq s < t \leq T$, $A \in \mathcal{A}$.   

Define $\alpha_{M}$ for every $x^*,y^* \in X^*$ as the random set function on $\mathcal{B}([0,T]) \otimes \mathcal{B}(U)$ given by
\begin{equation}\label{eqDefiAlphaMRandomMeasure}
 \alpha_{M}(\omega)(C)(x^*,y^*)=\tfrac{1}{4} \left( \nu_{x^*+y^*}(C)- \nu_{x^*-y^*}(C) \right),\quad C \in \mathcal{B}([0,T]) \otimes \mathcal{B}(U).
\end{equation}
We will assume that for every  $x^*, y^* \in X^*$, $\alpha_M(\omega)(\cdot)(x^*, y^*)$ is a well defined signed measure on $\mathcal{B}([0,T]) \otimes \mathcal{B}(U)$. We will show (see Theorem \ref{theoAlphaMBilinearVectorMeasure} below) that under mild assumptions on $\operQuadraVari{M}$ we have that $\alpha_{M}$ extends to a $\goth{Bil}(X^*,X^*)$-valued measure such that for each rectangle $(s,t] \times A$,  $\alpha_{M}(\omega)((s,t] \times A)(x^*,y^*)$ equals $\quadraVari{M(A)(x^*),M(A)(y^*)}_{s}^{t}(\omega)$ $\Prob$-a.e.  Our first step is the following:

\begin{theorem}\label{theoAlphaContinuousBilinearFormOnTheRing} For $\Prob$-a.e. $\omega \in \Omega$, 
$ \alpha_{M}(\omega)((s,t]\times A) \in \goth{Bil}(X^*,X^*)$  and 
$$\norm{\alpha_{M}(\omega)((s,t]\times A)}_{\goth{Bil}(X^*,X^*)} \leq \operQuadraVari{M}(\omega)((s,t]\times A),$$
for all $0 \leq s \leq t$, $A \in \calA$.
\end{theorem}

The proof of Theorem \ref{theoAlphaContinuousBilinearFormOnTheRing} will be carried out in several steps.  Fix $(x_{n}^*)_{n \geq 1} \subseteq X^*$ a set of linearly independent vectors such that $\mbox{span}(x_{n}^*)_{n \geq 1} $ is dense in $X^*$. Let $F=(y_{m}^*)_{m \geq 1}$ be the $\Q$-span of $(x_{n}^*)_{n \geq 1}$. 

Fix $0 \leq s \leq t$, $A \in \calA$. Denote by $\widehat{\Omega}(s,t,A)$ the set of all $\omega \in \Omega$ such that $(x^*,y^*) \mapsto \inner{M(A)(x^*) }{M(A)(y^*)}^{t}_{s}(\omega)$ is a bilinear form on $F \times F$. By the countability of $F$ and a.s. linearity of the (real) quadratic convariation, we have $\Prob ( \widehat{\Omega}(s,t,A) )=1$. 

Similarly, let $\tilde{\Omega}(s,t,A)$ be the set of all $\omega \in \Omega$ such that for all $x^*, y^* \in F$ we have
\begin{equation}\label{eqPolarizaNuEqualQuadraticCovariation}
\tfrac{1}{4} \left( \nu_{x^*+y^*}(\omega)((s,t]\times A) - \nu_{x^*-y^*}(\omega)((s,t]\times A) \right)
= \inner{M(A)(x^*) }{M(A)(y^*)}^{t}_{s}(\omega).
\end{equation}
By the countability of $F$ and Lemma \ref{lemmaExistMeasMuRealQuadraVaria}, we have $\Prob ( \tilde{\Omega}(s,t,A) )=1$. 
Moreover, notice that $\forall \omega \in \tilde{\Omega}(s,t,A)$ we have by \eqref{eqDefiAlphaMRandomMeasure} and \eqref{eqPolarizaNuEqualQuadraticCovariation} that
\begin{equation}\label{eqIdentityAlphaMAndQuadraticCovariation}
 \alpha_{M}(\omega)((s,t]\times A)(x^*,y^*)= \inner{M(A)(x^*) }{M(A)(y^*)}^{t}_{s}(\omega), \quad \forall x^*, y^* \in F.     
\end{equation}
Let  $\Lambda(s,t,A) \defeq \{ \omega :  \alpha_{M}(\omega)((s,t]\times A) \text{ is a bilinear form on } F \times F \}$. Since $ \widehat{\Omega}(s,t,A) \cap \tilde{\Omega}(s,t,A) \subseteq \Lambda(s,t,A)$ we conclude that $\Prob (\Lambda(s,t,A))=1$. 

\begin{lemma}\label{lemmaAlphaMBilinearFormOnQSpanF}
For $\Prob$-a.e. $\omega \in \Omega$,    $ \alpha_{M}(\omega)(C)$ is a bilinear form on $F\times F$ for all $C \in \mathcal{B}([0,T]) \otimes \mathcal{B}(U)$.
\end{lemma}
\begin{proof}
Fix $r \in \Q$ and $x^*,y^*,z^* \in F$. 
For $0 \leq s \leq t$, $A \in \calA$ and $\omega \in \Lambda(s,t,A)$  we have
$$
\alpha_{M}(\omega)((s,t]\times A)(rx^*+y^*,z^*)= r\alpha_{M}(\omega)((s,t]\times A)(x^*,z^*) + \alpha_{M} (\omega)((s,t]\times A)(y^*,z^*).     
$$
Then, by Theorem \ref{theoremComparisonSignedMeasuresOnRing} we have the existence of a set $\Lambda_{1}(r, x^*, y^*, z^*) \subseteq \Omega$ with probability 1 such that 
$$\alpha_{M}(\omega)(\cdot)(r x^* + y^*,z^*)= r \alpha_{M}(\omega)(\cdot)(x^*,z^*)+ \alpha_{M} (\omega)(\cdot)(y^*,z^*),$$
for all $\omega \in \Lambda_{1}(r,x^*,y^*,z^*)$. Likewise we can show the existence of a set $\Lambda_{2}(r,x^*,y^*,z^*) \subseteq \Omega$ with probability 1 such that 
$$\alpha_{M}(\omega)(\cdot)(x^*,ry^*+ z^*)= r\alpha_{M}(\omega)(\cdot)(x^*,y^*) + \alpha_{M} (\omega)(\cdot)(x^*,z^*),$$
for all $\omega \in \Lambda_{2}(r,x^*,y^*,z^*)$. Setting
$$ \Lambda = \bigcap_{(r,x^*,y^*,z^*)\in \Q  \times F^3} \Lambda_{1}(r,x^*,y^*,z^*) \cap \Lambda_{2}(r,x^*,y^*,z^*) ,$$
we obtain a subset of $\Omega$ of probability 1 for which  $ \alpha_{M}(\omega)(C)$ is a bilinear form on $F\times F$ for all $C \in \mathcal{P}_{T} \otimes \mathcal{B}(U)$.
\end{proof}

\begin{lemma}\label{lemmaContiAlphaMDependsOnRectangle}
For every $0 \leq s \leq t$ and $A \in \calA$ we have $\Prob$-a.e. 
$$\abs{\alpha_{M}(\omega)((s,t] \times A)(x^*,y^*)} \leq \operQuadraVari{M}((s,t] \times A)\norm{x^*}\norm{y^*}, \quad \forall \, x^*,y^* \in F.$$ 
\end{lemma}
\begin{proof}
Fix $0 \leq s \leq t$ and $A \in \calA$. Let $\widetilde{F} = \left\{ \frac{x^*}{\| x^* \|} : 0\neq x^* \in F \right\}$, this set is countable and dense in the unit sphere of $X^*$.  By Theorem \ref{theoUniquenessQuadraticVariation}, there is a set $\Gamma(s,t,A) \subseteq \Omega$ of probability 1 such that 
$ \left( \sup_{x^* \in \widetilde{F}} \nu_{x^{*}} \right)((s,t] \times A) =  \operQuadraVari{M}((s,t] \times A)$. 

If $0 \neq x^* \in F$, for $\omega \in  \widehat{\Omega}(s,t,A) \cap \tilde{\Omega}(s,t,A)  \cap \Gamma(s,t,A)$ we have by \eqref{eqPolarizaNuEqualQuadraticCovariation} that
\begin{eqnarray*}
\operQuadraVari{M}((s,t] \times A) 
& \geq & \nu_{\frac{x^*}{\norm{x^{*}}}}((s,t] \times A) = \quadraVari{M(A)\left( \frac{x^*}{\norm{x^{*}}} \right)}_{s}^{t} \\
& = & \frac{1}{\norm{x^{*}}^2}\quadraVari{M(A)(x^*)}_{s}^{t} =  \frac{1}{\norm{x^{*}}^2} \nu_{x^*}((s,t] \times A),
\end{eqnarray*}
and thus $\nu_{x^*}((s,t] \times A) \leq \operQuadraVari{M}((s,t] \times A) \norm{x^{*}}^2 $.

Now, for  $\omega \in  \widehat{\Omega}(s,t,A) \cap \tilde{\Omega}(s,t,A)  \cap \Gamma(s,t,A)$, and any $0 \neq x^* , y^* \in F$ we have by   \eqref{eqIdentityAlphaMAndQuadraticCovariation}, the Kunita-Watanabe inequality (e.g. see Theorem 11.4.1 in \cite{CohenElliott:2015}, p.240), and \eqref{eqPolarizaNuEqualQuadraticCovariation}, that 
\begin{eqnarray*}
\abs{\alpha_{M}(\omega)((s,t]\times A)(x^*,y^*)}
& = &  \abs{ \inner{M(A)(x^*) }{M(A)(y^*)}_{s}^{t}} \\ 
& \leq & \sqrt{ \quadraVari{ M(A)(x^*)}_{s}^{t} } \cdot \sqrt{ \quadraVari{ M(A)(y^*)}_{s}^{t}} \\
& = & \sqrt{ \nu_{x^*}((s,t] \times A) } \cdot \sqrt{ \nu_{y^*}((s,t] \times A)} \\
& \leq & \operQuadraVari{M}((s,t] \times A)\norm{x^*}\norm{y^*}.
\end{eqnarray*}
This finishes the proof, since $\Prob( \widehat{\Omega}(s,t,A) \cap \tilde{\Omega}(s,t,A) \cap \Gamma(s,t,A))=1$. 
\end{proof}

\begin{lemma}\label{lemmaContiAlphaMBilinearForm}
For $\Prob$-a.e. $\omega \in \Omega$, 
$$ \abs{\alpha_{M}(\omega)(C)(x^*,y^*)} \leq \operQuadraVari{M}(C)\norm{x^*}\norm{y^*}, \quad \forall \, C \in \mathcal{B}([0,T]) \otimes \mathcal{B}(U), \, x^*,y^* \in F.$$
\end{lemma}
\begin{proof}
It follows from Theorem \ref{theoremComparisonSignedMeasuresOnRing} and Lemma \ref{lemmaContiAlphaMDependsOnRectangle}. 
\end{proof}

\begin{proof}[Proof of Theorem \ref{theoAlphaContinuousBilinearFormOnTheRing}]  
By Lemmas \ref{lemmaAlphaMBilinearFormOnQSpanF}  and \ref{lemmaContiAlphaMBilinearForm} there exists $\Omega_{0} \subseteq \Omega$ with probability 1, such that for each $\omega \in \Omega_{0}$, we have    $ \alpha_{M}(\omega)((s,t]\times A) \in \goth{Bil}(F,F)$  and 
$$\norm{\alpha_{M}(\omega)((s,t]\times A)}_{\goth{Bil}(F,F)} \leq \operQuadraVari{M}(\omega)((s,t]\times A) <\infty,$$
for all $0 \leq s \leq t$, $A \in \calA$. Since $\alpha_{M}(\omega)((s,t]\times A)$ is bounded on $F\times F$, it can be extended to $X^* \times X^*$. 
\end{proof}

\begin{theorem}
\label{theoAlphaMBilinearVectorMeasure}
Assume that, for $\Prob$-a.e. $\omega \in \Omega$,
\begin{equation}\label{eqBoundedrangequadraticvariation}
    \sup_{A \in \calA} \operQuadraVari{M}(\omega)([0,T]\times A) <\infty.
\end{equation}
Then there exists a random $\goth{Bil}(X^*,X^*)$-valued measure $\alpha_{M}$ defined on $\mathcal{B}([0,T]) \otimes \mathcal{B}(U)$ such that for all $x^*,y^* \in X^*$, $0 \leq s \leq t\leq T$ and $A \in \mathcal{A}$, $\Prob$-a.e. $\omega \in \Omega$,
\begin{equation}
\label{eqDefiAlphaMCovariation}
\alpha_{M}(\omega)((s,t] \times A)(x^*,y^*) \\ 
=  \quadraVari{M(A)(x^*),M(A)(y^*)}_{s}^{t}(\omega).
\end{equation}
\end{theorem}

\begin{proof}
By Theorem \ref{theoAlphaContinuousBilinearFormOnTheRing} we have that, $\Prob$-a.e. $\alpha_{M}: \mathcal{R}_T \rightarrow \goth{Bil}(X^*,X^*)$ is a weakly countably additive measure, where $\mathcal{R}_T$ denotes the ring of subsets of $[0,T] \times U$ generated by the sets of the form $\{ 0 \} \times A$, and $(s,t] \times A$ for $ 0 \leq s \leq t \leq T$, $A \in \mathcal{A}$. 
By Theorem \ref{theoAlphaContinuousBilinearFormOnTheRing} and Equation \eqref{eqBoundedrangequadraticvariation}, $\Prob$-a.e. the range of $\alpha_{M}$ on $\mathcal{R}_T$ is bounded in $\goth{Bil}(X^*,X^*)$. 
Then by the Carath\'{e}odory-Hahn-Kluvanek extension theorem (see \cite{Kluvanek:1972}; see also Theorem I.5.2 in \cite{DiestelUhl:1977}, p.27)  $\alpha_{M}$ has a unique countably additive extension to a $ \goth{Bil}(X^*,X^*)$-valued measure on the $\sigma$-ring $\mathcal{S}_{T}$ generated by the ring $\mathcal{R}_T$. 

As for our final step, we will show that $\mathcal{S}_{T}= \calB([0,T]) \otimes \mathcal{B}(U)$. To see why this is true, recall that by definition $\mathcal{S}_{T}$ is the smallest $\sigma$-ring containing $\mathcal{R}_T$, in particular $\mathcal{S}_{T}$ is closed under countable unions. Therefore, because $U_{n} \in \mathcal{A}$, $\forall n \in \N$ and $U = \bigcup_{n \in \N} U_{n}$, we have $[0,T] \times U \in \mathcal{S}_{T}$. Hence, $\mathcal{S}_{T}$ is a $\sigma$-algebra and consequently $\mathcal{S}_{T}= \calB([0,T]) \otimes \mathcal{B}(U)$. 
\end{proof}


We are ready to introduce the \emph{quadratic variation operator measure}.
The  random measure $\alpha_{M}$ induces a random measure $\Gamma_{M}$ 
on $\calB([0,T]) \otimes \mathcal{B}(U)$ taking values in $\mathcal{L}(X^*,X^{**})$ by means of the prescription 
\begin{equation}\label{eqDefiMeasureGammaM}
\inner{\Gamma_{M}(\omega)(C)x^*}{y^*}=\alpha_{M}(\omega)(C)(x^*,y^*), \quad \forall\,  \, x^*, y^* \in X^*, \, C \in \calB([0,T]) \otimes \mathcal{B}(U).     
\end{equation}
The reader must be aware that on the left hand side of \eqref{eqDefiMeasureGammaM}, $\inner{\cdot}{\cdot}$ corresponds to the duality relation for the pair $(X^*, X^{**})$.

From Theorem \ref{theoAlphaContinuousBilinearFormOnTheRing} we get the following useful inequality
\begin{equation}
\label{eqAcotGammaDoleans}
\abs{\inner{\Gamma_{M}(\omega)(\cdot)x^*}{y^*}} \leq \norm{x^*}\norm{y^*} \operQuadraVari{M}(\omega)(\cdot) \text{ on }\ \calB([0,T]) \otimes \mathcal{B}(U).
\end{equation}

The proof of the existence of the operator-valued quadratic variation and some of its properties are contained in the following theorem.  

\begin{theorem}\label{theoExistenCovariaOperatorQ}
Let $T>0$, and assume that $M$ satisfies \eqref{eqBoundedrangequadraticvariation}. 
Then there exists a  process $Q_{M}: \Omega \times [0,T]  \times U \rightarrow \mathcal{L}(X^*, X^{**})$ such that $\Prob$-a.e. $\omega \in \Omega$
\begin{equation}\label{existenceofQ}
\inner{\Gamma_{M}(\omega)(C)x^*}{y^*}= \int_{C} \inner{Q_{M}(\omega,r,u)x^*}{y^*} \, \operQuadraVari{M}(\omega)(dr,du)    
\end{equation}
for all $x^*,y^* \in X^*$, $C \in \mathcal{B}([0,T]) \otimes \mathcal{B}(U)$. Moreover the following properties hold:
\begin{enumerate}
    \item For every $x^*, y^* \in X^*$, the mapping $(\omega,r,u) \mapsto \inner{Q_{M}(\omega,r,u)x^*}{y^*}$ is predictable, that is, $\calP_T \otimes \calB(U)$-measurable. \label{Qpredictable}
    \item $\Prob$-a.e. $\omega \in \Omega$, $Q_{M}(\omega,\cdot,\cdot)$ is positive and symmetric $\operQuadraVari{M}$-a.e. \label{Qpositiveandsymmetric}
    \item $\Prob$-a.e. $\omega \in \Omega$, $\norm{Q_{M}(\omega,\cdot,\cdot)}_{\mathcal{L}(X^*,X^{**})}=1$ $\operQuadraVari{M}$-a.e. \label{normoneoftheoperatorQ}
\end{enumerate}
\end{theorem}

\begin{proof} First observe that by \eqref{eqAcotGammaDoleans}, there is a full probability set $\Omega_0 \subseteq \Omega$  such that for $\omega \in \Omega_0$ we have $\inner{\Gamma_{M}(\omega)x^*}{y^*} \ll \operQuadraVari{M}(\omega)$ on $ \mathcal{B}([0,T]) \otimes \mathcal{B}(U)$   for each $x^*, y^* \in X^*$, and hence there is a Radon-Nikodyn density $q_{x^{*},y^{*}}(\omega)$ of the real-valued measure $\inner{\Gamma_{M}(\omega)x^*}{y^*} $ with respect to $\operQuadraVari{M}(\omega)$. This density is $\calB([0,T]) \otimes \calB(U)$-measurable and satisfies $\abs{q_{x^{*},y^{*}}(\omega)(r,u)} \leq \norm{x^{*}}\norm{y^{*}}$, $\operQuadraVari{M}(\omega)$-a.e. 

Since $X^*$ is separable, one can choose a modification of $q_{x^{*},y^{*}}$, such that for every $x^*, y^{*} \in X^*$ the following holds true (see the construction in the proof of Theorem 1.2.34 in \cite{Dinculeanu:2000}, p.37): 
\begin{enumerate}[label=\alph*)]
    \item For  every $(\omega,r,u)$ and $x^{*} \in X^{*}$, the mapping $q_{x^{*}}(\omega)(r,u): y^{*} \mapsto q_{x^{*},y^{*}}(\omega)(r,u)$ is a continuous linear form on $X^*$ satisfying $\inner{q_{x^{*}}(\omega)(r,u)}{y^{*}}=q_{x^{*},y^{*}}(\omega)(r,u)$ and  $\norm{q_{x^{*}}(\omega)(r,u)} \leq \norm{x^*}$. 
    \item For each $(\omega,r,u)$, the mapping $Q_{M}(\omega,r,u): x^{*} \mapsto q_{x^{*}}(\omega)(r,u)$ belongs to $\mathcal{L}(X^*, X^{**})$,  $ \inner{Q_{M}(\omega,r,u)x^*}{y^*} = q_{x^{*},y^{*}}(\omega)(r,u)$ for each $x^*  , y^* \in X^*$ and $\norm{Q_{M}} \leq 1$. 
\end{enumerate}

By the properties described above, we get that the mapping  
$Q_{M}:  \Omega \times \R_{+} \times U \rightarrow \mathcal{L}(X^{*},X^{**})$ satisfies \eqref{existenceofQ}.

To prove \ref{Qpredictable}, we use a modification of the proof for Theorem 3 in \cite{Radchenko:1989}.  Let $\mathcal{G} = \{ G_k : k \geq 1 \}$ be the collection of rectangles of the form $(s,t]\times A$, where $s < t$ are rational, and $A$ belongs to a countable family of sets that generates $\calB(U)$; this collection generates $\calB = \calB([0,T]) \otimes \calB(U)$.  For $n \in \mathbb{N}$, let $\Pi_n$ denote the finest partition of $[0,T]\times U$ by sets of the (finite) algebra generated by $\mathcal{G}_n = \{ G_k : 1 \leq k \leq n \}$.  Note that $\Pi_n \subseteq \Pi_{n+1}$ for every $n$, and $\sigma(\bigcup_{n \geq 1} \Pi_n) = \bigcup_{n \geq 1}\sigma(\Pi_n) =   \calB$.  Also, if $C \in \Pi_n$, then $C$ is of the form $(s,t] \times A$, where $s,t \in \mathbb{Q}$ and $A \in \calB(U)$.

For $\omega \in \Omega_0$ and $x^*, y^* \in X^*$, define the sequence of functions $( \phi_n : n\geq 1 )$ by 
$$
\phi_n(\omega)(r,u) = \frac{ \inner{\Gamma_{M}(\omega)(C)x^*}{y^*} }{ \operQuadraVari{M}(\omega)(C) }, \quad \mbox{if } (r,u) \in C \in \Pi_n, \ \operQuadraVari{M}(\omega)(C) > 0.
$$
For fixed $\omega \in \Omega_0$ and $\lambda > 0$, the level set $\{ (r,u) \in [0,T]\times U : \phi_n(\omega)(r,u) > \lambda  \}$ is the (finite) union of those sets $C$ in $\Pi_n$ that satisfy
$$
\inner{\Gamma_{M}(\omega)(C)x^*}{y^*} > \lambda \operQuadraVari{M}(\omega)(C), \quad \operQuadraVari{M}(\omega)(C) > 0. 
$$
If $C = (s,t] \times A$, the predictability of both measures tells us that the set of possible values $\omega$ for which the previous conditions are satisfied is $\calF_s$-measurable, which implies that the functions $\phi_n$ are predictable. 

Finally, by Theorem 48.3 in \cite{Parthasarathy:1978}, we have that the sequence $\phi_n(\omega)$ converges to $q_{x^*, y^*}(\omega)$ $\operQuadraVari{M}(\omega)$-a.e.  Therefore the map $(\omega, r, u) \mapsto q_{x^*, y^*}(\omega)(r,u)$ is predictable, and so is $\inner{Q_M (\cdot) x^*}{y^*}$. 

By the construction of $Q_M$ and the predictability proven in the previous step, \ref{Qpositiveandsymmetric} holds true.

We have, by the previous construction, that for all $\omega \in \Omega_0$, $\norm{Q_{M}(\omega,r,u)} \leq 1$ $\forall (r,u)\in [0,T]\times U$. We claim that for all $\omega \in \Omega_0$, $\norm{Q_{M}(\omega,\cdot,\cdot)} = 1$, $\operQuadraVari{M}$-a.e. on $[0,T] \times U$. 

Suppose that there is $0< \beta < 1$ such that the set $C = \{ (r,u) \in [0,T] \times U : \| Q_M(\omega, r, u) \| \leq \beta \}$ has positive $\operQuadraVari{M}$-measure. Let $(x_n^*)$ be any dense sequence on the unitary sphere. Since by Theorem \ref{theoUniquenessQuadraticVariation} $ \operQuadraVari{M}(\omega)= \sup_{\| x_n^* \| = 1 } \nu_{x_n^{*}}(\omega)  = \sup_{n\geq 1} \alpha_{M}(\omega)(x_n^{*},x_n^{*})$.   
We have, by \eqref{eqDefiMeasureGammaM} and \eqref{existenceofQ} that  
\begin{align*}
\alpha_{M}(\omega)(C)(x_n^{*},x_n^{*}) & =  \inner{\Gamma_M(\omega)(C)x_n^*}{x_n^*} \\
& = \int_{C}  \inner{Q_{M}(\omega,r,u)x_n^*}{x_n^*} \, \operQuadraVari{M}(\omega)(dr,du) \leq \operQuadraVari{M}(C)
\end{align*}
Taking the supremum over $n$, we obtain $\operQuadraVari{M}(C) \leq \beta \operQuadraVari{M}(C)$, which is only possible if $\operQuadraVari{M}(C) = 0$.  Therefore \ref{normoneoftheoperatorQ} is satisfied.
\end{proof}

In many practical situations one can compute $Q_{M}$ via the identity 
\begin{equation}\label{eqAlphaMCalculationQM}
\alpha_{M}(\omega)( (s,t] \times A ) (x^{*},x^{*}) = \int_{(s,t] \times A}  \inner{Q_{M}(\omega,r,u)x^*}{x^*} \, \operQuadraVari{M}(\omega)(dr,du), 
 \end{equation}
by calculating $\alpha_{M}$ and $\operQuadraVari{M}$ beforehand. This idea is explored in the following examples.

\begin{example}\label{exampleCylinMartingaleDeltaLocBounCova}
Let $T>0$ and let $M$ denotes the cylindrical martingale-valued measure defined in Example \ref{examCMVMCylindriLevyProcesses}. We have 
$$ \nu_{x^{*}}(\omega)(ds,du)=\sum_{k=1}^{n} \lambda_{\quadraVari{Z^{k}(\omega)(x^{*})}}(ds) \delta_{a_{k}}(du) = \sum_{k=1}^{n}  ( Q^{k} (x^{*} ) x^{*}) ds\delta_{a_{k}}(du) , $$
and 
$$\operQuadraVari{M}(\omega)(ds,du)= \sum_{k=1}^{n} \norm{Q^{k}} ds \delta_{a_{k}}(du).$$ 
Observe that \eqref{eqBoundedrangequadraticvariation} is satisfied, since
$$\sup_{A \in \calA} \operQuadraVari{M}(\omega)([0,T]\times A) \leq T \sum_{k=1}^{n} \norm{Q^{k}} < \infty. $$
Then there exists a process $Q_{M}: \Omega \times [0,T]  \times U \rightarrow \mathcal{L}(X^*, X^{**})$ satisfying the conditions in Theorem \ref{theoExistenCovariaOperatorQ}.

Notice that by \eqref{eqDefiAlphaMRandomMeasure} we have 
$$ \alpha_{M}(\omega)( (s,t] \times A) (x^{*},x^{*})=  \tfrac{1}{4}\nu_{2 x^{*}}(\omega)((s,t] \times A))
= (t-s) \sum_{k=1}^{n} (Q^{k}(x^{*})x^{*}) \delta_{a_{k}}(A).$$

Then by \eqref{eqAlphaMCalculationQM} we conclude
$$ Q_{M}(\omega, r, u) = \sum_{k=1}^{n} \frac{Q^{k}}{\norm{Q^{k}}} \mathbbm{1}_{\{a_{k}\}}(u). $$
\end{example}

\begin{example}\label{exampleCMVMWithFamilySeminorms}
Let $H$ be a separable Hilbert space. In \cite{AlvaradoFonseca:2021} they consider cylindrical martingale-valued measures  for which the following properties are satisfied:
\begin{enumerate} 
\item \label{indepIncrementsCMVM} For $0\leq s < t$, $(M(t,A)- M(s,A))(h)$ is independent of $\mathcal{F}_{s}$, for all $A \in \mathcal{A}$, $h \in H$. 
\item \label{nuclearCovarianceCMVM} For each $A \in \mathcal{A}$ and $0 \leq s < t$, 
$$
\Exp \left[ \abs{ (M(t,A)- M(s,A))(h)}^{2} \right] = \int_{s}^{t} \int_{A} q_{r,u}(h)^{2} \gamma(du) m(dr) , \quad \forall \, h \in H.
$$
where 
\begin{enumerate}
	\item  $m$ is a $\sigma$-finite measure on $(\R_{+},\mathcal{B}(\R_{+}))$, finite on bounded intervals,
 	\item $\gamma$ is a $\sigma$-finite measure on $(U, \mathcal{B}(U))$ satisfying $\gamma(A)< \infty$, $\forall \, A \in \mathcal{A}$,
	\item $\{q_{r,u}: r \in \R_{+}, \, u \in U \}$ is a family of  continuous Hilbertian semi-norms on $H$, such that for each $h_{1}$, $h_{2}$ in $H$, the map $(r,u) \mapsto q_{r,u}(h_{1},h_{2})$ is $\mathcal{B}(\R_{+}) \otimes \mathcal{B}(U)/ \mathcal{B}(\R_{+})$-measurable and  bounded on $[0,T] \times U$ for all $T>0$. Here, $q_{r,u}(\cdot,\cdot)$ denotes the continuous positive, symmetric, bilinear form on $H \times H$  associated to $q_{r,u}(\cdot) $.  
\end{enumerate} 
\end{enumerate}

One can infer from \ref{indepIncrementsCMVM} and \ref{nuclearCovarianceCMVM} above that the
family of intensity measures of $M$ is of the form:
$$ \nu_{h}(\omega)(C) = \int_{C}
q_{r,u}(h)^{2} (m \otimes \gamma) (dr,du), \quad \forall \, C \in \mathcal{B}(\R_{+}) \otimes \mathcal{B}(U),$$

We shall prove that the family $(\nu_{h})$ satisfies the sequential boundedness property. 
By Lemma 3.5 in \cite{AlvaradoFonseca:2021} for any given $T>0$ there exists $K=K(T)>0$ such that 
\begin{equation}\label{eqUniforBoundedSeminormsQru}
q_{r,u}(\cdot) \leq K \norm{\cdot}, \quad \forall (r,u) \in [0,T] \times U.   
\end{equation}
Let $(h_n)$ be a sequence on the unit sphere such that $h_n\fle h$. 
\begin{eqnarray*}
    \abs{\nu_{h_n}(\omega)(C)- \nu_{h}(\omega)(C)} & \leq & 
    \int_C \abs{q_{r,u}(h)^2 - q_{r,u}(h_n)^2} (m\otimes \gamma)(dr,du) \\
 & \leq & 2K^2 \norm{h_n-h} (m\otimes \gamma)(C) \fle 0.
\end{eqnarray*}
It follows that
$$
\nu_{h}(\omega)(C) = \lim \nu_{h_n}(\omega)(C) \leq \sup_{n\geq 1}\nu_{h_n}(\omega)(C).
$$
By Remark \ref{remaSufCondSeqBound} we conclude that $(\nu_{h})$ satisfies the sequential boundedness property.

Now we prove that $M$ has a quadratic variation. First notice that for all $C \in \mathcal{B}(\R_{+}) \otimes \mathcal{B}(U)$, 
$$  \nu_{h}(\omega)(C) \leq  K^{2} \int_{C}
\norm{h}^{2} (m \otimes \gamma ) (dr,du) = K^{2} \norm{h}^{2} (m \otimes \gamma)(C). $$
We can construct a measure that dominates each $\nu_{h}$ for $\norm{h}=1$. In fact, for each $n \in \N$ let $K_{n}$ with the property \eqref{eqUniforBoundedSeminormsQru}.   Define a measure $\kappa $ on  $\mathcal{B}(\R_{+}) \otimes \mathcal{B}(U)$ by the prescription
$$ \kappa (C)= \sum_{n \in \N} K_{n} (m \otimes \gamma)(C \cap ([n-1,n) \times U)).$$
Observe that for $t>0$ and $A \in \mathcal{A}$, if $n-1 \leq  t <n$ then 
$$ \kappa([0,t] \times A) \leq \left( \max_{1 \leq k \leq n} K_{n} \right) m([0,t]) \gamma(A) < \infty. $$

Moreover, from the calculations above we have $\nu_{h} \leq \kappa$ on $\mathcal{B}(\R_{+}) \otimes \mathcal{B}(U)$  $\forall \norm{h}=1$. Therefore 
$$
\sup_{n\geq 1} \nu_{h_n} \leq \sup_{\norm{h}=1} \nu_h \leq \kappa.
$$
Then, by Theorem \ref{theoSuffiCondiExistQuadraVariat}, $M$ has a unique quadratic variation  $\operQuadraVari{M}=\sup_{n\geq 1} \nu_{h_n}$. Furthermore by Lemma \ref{lemmIntegralOfSupremum}, we have 
$$ \operQuadraVari{M}(C)=\sup_{n\geq 1} \nu_{h_n}(C)=
\int_{C} \sup_{n\geq 1} q_{r,u}(h_n)^{2} (m \otimes \gamma ) (dr,du),$$
for all $C \in \mathcal{B}(\R_{+}) \otimes \mathcal{B}(U)$. 
If we further assume that 
$$
\int_{0}^{T} \int_{U} \sup_{n} q_{r,u}(h_n)^{2} \gamma(du) m(dr) < \infty,
$$
(this holds true if, for example, $\gamma(U)< \infty$ since by \eqref{eqUniforBoundedSeminormsQru} we have $\mu_{M} (\Omega \times [0,T] \times U) \leq K m([0,T])\gamma(U)$) then \eqref{eqBoundedrangequadraticvariation} is satisfied, and hence the conditions in Theorem \ref{theoExistenCovariaOperatorQ}.  Our next goal is to calculate the  process $Q_{M}: \Omega \times [0,T]  \times U \rightarrow \mathcal{L}(H, H)$; 
to do this, notice that by definition we have 
$$ \alpha_{M}(\omega)((s,t] \times A) (h,h)=  \tfrac{1}{4} \nu_{2 h}((s,t] \times A)
= \int_{s}^{t} \int_{A} q_{r,u}(h)^{2} \gamma(du) m(dr).$$

Then by \eqref{eqAlphaMCalculationQM} we conclude that $Q_{M}(\omega, r, u) \in \mathcal{L}(H,H)$ is defined via  the identity
$$ (Q_{M}(\omega, r, u) h,g)_{H} = \frac{q_{r,u}(h,g)}{\norm{ q_{r,u}(\cdot,\cdot)}_{\goth{Bil}(H,H)}}, \quad \forall h,g \in H,$$
where $q_{r,u}(\cdot,\cdot)$ denotes the continuous positive, symmetric, bilinear form on $H \times H$ associated to the continuous Hilbertian semi-norm $q_{r,u}$. 
\end{example}

\begin{example}\label{exampleLevyHValuedMVM} In this example we recall a construction of a cylindrical martingale-valued measures defined by an $H$-valued c\`{a}dl\`{a}g L\'{e}vy process $L=(L_{t}: t \geq 0)$. For definitions and notation we refer to  Example 3.2 in \cite{AlvaradoFonseca:2021}. 
Assume $L$ has a L\'{e}vy-It\^{o} decomposition
\begin{equation}\label{eqLevyItoDecomp}
L_{t}= t \xi +W_{t}+\int_{\norm{h} <1} \, h \,  \widetilde{N}(t,dh)+ \int_{\norm{h} \geq 1} \, h \, N(t,dh). 
\end{equation}
Let $\lambda$ denotes the L\'{e}vy measure of $L$. Recall  $\lambda$ is a Borel measure on $H$ with $\lambda( \{ 0 \})=0$, and  $\displaystyle{ \int_{H} \norm{h}^{2} \wedge 1 \,  \lambda (dh)< \infty}$.  Let $Q$ denotes the positive trace class operator corresponding to the covariance operator of the Wiener process $W=(W_{t}:t \geq 0)$.

Let $U \in \mathcal{B}(H)$ be such that $0 \in U$ and $\int_{U} \, \norm{u}^{2} \lambda(du)< \infty$. Take $ \mathcal{A}= \{ A \subseteq U : A - \{0\} \textup{ is bounded below}\}$ and let $M=(M(t,A): t \geq 0, A \in \mathcal{A})$ be given by
\begin{equation} \label{levyMartValuedMeasExam} 
M(t,A) = W_{t} \delta_{0}(A) + \int_{A \backslash \{0 \}} \,  u \, \widetilde{N}(t,du), \quad \forall \, t \geq 0, \, A \in \mathcal{A}. 
\end{equation}    
As shown in Example 3.2 in \cite{AlvaradoFonseca:2021}, $M$ given by \eqref{levyMartValuedMeasExam} is a 
particular case of the cylindrical martingale-valued measures introduced in Example \ref{exampleCMVMWithFamilySeminorms}. Indeed, with the notation introduced above we can take $m$ as the Lebesgue measure on $(\R_{+}, \mathcal{B}(\R_{+}))$, $\gamma= \delta_{0}+\restr{\lambda}{U}$, and the family $\{q_{r,u}: r \in \R_{+}, \, u \in U \}$ is given by
$$
q_{r,u}(h)^{2}=
\begin{cases}
(h,Qh)_{H}, \mbox{ if } u =0, \\
(u,h)_{H}^{2}, \mbox{ if } u \in U \setminus \{0\}. 
\end{cases}
$$ 
Then  
\begin{equation*}\label{eqMeasuresNuhLevyCMVM}
\nu_{h}((s,t] \times A)(\omega)=(t-s)   \left[ (h,Qh)_{H} \delta_{0}(A) + \int_{A \backslash \{0 \}} (u,h)_{H}^{2}  \, \lambda(du) \right]. 
\end{equation*}
and the quadratic variation of $M$ is given by 
\begin{equation}\label{eqQuadraticVariationLevyCMVM}
\operQuadraVari{M}((s,t] \times A) = (t-s) \left[ \norm{Q} \delta_{0}(A) + \int_{A \backslash \{0\}} \, \norm{u}^{2} \, \lambda(du) \right].  
\end{equation}
Furthermore, 
$$\int_{0}^{T} \int_{U} \sup_{n} q_{r,u}(h_n)^{2} \gamma(du) m(dr)   \leq   T   \left[ \norm{Q}_{\mathcal{L}_{1}(H)}  + \int_{U} \norm{u}^{2}  \, \lambda(du) \right]  < \infty. $$ 
Thus condition \eqref{eqBoundedrangequadraticvariation} is satisfied. 
Then we can apply Theorem \ref{theoExistenCovariaOperatorQ} to obtain the existence of the corresponding process $Q_{M}: \Omega \times [0,T]  \times U \rightarrow \mathcal{L}(X^*, X^{**})$.  
In this case
\begin{eqnarray*}
\alpha_{M}(\omega)((s,t] \times A) (h,h)
& = & (t-s) \left[ (h,Qh)_{H} \delta_{0}(A) + \int_{A \backslash \{0 \}} (u,h)_{H}^{2}  \, \lambda(du) \right], 
\end{eqnarray*}

Then by \eqref{eqAlphaMCalculationQM} we have
\begin{equation}\label{eqOperatorQMLevyCMVM}
Q_{M}(\omega, r, u) = \frac{Q}{\norm{Q}} \mathbbm{1}_{\{0\}}(u) + P_{u} \mathbbm{1}_{U \setminus \{0\}}(u),     
\end{equation}
where for $u \neq 0$,  $P_{u} \in \mathcal{L}(H)$ is given by 
$$ P_{u}(h)= \frac{(u,h)_{H} }{\norm{u}^2} u, \quad \forall h \in H. $$
\end{example}

Other examples of cylindrical martingale-valued measures with the conditions given in Example \ref{exampleCMVMWithFamilySeminorms} can be found in Section 3 in \cite{AlvaradoFonseca:2021}.

\subsection{A discussion on cylindrical Poisson random measures}\label{subSectCylindricalPRM}

Poisson random measures provide a very common type of noise used in the modeling of stochastic ordinary and partial differential equations (see \cite{BrzezniakLiuZhu:2014, KallianpurXiong, MandrekarRudiger, PeszatZabczykSPDE, WangYangZhaiZhang:2024, WuYangWangSong:2021} to cite but a few). For that reason, it is customary to discuss on the possibility of introducing cylindrical versions of Poisson random measures that could fit into our theory of cylindrical orthogonal martingale-valued measures. We are not aware of any work on the literature that introduces such concept. 

From now on in this Section, $U$ is a Lusin space with corresponding ring $\mathcal{A}$ as the ones considered in Section \ref{sectHilbertMartValuedMeasures}. Let $E=\R_{+} \times U$ and $\mathcal{E}=\mathcal{B}(\R_{+}) \otimes \mathcal{B}(U)$. Consider the following conditions for a mapping $N:\Omega \times \mathcal{E} \times X^{*} \rightarrow \N_{0}$.
\begin{enumerate}
    \item \label{measurePropPRM} $N(\omega,\cdot,x^{*})$ is a measure on $\mathcal{E}$ for every $\omega \in \Omega$ and $x^{*} \in X^{*}$.
    \item $N(\cdot, [0,t] \times A, x^{*})$ is $\mathcal{F}_{t}$-measurable and has a Poisson distribution with intensity $t \lambda_{A,x^{*}}$ for each $t >0$, $A \in \mathcal{A}$ and $x^{*} \in  X^{*}$.  
    \item For every $x^{*} \in  X^{*}$, $N(\cdot, \cdot, x^{*})$ is an independently scattered measure, that is,  $N(\cdot, [0,s_{1}] \times A_{1}, x^{*}), \dots, N(\cdot, [0,s_{n}] \times A_{n}, x^{*})$ are independent for disjoint $A_{1}, \dots, A_{n} \in \mathcal{A}$ and any choice of $s_{1}, \dots, s_{n} \in \R_{+}$.  
    \item \label{cylinPropertyPRM} For any given $t \geq 0$, $ A \in \mathcal{A}$, the mapping $x^{*} \mapsto N(\cdot, [0,t] \times A, x^{*})$ is linear and continuous from $X^*$ into $L^{0}(\Omega, \mathcal{F}, \Prob)$. 
\end{enumerate}
It is reasonable to define a cylindrical Poisson random measure as one satisfying conditions (i) - (iv). Although that seems reasonable, there is an incompatibility between the linearity in $x^{*}$ (the cylindrical property of $N$) and the fact that $N$ is $\N_{0}$-valued (it is Poisson distributed).  In fact, let $x^{*} \in X^{*}$ ($x^{*}\neq 0$), $t> 0$ and $A \in \mathcal{A}$. Observe that by \ref{cylinPropertyPRM} we must have $N(\cdot, [0,t] \times A, x^{*}) + N(\cdot, [0,t] \times A, -x^{*})= N(\cdot, [0,t] \times A, 0)=0$ $\Prob$-a.e., therefore either $N(\cdot, [0,t] \times A, x^{*})$ or $N(\cdot, [0,t] \times A, x^{*})$ must take negative values, which is a contradiction. A similar contradiction is reached, for instance, by comparing $\alpha N(\cdot, [0,t] \times A, x^{*})$ with $N(\cdot, [0,t] \times A, \alpha x^{*})$ for real $\alpha \notin \N_{0}$. 

The arguments in the above paragraph show that one can not define a cylindrical version of the Poisson random measure satisfying \ref{measurePropPRM}-\ref{cylinPropertyPRM} above. However, the above does not prevent us from defining a cylindrical martingale valued measure which uses a Poisson random measure as a building block. Below we illustrate this idea with a couple of examples wherein the corresponding cylindrical martingale-valued measure possesses a quadratic variation in the sense of Definition \ref{defiQuadraticVariation}. 

\begin{example}\label{exampleLinearCPoissonRM}
Let $\lambda$ be a measure on $U$ for which there exists a sequence $U_{n} \in \mathcal{A}$, $n \in \N$, such that $U_{n} \nearrow U$ and $\lambda(U_{n})<\infty$. Let $p=(p(\omega,s))$ be a stationary Poisson point processes with characteristic (intensity) measure $\lambda$. The corresponding Poisson random measure $N$, given by 
\begin{equation} \label{eqDefiPoissonRMPointProcess}
N(\omega, [0,t] \times A)=\sum_{0 \leq s \leq t} \mathbbm{1}_{A}(p(\omega,s)), \quad \forall \omega \in \Omega, \, t \geq 0, \, A \in \mathcal{B}(U),    
\end{equation}
satisfies $N(dtdu)=dt\lambda(du)$. We assume $N(\omega, [0,t] \times A)$ is finite (hence Poisson distributed) for every $A \in \mathcal{A}$. The corresponding compensated Poisson random measure is given by 
$$ \tilde{N}([0,t]\times A)=N([0,t] \times A)-t \lambda(A).$$ 

Let $X$ be a Banach space with separable dual $X^{*}$ and let $\xi \in X$, $\xi \neq 0$. Generalizing the construction in Example 3.4 in \cite{ApplebaumRiedle:2010}, define $M=(M(t,A): t \geq 0, A \in \mathcal{A})$ by
$$ M(t,A)(x^{*})=\inner{\xi}{x^{*}}\tilde{N}(\cdot, [0,t]\times A), \quad \forall t \geq 0, \, A \in \mathcal{A}, \, x^{*} \in X^{*}.$$
One can easily check that $M$ is an orthogonal cylindrical martingale-valued measure. Moreover, since
$$ \nu_{x^{*}}((s,t]\times A)=\abs{\inner{\xi}{x^{*}}}^{2} (t-s) \lambda(A),$$
the quadratic variation of $M$ is given by 
\begin{equation} \label{eqQuadraVariaExamCylPoissonRM}
 \operQuadraVari{M}((s,t]\times A)=\norm{\xi}_{X}^{2} (t-s) \lambda(A),     
\end{equation}
that is, $\operQuadraVari{M}(ds,du)=\norm{\xi}_{X}^{2} (\mbox{Leb} \otimes \lambda )(ds,du)$.

Assume moreover that the measure $\lambda$ satisfies $\sup_{A \in \calA} \lambda(A)<\infty$. Then, the quadratic variation $\operQuadraVari{M}$ of $M$ satisfies \eqref{eqBoundedrangequadraticvariation}. By Theorem \ref{theoAlphaMBilinearVectorMeasure} we have 
$$ \alpha_{M}((s,t] \times A)(x^{*},y^{*})=\inner{\xi}{x^{*}} \inner{\xi}{y^{*}}(t-s)\lambda(A).$$
Then by Theorem \ref{theoExistenCovariaOperatorQ} the corresponding process $Q_{M}: \Omega \times [0,T]  \times U \rightarrow \mathcal{L}(X^*, X^{**})$ is given by
\begin{equation}\label{operatorQMExampleCylinPoissonRM}
Q_{M}(\omega,r,u)x^{*}=\frac{\inner{\xi}{x^{*}}\xi}{\norm{\xi}_{X}^{2}}.     
\end{equation}
\end{example}

\begin{example}\label{exampleNonLinearCylinPRM}
    Let $\rho$ be a Radon $\sigma$-finite measure on $U$ such that $\rho(\{ x\})=0$ for each $x\in U$. By using Exercise $7.11$ in \cite{Folland:1999}, and the regularity of $\rho$, it can be shown that any finite measure set can be expressed as a countable union of sets of measure less than $1$. Let $\tilde{N}$ be a compensated Poisson random measure on $[0, \infty) \times U$, with intensity measure $\textup{Leb} \otimes \rho$ (see for example \cite[Section 6.1]{PeszatZabczykSPDE}).   For $p > 1$ consider the space $X = L^1(U) + L^p(U)$ with the norm 
     $$ \norm{ \phi }_{L^1(U) + L^p(U)} = \inf \{ \norm{ f }_{L^1(U)} + \norm{ g }_{L^p(U)}: \phi=f+g, \, f \in L^1(U), \, g \in L^p(U) \}, $$
    which is a separable Banach space with dual $X^* = L^{\infty}(U) \cap L^q(U)$, with norm $\|\! \cdot \! \|_{L^{\infty}(U) \cap L^q(U)} = \max\{ \|\! \cdot \! \|_{L^{\infty}(U)} , \|\! \cdot \! \|_{L^q(U)} \}$, where $q$ is the conjugate index of $p$ (see \cite[Chapter 4]{BennettSarpley:1988}).  Let $\beta \in L^2(U)$, and for $t > 0$, $A \in \mathcal{A}$ and $\varphi \in X^*$ define
    $$
    M(t, A)(\varphi) = \int_{0}^{t}\!\! \int_A \varphi(u)\beta(u) \tilde{N}(dr, du).
    $$
    By the square integrability of $\beta$ and the boundedness of $\varphi$, the previous integral is well defined.  It is not difficult to see that $M$ is a cylindrical orthogonal martingale-valued measure on $X^*$.  Its family of intensity measures (and real quadratic variation) is given by (see \cite[{Section 7.3}]{PeszatZabczykSPDE})
    \begin{equation}\label{eqIntensMeasurExaIntegPoissonRM}
    \nu_{\varphi}((0,t] \times A) = \Exp[|M(t, A)(\varphi)|^2] = t \| \varphi \beta \|^2_{L^2(A)}.     
    \end{equation}
    Note that, for every $\varphi$ in the unit sphere of $X^*$ we have
    \begin{align*}
     \nu_{\varphi}((0,t] \times A) & \leq t \| \varphi \|_{L^{\infty}(A)}^2 \| \beta \|_{L^{2}(A)}^2    \leq t \| \varphi \|_{L^{\infty}(A) \cap L^q (A)}^2 \| \beta \|_{L^{2}(A)}^2 \\
     & \leq t \| \beta \|_{L^{2}(A)}^2 = t \int_A |\beta(u)|^2 \, \rho(du) \\
     & = \textup{Leb}\otimes \eta(A),
    \end{align*}
    where $d\eta = |\beta|^2 \, d\rho $ is a finite measure on $\mathcal{B}(U)$. By the usual extension process, $\nu_{\varphi} \leq \textup{Leb} \otimes \eta$ for all $\varphi$ in the unit sphere of $X^*$, therefore there is a quadratic variation 
    $$
    \operQuadraVari{M} = \sup_{\norm{\varphi}=1} \nu_{\varphi} = Leb\otimes \mu \leq Leb\otimes \eta,
    $$
    where
    $$
    \mu := \sup_{\norm{\varphi}=1} \mu_{\varphi} \leq \eta,\qquad d\mu_{\varphi} = \abs{\varphi}^2\abs{\beta}^2 d\rho.
    $$
    Now let $A\in\calA$ such that $0<\rho(A)\leq 1$. Taking $\varphi := \caract_{A}$ we have $\norm{\varphi} = \max\{ 1, \rho(A)^{1/q} \} = 1$ and therefore
    $$
    \mu(A) \geq \mu_{\varphi}(A) = \eta(A).
    $$
    Since every set on $\calA$ can be partitioned as a countable number of sets $A_n$ with $\rho(A_n)\leq 1$, this implies $\mu = \eta$ and therefore 
    $$
    \operQuadraVari{M} = Leb\otimes \eta.
    $$ 
\end{example}

\begin{remark}\label{remaNotExistQuaVariaCyliPoissonIntegral}
    We could try to extend the example on Section $7.2$ of \cite{PeszatZabczykSPDE}, by defining a cylindrical Poisson random measure
    $$
    M(t,A)(h) = \int_0^t\int_A\int_B h(u)\beta \tilde{N}(dr,du,d\beta), \quad h\in L^2(U).
    $$
    However, in that case we would obtain
    $$
    \nu_h(A) = \theta \int_A h^2(u)du,\quad \theta = \int_B \beta^2 d\rho(\beta)
    $$
    and we can show, similarly to the Example \ref{examCounterExampleCMVMStochasticIntegral}, that the quadratic variation does not exist as a finite measure. The trouble comes from the fact that $h$ can take arbitrarily large values on sets of small measure, giving rise to a large number of disjoint sets $C$ for which
    $$
    \sup_{\norm{h}_{L^2(U)}=1} \nu_h(C) \geq 1.
    $$
    The boundedness of $h$ in Example \ref{exampleNonLinearCylinPRM} allows us to overcome this obstacle.
\end{remark}

\section{Stochastic integration}\label{sectStochasticIntegration}

In this section we use our recently developed theory of quadratic variation in order to develop a theory of Hilbert-space valued stochastic integration for operator-valued processes with respect to cylindrical martingale-valued measures.  

\subsection{Construction of the stochastic integral}\label{subsecConstrIntegral}

We consider Hilbert spaces $H$, $G$. We denote by $\mathcal{HS}(H,G)$  the space of Hilbert-Schmidt operators from $H$ to $G$. \medskip

From now on we assume that $M$ is a cylindrical orthogonal martingale-valued measure on $H$ (Definition \ref{defiCylindricalOrthoMartinValuMeasure}) with an unique quadratic variation $\operQuadraVari{M}$, and satisfying the conditions of Theorem \ref{theoExistenCovariaOperatorQ}.  Let $Q_M: \Omega \times [0, T] \times U \to \mathcal{L}(H,H)$ be the corresponding operator-valued process given in \eqref{existenceofQ}. For $0\leq s < t\leq T$ and $A\in \calA$ it is convenient to denote
$$
M\left( (s,t],A \right) = M(t,A)-M(s,A)
$$

\begin{definition}
We define the Dol\'eans measure $\mu_M$ associated to $M$ on  $(\Omega\times [0,T]\times U,\mathcal{P}_{T}\otimes \mathcal{B}(U)) $ by
$$
\mu_M(C) = \mathbb{E} \int_{[0,T]\times U} \caract_{C}\, d\operQuadraVari{M}, \quad C \in \mathcal{P}_{T}\otimes \mathcal{B}(U). 
$$
\end{definition}

\begin{lemma}\label{lemmEqualIntegralDoleans}
Let $g: \Omega \times [0,T] \times U \rightarrow [0,\infty)$ be predictable. Then, 
\begin{equation}\label{eqEqualIntegralDoleans}
\int_{\Omega \times [0,T] \times U} g(\omega, s,u) \,  d \mu_{M}(\omega, s,u)
= \mathbb{E} \left[ \int_{[0,T]\times U} g(\cdot, s,u) \,  \operQuadraVari{M}(ds, du) \right].    
\end{equation}
\end{lemma}
\begin{proof}
Since $g$ is non-negative, it can be expressed as the pointwise increasing limit of a sequence of bounded predictable processes $(g^{n}: n \in \N)$. Then, if \eqref{eqEqualIntegralDoleans} holds for each $g^n$, by an application of the monotone convergence theorem \eqref{eqEqualIntegralDoleans} holds for $g$.  Likewise, each non-negative bounded predictable process $h$ is a limit of a sequence of simple predictable processes $(f_{n}: n \in \N)$, where the corresponding indicators are over measurable rectangles. If for each $f_{n}$ \eqref{eqEqualIntegralDoleans} holds, by an application of bounded convergence it holds for $h$. By linearity of the integrals the proof is reduced to check that \eqref{eqEqualIntegralDoleans}  holds for the elementary predictable process $e(\omega,s,u) =  \caract_{(s,t]}(s)\caract_{F}(\omega) \caract_{A}(u)$.  In fact, 
$$ \int_{\Omega \times [0,T] \times U} e(\omega, s,u)  \, d \mu_{M} = \mu_{M}(F \times (s,t] \times A) 
= \mathbb{E} \left[ \int_{[0,T]\times U} e(\cdot, s,u) \, d \operQuadraVari{M} \right].    
$$
\end{proof}

\begin{lemma}\label{lemmaSecondMomeCMVMDoleans}
Let $\mu_M$ be the Dol\'eans measure associated to $M$. 
For $s<t$, $F\in \calF_s$, $A\in \mathcal{A}$ and $h\in H$ we have
$$
\Exp\left[ \caract_F \left| M((s,t],A)(h) \right|^2  \right] =
\int_{F\times (s,t]\times A} \norm{Q_M^{1/2}(\omega,r,u)h}_H^2  d\mu_M.
$$
\end{lemma}
\begin{proof}
Since $F\in\calF_s$ and the fact that $\left| M((s,t],A)(h) \right|^2$ has the same conditional expectation with respect to $\calF_s$ as $\Gamma_M((s,t]\times A)(h)$, we have
\begin{eqnarray*}
\Exp\left[ \caract_F \left| M((s,t],A)(h) \right|^2  \right] 
& = & \int_F  \langle \Gamma_M ((s,t]\times A)h,h \rangle\, d\Prob \\
& = & \int_{F\times (s,t]\times A} \left( Q_M (\omega, r,u)h, h \right)_{H} \, d\mu_M,
\end{eqnarray*}
where we have used Theorem \ref{theoExistenCovariaOperatorQ} and Lemma \ref{lemmEqualIntegralDoleans}.
\end{proof}

Before we introduce our class of integrands, let us fix $(\omega,t,u)\in \Omega\times [0,T]\times U$ and denote $Q_{M} = Q_M(\omega,t,u)$. Following \cite[Section 2.3.2]{Liu-Rockner:2015} we consider
$$
H_{Q_{M}} = Q_{M}^{1/2}(H),\quad \left( h,g \right)_{H_{Q_{M}}} = \left( Q_{M}^{-1/2} h, Q_{M}^{-1/2} g \right)_H
$$
where $Q_{M}^{-1/2}$ is the pseudo inverse of $Q_{M}^{1/2}$. Being isometric to $( (\text{Ker}\,(Q_{M}^{1/2}))^\perp,\left( \cdot,\cdot \right)_{H})$, the Hilbert space $(H_{Q_{M}},\left( \cdot,\cdot \right)_{H_{Q_{M}}})$ is separable. In fact, given a complete orthonormal system $\{ h_k \}$ on $(\text{Ker}\,(Q_{M}^{1/2}))^\perp$, we readily obtain a complete orthonormal system $\{ Q_{M}^{1/2} h_k\}$ on $H_{Q_{M}}$. Since that system can be suplemented to a complete orthonormal system on $H$ by elements of Ker$(Q_{M}^{1/2})$, for any $L \in \mathcal{HS}(H_{Q_{M}},G)$ we get
$$
\norm{L}_{\mathcal{HS}(H_{Q_{M}},G)} = \norm{L\circ Q_{M}^{1/2}}_{\mathcal{HS}(H,G)}.
$$

We are ready to introduce our main class of stochastic integrands. 

\begin{definition}
\label{defiNewIntegrand}
Denote by $\Lambda^2(M,T)$ the space of families of operators $\Phi(\omega,t,u),$ indexed by $(\omega,t,u)\in \Omega\times [0,T]\times U$, such that:
\begin{enumerate}
\item For each $(\omega,t,u) \in \Omega\times [0,T]\times U$, $\Phi(\omega,t,u)\in \mathcal{HS}(H_{Q_{M}},G)$.
\item For every $h \in H$, the $G$-valued process $\Phi \circ Q^{1/2}_{M}(h)$ is predictable, i.e. 
the mapping 
$$ (\omega,t,u) \mapsto \Phi(\omega,t,u) \circ Q^{1/2}_{M}(\omega,t,u)(h),$$
is $\mathcal{P}_{T}\otimes \mathcal{B}(U)/\mathcal{B}(G)$-measurable. 
\item $\norm{\Phi}^2_{\Lambda^2(M,T)}$ is finite, where this quantity is defined by
\begin{equation}
\label{NewIntegrands}
\norm{\Phi}^2_{\Lambda^2(M,T)} \defeq
\int_{\Omega\times [0,T]\times U} \| \Phi(\omega,t,u)\circ Q^{1/2}_M(\omega,t,u)\|^2_{\mathcal{HS}(H,G)} \, d\mu_M.
\end{equation}
\end{enumerate}
\end{definition}
When we need to emphasize the spaces $H$ and $G$, we denote $\Lambda^2(M,T)$ by $\Lambda^2(M,T;H,G)$. \medskip

It is easy to check that $\Lambda^2(M,T)$ is a linear space and that $\norm{\cdot}_{\Lambda^2(M,T)}$ defines a Hilbertian seminorm on it. 
When working with $\Phi \in \Lambda^2(M,T)$, we can think of its restriction to $Q_M^{1/2}(H)$ instead, which lives in $\mathcal{HS} \left( Q_M^{1/2}(H), G\right)$ for any $(\omega,t,u)$. Equivalently, we consider $\Phi,\Psi\in \Lambda^2$ as equal when $\Phi\circ Q_M^{1/2} = \Psi\circ Q_M^{1/2} \,\, \mu_M$-a.e. which is equivalent to $\norm{\Phi - \Psi}_{\Lambda^2(M,T)} =0$. 

By Lemma \ref{lemmEqualIntegralDoleans} it is clear that 
$$
\norm{\Phi}^2_{\Lambda^2(M,T)} = 
\mathbb{E} \int_{[0,T]\times U} \| \Phi(s,u)\circ Q_M^{1/2}(s,u)\|^2_{\mathcal{HS}(H,G)} \operQuadraVari{M}(ds,du).
$$

\begin{remark}
It is not hard to verify that
\begin{equation}
\label{trace norm}
\| \Phi(\omega, s,u)\circ Q_M^{1/2}(\omega, s,u)\|^2_{\mathcal{HS}(H,G)}={\rm trace}\, \Bigl(\Phi(\omega, s,u) \circ Q_M(\omega, s,u)\circ \Phi(\omega, s,u)^{*} \Bigr).
\end{equation}
Moreover, the space $(\Lambda^2(M,T), \norm{\cdot}_{\Lambda^2(M,T)})$ is complete. This can be proven by following the same kind of arguments as in \cite{MetivierPellaumail} (Section 14.4, p.168, 169).
\end{remark}

\begin{definition}
We define the class $\mathcal{S}=\mathcal{S}(M,T)$ of simple functions, as those that can be expressed in the form
\begin{equation}
\label{simple}
\Phi(\omega,s,u) = \sum_{i=1}^{n} \sum_{j=1}^{m} \caract_{(s_i,t_i]}(s)\caract_{F_i}(\omega) \caract_{A_j}(u) S_{ij}
\end{equation}
where $m,n,\in\N$, $0\leq s_i \leq t_i$, $F_i\in \mathcal{F}_{s_i}$, $A_j\in \mathcal{A}$ and $S_{ij}\in \mathcal{HS}(H,G)$.\medskip
\end{definition}

We can always assume (taking a smaller partition if needed) that for $i\neq k$, either $(s_i,t_i]\cap (s_k,t_k]
= \emptyset$ or $(s_i,t_i] = (s_k,t_k]$ and $F_i\cap F_k = \emptyset$. \medskip

For $\Phi\in \mathcal{S}$ given by (\ref{simple}), we define its stochastic integral $I(\Phi)$ as the $G$-valued process given by 
\begin{equation}
\label{NewDefIntSimpleIntegrand}
I_t(\Phi) : = \int_0^t\! \! \int_U \Phi(s,u) M(ds,du) = 
\sum_{i=1}^{n} \sum_{j=1}^{m} \caract_{F_i}(\omega) Y_{i,j}(t)
\end{equation}
where $Y_{i,j}$ is the square integrable martingale taking values in $G$ and satisfying
\begin{equation}
\label{eqDefRadonProcessY}
\left( Y_{i,j}(t),g \right)_{G} = M((s_i \land t, t_i\land t], A_j) (S_{ij}^*g)
\end{equation}
according to Theorem 2.1 of \cite{AlvaradoFonseca:2021}. It is worth noting that for each pair $(\omega,t)$, the double sum in the definition reduces to at most one term. 


We are ready to prove the It\^{o} isometry for simple integrands.
\begin{theorem} 
\label{theoItoIsometrySimpleIntegrands}
For every $\Phi \in \mathcal{S}(M,T)$, $(I_{t}(\Phi): t \in [0,T]) \in \mathcal{M}_{T}^{2}(G)$ and for each $t \in [0,T]$, we have $\Exp (I_{t}(\Phi))=0$ and also       
\begin{equation}
\label{eqItoIsometrySimpleIntegrands}
\Exp \left[ \norm{I_{t}(\Phi)}^{2} \right] = \int_{\Omega \times [0,t] \times U} \norm{ \Phi(\omega, r,u)\circ Q^{1/2}_M(\omega,r,u)}_{\mathcal{HS}(H,G)}^{2} d\mu_M.
\end{equation}
In particular, the map $I: \mathcal{S}(M,T) \rightarrow \mathcal{M}^{2}_{T}(G)$, given by $\Phi \mapsto (I_{t}(\Phi): t \in [0,T])$ is linear continuous. 
\end{theorem}

\begin{proof}
Let $\Phi \in \mathcal{S}(M,T)$ be of the form \eqref{simple}. For simplicity, within the proof we denote $I_t=I_t(\Phi)$. The fact that  $(I_{t}: t \in [0,T]) \in \mathcal{M}_{T}^{2}(G)$ is a direct consequence of the definition and the fact that each $(Y_{i,j}(t): t \in [0,T])$ is in  $\mathcal{M}_{T}^{2}(G)$.

To prove \eqref{eqItoIsometrySimpleIntegrands}, let $(g_{k}: k \in \N)$ be a complete orthonormal system in $G$. From \eqref{eqDefRadonProcessY} and 
 the definition given by \eqref{NewDefIntSimpleIntegrand} we have 
\begin{eqnarray*}
\Exp \left[ \norm{I_{t}}^{2} \right] & = & \Exp \sum_{k=1}^{\infty} \abs{ \left(I_{t},g_{k}\right)_{G}}^{2} = \Exp \sum_{k=1}^{\infty} \abs{ \sum_{i=1}^{n} \sum_{j=1}^{m} \mathbbm{1}_{F_{i}} \left( Y_{i,j}(t), g_{k} \right)_{G} }^{2} \\
& = & \sum_{k=1}^{\infty} \Exp  \abs{ \sum_{i=1}^{n} \sum_{j=1}^{m} \mathbbm{1}_{F_{i}}  M((s_{i} \wedge t, t_{i} \wedge t],A_{j})(S_{i,j}^{*} g_{k}) }^{2} \\
& = & \sum_{k=1}^{\infty}  \sum_{i=1}^{n} \sum_{j=1}^{m} \Exp \left[ \caract_{F_i} \abs{ M((s_{i} \wedge t, t_{i} \wedge t],A_{j})(S_{i,j}^{*}g_{k}) }^{2} \right] 
\end{eqnarray*}
where we used the fact that all cross terms are null. Using Lemma \ref{lemmaSecondMomeCMVMDoleans}  we get
\begin{eqnarray*}
\Exp \left[ \norm{I_{t}}^{2} \right] 
& = & \sum_{i=1}^{n} \sum_{j=1}^{m}
\int_{F_i\times (s_i\land t,t_i\land t]\times A_j}  \sum_{k=1}^{\infty} \norm{Q_M^{1/2}(\omega, r,u)S^*_{ij}(g_k)}_H^2 d\mu_M \\
& = & \sum_{i=1}^{n} \sum_{j=1}^{m}
\int_{F_i\times (s_i\land t,t_i\land t]\times A_j} \norm{Q_M^{1/2}(\omega,r,u)S^*_{ij}}_{\mathcal{HS}(G,H)}^2 d\mu_M \\
& = & \sum_{i=1}^{n} \sum_{j=1}^{m}
\int_{F_i\times (s_i\land t,t_i\land t]\times A_j} \norm{S_{ij}Q_M^{1/2}(\omega, r,u)}_{\mathcal{HS}(H,G)}^2 d\mu_M \\
& = & \int_{\Omega \times [0,t] \times U} \norm{\Phi(\omega, r,u)Q_M^{1/2}(\omega, r,u)}_{\mathcal{HS}(H,G)}^2 d\mu_M.
\end{eqnarray*}  

Linearity of $I: \mathcal{S}(M,T) \rightarrow \mathcal{M}^{2}(G)$ follows easily from the definition. To prove continuity, observe that from Doob's inequality,  \eqref{NewIntegrands} and  \eqref{eqItoIsometrySimpleIntegrands}, for each $\Phi \in \mathcal{S}(M,T)$ we have
$$
\norm{I(\Phi)}^{2}_{\mathcal{M}^{2}_{T}(G)} = \Exp \left[ \sup_{t\in [0,T]} \norm{I_{t}(\Phi)}^{2} \right] \leq  4 T \, \Exp \left[ \norm{I_{T}(\Phi)}^{2} \right]= 4T \, \norm{\Phi}^{2}_{\Lambda^{2}(M,T)}. $$
\end{proof}

For $\Phi \in \Lambda^2(M,T)$ the above argument gives
$$
\norm{\Phi}_{\Lambda^2(M,T)}^2 = \int_{\Omega\times [0,T]\times U}\norm{\Phi(\omega,r,u)}_{\mathcal{HS}(H_{Q_M},G)}^2 d\mu_M
$$
but we must keep in mind that $H_{Q_M}$ varies with $(\omega,r,u)$. \medskip


\begin{theorem}
	\label{Density of S}
The space $\mathcal{S}(M,T)$ is dense in $\Lambda^2(M,T)$.
\end{theorem}

\begin{proof}
Let $C(M,T)$ be the collection of all maps $\Psi(\omega,r,u)$ taking the simple form
\begin{equation} 
\label{simpleFamiliesProofDenseSupspaceIntegrands}
\Psi(\omega,r,u)= \ind{(s,t]}{r} \ind{F}{\omega} \ind{A}{u} S, \quad \omega \in \Omega, r \in [0,T], \, u \in U,
\end{equation}  
where $0 \leq s < t \leq T$, $F \in \mathcal{F}_{s}$, $A \in \mathcal{A}$ and $S \in \mathcal{HS}(H,G)$.  It is enough to show that if $\Phi \in C^{\perp}$, then $\Phi\circ Q_M^{1/2} = 0\,\, \mu_M$-a.e.
Let $\Phi \in C^{\perp}$, which means
$$
\int_{F\times (s,t]\times A} \left( \Phi \circ Q_M^{1/2},S\circ Q_M^{1/2} \right)_{\mathcal{HS}(H,G)} d\mu_M = 0, $$
for any $S\in \mathcal{HS}(H,G)$ and any cylinder set $F\times (s,t]\times A \subseteq \Omega \times [0,T] \times U$, with $F\in \mathcal{F}_s$.

Let $(h_{n}: n \in \N)$ and $(g_{k}: k \in \N)$ be complete orthonormal systems on $H$ and $G$ respectively. Working inside the integral, let $f_n = Q_M^{1/2}(h_n)$. The collection $(f_{n} \otimes g_{k})_{n,k \in \N}$ is a complete orthonormal system in $\mathcal{HS}(H_{Q_M},G)$. We restrict $\Phi$ to $H_{Q_M}$ and write out the definition of the $\mathcal{HS}(H_{Q_M},G)-$norm using the system $\{ f_n\}$. We obtain
$$
 \left\langle \Phi, f_n \otimes g_k \right\rangle_{\mathcal{HS}(H_{Q_M},G)} =  \left( \Phi(f_n),g_k \right)_{G} = \left( \Phi\circ Q_M^{1/2} (h_n),g_k \right)_{G}
$$
and consequently
$$
\int_{F\times (s,t]\times A} \left( \Phi\circ Q_M^{1/2} (h_n),g_k \right)_{G} d\mu_M = 0
$$
for $k,n\in\N$ and any predictable cylinder $F\times (s,t]\times A$. Since these cylinders generate the $\sigma-$\'algebra $\mathcal{P}_T\otimes \mathcal{B}(U)$, it follows that $\Phi \circ Q_M^{1/2} = 0,\,\, \mu_M\,-$ a.e.
\end{proof}

The proof of the following theorem is now straightforward.

\begin{theorem}\label{theoItoIsometry}
    The map $:\mathcal{S}(M,T) \fle \mathcal{M}_T^2(G)$ has a continuous linear extension 
    $$
    I:\Lambda^2(M,T) \fle \mathcal{M}_T^2(G)
    $$ 
    such that, for each $\Phi \in \Lambda^2(M,T)$, $t\in [0,T]$, $\Exp[I_t(\Phi)] = 0$ and
    \begin{equation}
    \label{eqIsometryGeneralPhi}
    \Exp\norm{I_t(\Phi)}^2 = \int_{\Omega\times [0,t]\times U} \norm{\Phi(\omega,r,u) \circ Q_M^{1/2}(\omega,r,u)}_{\mathcal{HS}(H,G)}^2 d\mu_M.
    \end{equation}
\end{theorem}

\begin{definition}
    We call the map $I$ defined in last theorem, the \textit{stochastic integral mapping} and, for each $\Phi\in \Lambda^2(M,T)$, we call $I(\Phi)$ the \textit{stochastic integral} of $\Phi$. We will also denote the process $I(\Phi)$ by
    $$
    I(\Phi) = \left\{ \int_0^t\!\! \int_U \Phi(r,u)M(dr,du),t\in [0,T] \right\}.
    $$
\end{definition}

\begin{example}\label{exampleStochIntegHilbertLevyProcess}
Let $M$ be the $H$-valued orthogonal martingale-valued measure of Example \ref{exampleLevyHValuedMVM}. By \eqref{eqQuadraticVariationLevyCMVM} and \eqref{eqOperatorQMLevyCMVM}  the square integrability condition \eqref{NewIntegrands} takes the form 
$$ \norm{\Phi}^2_{\Lambda^2(M,T)} = \Exp \int_{0}^{T}  \| \Phi(s,0)\circ Q^{1/2}\|^2_{\mathcal{HS}(H,G)} \, ds + \Exp \int_{0}^{T}\!\! \int_{U} \| \Phi(s,u) u \|^2_{G} \, \lambda(du) ds< \infty, 
$$
We can give an explicit description of the stochastic integral with respect to $M$ for an integrand $\Phi\in \Lambda^2(M,T)$. 

Assume first that $\Phi$ is a simple integrand of the form \eqref{simple}, then by \eqref{eqLevyItoDecomp}, \eqref{NewDefIntSimpleIntegrand} and \eqref{eqDefRadonProcessY} we have $\Prob$-a.e. for every $g \in G$, 
\begin{flalign*}
& \left( \int_0^t\! \! \int_U \Phi(s,u) \, M(ds,du), g \right)_{G} \\
& =  \sum_{i=1}^{n} \sum_{j=1}^{m} \caract_{F_i}(\omega) 
\left( S_{ij} W_{t_i\land t}(\omega) - S_{ij} W_{s_i\land t}(\omega), g \right)_{G}  \delta_{0}(A_j)  \\
& \hspace{15pt} +   \sum_{i=1}^{n} \sum_{j=1}^{m} \caract_{F_i}(\omega)  \int_{s_i\land t}^{t_i\land t} \int_{A_{j} \setminus \{0\}}  \left( S_{ij} u , g \right)_{G}\,  \tilde{N}(ds,du)
\\
 & = \left( \int_0^t  \Phi(s,0) \, dW_s + \int_{0}^{t}\!\!\int_{U \setminus \{0\}} \Phi(s,u)u\, \tilde{N}(ds, du), g \right)_{G}.
\end{flalign*}
By a limit argument and using the independence of $W$ and $N$ we conclude that for    
$\Phi\in \Lambda^2(M,T)$ we have 
\begin{equation*}
\int_0^t\!\! \int_U \Phi(r,u)\, M(dr,du)   = \int_0^t  \Phi(s,0) \, dW_s + \int_{0}^{t}\!\!\int_{U \setminus \{0\}} \Phi(s,u)u\, \tilde{N}(ds, du),  
\end{equation*}
where the stochastic integrals with respect to the Wiener process $W$ and with respect to the compensated Poisson random measure $\tilde{N}$ are as defined for instance in Chapter 8 in \cite{PeszatZabczykSPDE}.  
\end{example}

\begin{example}\label{exampleStochIntegLinearCylinPoissonRM}
Let $\xi \neq 0$ be an element of the separable Hilbert space $H$. Let $M$ be the cylindrical orthogonal martingale-valued measure of Example \ref{exampleLinearCPoissonRM} for $X=H$. 
Observe that by \eqref{eqQuadraVariaExamCylPoissonRM} and \eqref{operatorQMExampleCylinPoissonRM}, the square integrability condition \eqref{NewIntegrands} takes the form 
$$ \norm{\Phi}^2_{\Lambda^2(M,T)} = \Exp \int_{0}^{T}\!\!  \int_{U} \norm{\Phi(s,u)\xi}^{2}_{G} \, ds \lambda(du) < \infty. $$
To give a explicit form for the stochastic integral of an integrand $\Phi\in \Lambda^2(M,T)$, as before, for a simple integrand $\Phi$ of the form \eqref{simple}, and by \eqref{eqDefiPoissonRMPointProcess}, \eqref{NewDefIntSimpleIntegrand} and \eqref{eqDefRadonProcessY} we have $\Prob$-a.e. for every $g \in G$, 
\begin{flalign*}
& \left( \int_0^t\! \! \int_U \Phi(s,u) \, M(ds,du), g \right)_{G} \\
& =  \sum_{i=1}^{n} \sum_{j=1}^{m} \caract_{F_i}(\omega) 
\tilde{N}(\omega ,( s_i \land t, t_i\land t] \times  A_j) \left( \xi, S_{ij}^*g \right)_{H} \\
& = \sum_{0 \leq s \leq t}  \sum_{i=1}^{n} \sum_{j=1}^{m} \caract_{F_i}(\omega) \caract_{( s_i, t_i]}(s) \caract_{A_{i}}(p(\omega,s)) \left( S_{ij} \xi, g \right)_{G} \\
& \hspace{15pt} -  \sum_{i=1}^{n} \sum_{j=1}^{m} \caract_{F_i}(\omega)    (t_{i}-s_{i})\lambda(A_{i}) \left( S_{ij} \xi, g \right)_{G} \\
& = \left( \sum_{0 \leq s \leq t} \Phi(\omega, s,p(\omega,s))\xi - \int_{0}^{t}\!\! \int_{U} \Phi(\omega, s,u)\xi \, ds \lambda(du), g \right)_{G} .
\end{flalign*}
From the above calculation we can conclude that if   
$\Phi\in \Lambda^2(M,T)$, then we have 
\begin{equation*}
\int_0^t\!\! \int_U \Phi(r,u)\, M(dr,du)   = \sum_{0 \leq s \leq t} \Phi(\omega, s,p(\omega,s))\xi - \int_{0}^{t}\!\!  \int_{U} \Phi(\omega, s,u)\xi \, ds \lambda(du).  
\end{equation*}
where the first term in the right-hand side is a random finite sum and the second term is a Bochner integral.  
\end{example}

\begin{remark}
    One can try to establish an upper bound for an It\^o isomorphism.  In our setting, that would correspond to $L^p$ inequalities of the form
    \begin{equation}
    \label{eqIsomorphism}
    \Exp\norm{I_t(\Phi)}^p \leq C_{p} \int_{\Omega\times [0,t]\times U} \norm{\Phi \circ Q_M^{1/2}}_{\mathcal{HS}(H,G)}^p d\mu_M.
    \end{equation}
    In \cite[Section 8.8]{PeszatZabczykSPDE}, this kind of inequality is proved for an $L^p$-valued stochastic integral (same index as the one in the inequality above), in the cases of a Poisson random measure and a Wiener process; in \cite{Dirksen:2014},  there are upper $L^p$ bounds for $L^q$-valued integrals with respect to a compensated Poisson measure.  This kind of results, usually require inequalities of Burkholder-Davies-Gundy type, which are a consequence of an It\^o's formula for the stochastic integral.  A version of It\^o's formula is already proved by the authors in the Hilbert-valued case, but it is carried over to \cite{CCFM:Ito}, in order to maintain the present paper at a reasonable length.  However, in order to obtain these $L^p$ bounds, it is necessary to develop this theory in the Banach space-valued setting.
    
\end{remark}


\subsection{Properties of the stochastic integral}\label{subsecPropeIntegral}

\begin{proposition}
If for each $A\in\calA$ and $h\in H$ the real-valued stochastic process $M(\cdot,A)(h)$ is continuous, then for each $\Phi\in \Lambda^2(M,T)$ the stochastic integral $I(\Phi)$ is a continuous process.
\end{proposition}
\begin{proof}
For a simple process $\Phi \in \mathcal{S}(M,T)$, the result follows directly from \eqref{NewDefIntSimpleIntegrand} and the fact that the radonified version $Y_{i,j}$ defined by \eqref{eqDefRadonProcessY} is a continuous process if each $(M(t,A)(\phi): t \geq 0)$ is continuous. Since the sequential limit of $G$-valued continuous processes is continuous, the result now extends by the denseness of $\mathcal{S}(M,T)$ in $\Lambda^{2}(M,T)$ and the continuity of the stochastic integral mapping to every $\Phi \in \Lambda^{2}(M,T)$. 
\end{proof}

\begin{proposition}\label{propMappingIntegContOpera} Let $E$ be a separable Hilbert space and let  $R \in \mathcal{L}(G,E)$. Then, for each $\Phi \in \Lambda^{2}(M,T;H,G)$, we have $R \circ \Phi=\{ R \circ \Phi(r,\omega,u): r \in [0,T],  \omega \in \Omega, u \in U\} \in \Lambda^{2}(M,T;H,E)$, moreover we have $\Prob$-a.e. for every $t \in [0,T]$, 
\begin{equation} \label{eqIntUnderContMapping}
\int^{t}_{0}\!\! \int_{U} R \circ \Phi (r,u) M (dr, du)= R \left(\int^{t}_{0}\!\! \int_{U} \Phi (r,u) M (dr, du) \right). 
\end{equation} 
\end{proposition}
\begin{proof} One can easily  check that $R \circ \Phi \in \Lambda^{2}(M,T;H,E)$. By continuity of the stochastic integral, it suffices to check that  \eqref{eqIntUnderContMapping} holds for $\Phi$ of the form \eqref{simple}. 

In fact, observe $R \circ \Phi \in 
 \mathcal{S}(M,T;H,E)$, for 
$$ R \circ \Phi(r,\omega,u)=  \sum_{i=1}^{n} \sum_{j=1}^{m}  \ind{]s_{i}, t_{i}]}{r} \ind{F_{i}}{\omega} \ind{A_{j}}{u} R \circ S_{i,j}. $$
Moreover, $\Prob$-a.e. for all $t \in [0,T]$ and $e \in E$, we have 
\begin{eqnarray*}
\left( \int^{t}_{0}\!\! \int_{B} R \circ \Phi (r,u) M (dr, du),e \right)_{E} 
& = & \sum_{i=1}^{n} \sum_{j=1}^{m} \mathbbm{1}_{F_{i}} 
M((s_{i} \wedge t, t_{i} \wedge t],A_{j})(S_{i,j}^{*}\circ R^{*} e) 
\\
& = & \sum_{i=1}^{n} \sum_{j=1}^{m} \mathbbm{1}_{F_{i}} \left(Y_{i,j}(t),R^{*}e \right)_{G}
\\
& = & \left( R \left(\int^{t}_{0} \!\!\int_{U} \Phi (r,u) M (dr, du) \right), e \right)_{E}.
\end{eqnarray*}
But since $E$ is a separable Hilbert space, the above equality shows that \eqref{eqIntUnderContMapping} holds for $\Phi \in \mathcal{S}(M,T;H,G)$.  
\end{proof}

\begin{proposition} \label{propIntegralInSubintervalAndRandomSubset}
Let $0 \leq s_{0} < t_{0} \leq T$ and $F_{0} \in \mathcal{F}_{s_{0}}$. Then, for every $\Phi \in \Lambda^{2}(M,T)$, $\Prob$-a.e. we have  $\forall \, t \in [0,T]$,
\begin{multline} \label{eqIntegralInSubintervalAndRandomSubset}
\int_{0}^{t}\!\! \int_{U} \, \mathbbm{1}_{]s_{0},t_{0}]\times F_{0}} \Phi(r,u) M(dr,du) \\
= \mathbbm{1}_{F_{0}} \left( \int_{0}^{t \wedge t_{0}}\!\! \int_{U} \, \Phi(r,u) M(dr,du) -\int_{0}^{t \wedge s_{0}}\!\!  \int_{U} \, \Phi(r,u) M(dr,du) \right). 
\end{multline}
\end{proposition}
\begin{proof} First, by using similar ideas to those used in the proof of Proposition \ref{propMappingIntegContOpera} one can easily check that \eqref{eqIntegralInSubintervalAndRandomSubset} holds for $\Phi \in \mathcal{C}(M,T)$ of the  form \eqref{simpleFamiliesProofDenseSupspaceIntegrands}. Since $\mathcal{C}(M,T)$ spans $\mathcal{S}(M,T)$, the linearity of the integral shows that \eqref{eqIntegralInSubintervalAndRandomSubset} is valid for $\Phi \in \mathcal{S}(M,T)$. Finally, that \eqref{eqIntegralInSubintervalAndRandomSubset} holds for $\Phi \in \Lambda^{2}(M,T)$ can be show using the density of $\mathcal{S}(M,T)$  and the continuity of the stochastic integral. 
\end{proof}
 
\begin{proposition} \label{propStoppedIntegral}
Let $\Phi \in \Lambda^{2}(M,T)$ and $\sigma$ be a $(\mathcal{F}_{t})$-stopping time such that $\Prob (\sigma \leq T)=1$. Then, $\Prob$-a.e. for every $t \in [0,T]$,
\begin{equation} \label{eqStoppedIntegral}
\int_{0}^{t}\!\! \int_{U} \, \ind{[0,\sigma]}{r} \Phi(r,u) M(dr,du) = \int_{0}^{t \wedge \sigma}\!\! \int_{U} \,  \Phi(r,u) M(dr,du).
\end{equation}
\end{proposition}


\begin{proof}
We follow the steps of the proof of Lemma 2.3.9 in \cite{Liu-Rockner:2015}, simplified by the fact that our proposition \ref{propIntegralInSubintervalAndRandomSubset} allows us to do Step 1 for general $\Phi$. \medskip

\textbf{Step 1.} Consider a simple stopping time 
$$
\sigma = \sum_{k=1}^n a_k \caract_{\Omega_k},
$$
where $a_k \in [0,T]$ and $\Omega_k \in \calF_{a_k}$ for each $k$. By linearity of $I_t$ and Proposition \ref{propIntegralInSubintervalAndRandomSubset}, for $\Phi\in \Lambda^2(M,T)$ we have
$$
I_t\left( \caract_{(\sigma,T]}\Phi \right) = I_t\left(  \sum_{k=1}^n \caract_{\Omega_k\times (a_k,T]} \Phi  \right) = \sum_{k=1}^n \caract_{\Omega_k} (I_t(\Phi) - I_{t\land a_k} (\Phi)) = I_t(\Phi) - I_{t\land \sigma} (\Phi).
$$
Consequently
$$
I_t\left( \caract_{[0,\sigma]}\Phi \right) = I_t(\Phi) - I_t\left( \caract_{(\sigma,T]}\Phi \right) = I_{\sigma\land t} (\Phi).
$$
\textbf{Step 2.} For a given stopping time $\sigma \leq T$, let $\sigma_n$ be a sequence of simple stopping times such that $T\geq \sigma_n \downarrow \sigma$. Because of the continuity on the right of $I_t$ we have $\mathbb{P}-$a.e.
$$
I_{\sigma_n \land t}(\Phi) \fle I_{\sigma \land t}(\Phi).
$$
On the other hand, by monotone convergence
$$
\norm{\caract_{(0,\sigma_n]}\Phi - \caract_{(0,\sigma]}\Phi}_{\Lambda^2(M,T)}^2
= \int_{\Omega\times [0,t]\times U} \caract_{(\sigma,\sigma_n]}(\cdot) \norm{\Phi\circ Q_M^{1/2}}^2_{\mathcal{HS}(H,G)} d\mu_M \fle 0  .
$$
Thanks to the isometry \eqref{eqIsometryGeneralPhi} we get
$$
\Exp \norm{I_t(\caract_{(0,\sigma_n]}\Phi) - I_t(\caract_{(0,\sigma]}\Phi)}^2 \fle 0
$$
for all $t\in [0,T]$. Using step 1, through a suitable subsequence we finally get
$$
I_{\sigma \land t} (\Phi) = \lim_{n\fle\infty} I_{\sigma_n \land t} (\Phi) = \lim_{n\fle\infty} I_t( \caract_{(0,\sigma_n]} \Phi) =  I_t( \caract_{(0,\sigma]} \Phi).
$$
\end{proof}

\subsection{Extension to locally integrable functions}\label{subsecExtensIntegral}

\begin{definition}
\label{defiLocallyIntegrableIntegrand}
Denote by $\Lambda^2_{\textup{loc}}(M,T)$ the space of families of operators $\Phi(\omega,t,u),$ indexed by $(\omega,t,u)\in \Omega\times [0,T]\times U$, satisfying (i) and (ii) from Definition \ref{defiNewIntegrand}, and such that:
\begin{equation}\label{LocallyIntegrable}
\Prob \left( \left\{ \omega \in \Omega : 
\int_{[0,T]\times U} \| \Phi(\omega,s,u)\circ Q_M^{1/2}(\omega,s,u)\|^2_{\mathcal{HS}(H,G)} \, d\operQuadraVari{M} < \infty \right\}\right) = 1
\end{equation}
\end{definition}

We equip the space $\Lambda^2_{\textup{loc}}(M,T)$ with the complete, metrizable, linear topology generated by the following neighborhood basis 
$$\left\{ \Phi \in \Lambda^2_{\textup{loc}}(M,T):  \Prob \left( \left\{ \omega: 
\int_{[0,T]\times U} \| \Phi(\omega,s,u)\circ Q_M^{1/2}(\omega,s,u)\|^2_{\mathcal{HS}(H,G)} \, d\operQuadraVari{M} > \epsilon \right\}\right)  \leq \delta \right\},$$
for $\epsilon, \delta>0$. 

Let $\Phi \in \Lambda^2_{\textup{loc}}(M,T)$, for each $n \in \N$ we define the stopping time 
\begin{equation}\label{StopTimeLocalIntegrable}
    \tau_n(\omega) : = \inf \left\{  t \in [0,T] : \int_{[0,T]\times U} \| \Phi(\omega,s,u)\circ Q_M^{1/2}(\omega,s,u)\|^2_{\mathcal{HS}(H,G)} \, d\operQuadraVari{M} \geq n  \right\}
\end{equation}
Note that the sets in the previous equation are decreasingly nested, then $(\tau_n : n \in \N)$ is an increasing sequence.  Also, by \eqref{LocallyIntegrable}, $\tau_n \to T$  $\Prob$-a.e. Since this sequence is $\calF_t$-adapted, for each $n \in \N$, the function $\Phi^{\tau_n} = \Phi \caract_{[0, \tau_n]}$ belongs to $\Lambda^2(M,T)$.  Thus, we can extend the definition of the integral to $\Phi \in \Lambda^2_{\textup{loc}}(M,T)$ by setting 
\begin{equation}
    I(\Phi) : = I(\Phi^{\tau_n}), \ \mbox{if}\  \int_{[0,T]\times U} \| \Phi(\omega,s,u)\circ Q_M^{1/2}(\omega,s,u)\|^2_{\mathcal{HS}(H,G)} \, d\operQuadraVari{M} \leq n.
\end{equation}
Proposition \ref{propStoppedIntegral} (applied to $\sigma = \tau_n \wedge \tau_m$ for $m < n$) guarantees that this definition is consistent. 

By construction $I(\Phi) \in \mathcal{M}_T^{2,loc}(G)$, where  $\mathcal{M}_T^{2,loc}(G)$ denotes the linear space of all the $G$-valued locally  zero-mean square integrable martingales, which we consider equipped with the topology of uniform convergence in probability on compact intervals of time. One can easily check that the stochastic integral mapping $I:\Lambda^2_{\textup{loc}}(M,T) \rightarrow \mathcal{M}_T^{2,loc}(G) $ is linear. Its continuity is a consequence of the following result. 

\begin{proposition}
Assume $\Phi \in \Lambda^2_{\textup{loc}}(M,T)$. Then, for $a, b >0$ we have
\begin{equation*}\label{eqInequaContinuityUCPStocIntegral}
\Prob \left( \sup_{t \in [0,T]} \norm{I_{t}(\Phi)} > a \right) \leq \frac{b}{a^{2}} +  \Prob \left(  
\int_{[0,T]\times U} \| \Phi(\omega,s,u)\circ Q_M^{1/2}(\omega,s,u)\|^2_{\mathcal{HS}(H,G)} \, d\operQuadraVari{M} >b \right) .   
\end{equation*}
\end{proposition}
\begin{proof}
Define 
$$ \tau_{b}(\omega)=\inf \left\{  t \in [0,T]: \int_{[0,T]\times U} \| \Phi(\omega,s,u)\circ Q_M^{1/2}(\omega,s,u)\|^2_{\mathcal{HS}(H,G)} \, d\operQuadraVari{M} > b \right\}$$
Then,  by Doob's inequality and \eqref{eqIsometryGeneralPhi} we have 
\begin{flalign*}
& \Prob \left( \sup_{t \in [0,T]} \norm{I_{t}(\Phi)} > a \right) \\
&  =
 \Prob \left(  \sup_{t \in[0,T] } \norm{I_{t}( \caract_{[0, \tau_b]}\Phi  ) } >a \mbox{ and }
\int_{[0,T]\times U} \| \Phi(s,u)\circ Q_M^{1/2}(s,u)\|^2_{\mathcal{HS}(H,G)} \, d\operQuadraVari{M} \leq  b \right) \\
& + \Prob \left( \sup_{t \in[0,T] } \norm{I_{t}( \caract_{[0, \tau_b]}\Phi  ) } >a \mbox{ and } \int_{[0,T]\times U} \| \Phi(s,u)\circ Q_M^{1/2}(s,u)\|^2_{\mathcal{HS}(H,G)} \, d\operQuadraVari{M} > b \right) \\
& \leq  
 \frac{1}{a^{2}} \Exp  \left(   \norm{I_{T}(\caract_{[0, \tau_b]} \Phi)}^{2} \right) + \Prob \left( \int_{[0,T]\times U} \| \Phi(s,u)\circ Q_M^{1/2}(s,u)\|^2_{\mathcal{HS}(H,G)} \, d\operQuadraVari{M} > b \right) \\
& = 
 \frac{1}{a^{2}}   
\int_{\Omega \times [0,T]\times U} \| \caract_{[0, \tau_b]} \Phi \circ Q_M^{1/2}\|^2_{\mathcal{HS}(H,G)} \, d\mu_{M}  \\
& +
\Prob \left( \int_{[0,T]\times U} \| \Phi(s,u)\circ Q_M^{1/2}(s,u)\|^2_{\mathcal{HS}(H,G)} \, d\operQuadraVari{M} > b  \right).
\end{flalign*}
Therefore the result follows. 
\end{proof}

\begin{corollary}
The stochastic integral mapping $I:\Lambda^2_{\textup{loc}}(M,T) \rightarrow  \mathcal{M}_T^{2,loc}(G)$ is continuous.     
\end{corollary}
\smallskip

\subsection{The stochastic Fubini theorem}\label{subsecStochFubini}

In this section, we prove a stochastic version of Fubini's theorem involving the integral constructed in the previous section. We start by describing the class of integrands for which the theorem is valid.

\begin{definition} \label{Fubini integrands}
Let $(E, \mathscr{E}, \varrho)$ be a $\sigma$-finite measure space. We denote by $\Lambda^{1,2}(M,T,E)$ the space of families of operators $\Phi(\omega, r, u,e)$, indexed by  $(\omega,t,u,e)\in \Omega\times [0,T]\times U \times E $, such that:
\begin{enumerate}
\item For each $(\omega,t,u,e) \in \Omega \times [0,T] \times U \times E$,  one has that
\begin{enumerate}
	\item $\Phi(\omega,t,u,e): {\rm Dom}\,(\Phi(\omega,t,u,e)) \to G$ is linear. 
	\item ${\rm Dom}\,(\Phi(\omega,t,u,e)) \supseteq H_{Q_{M}}$.
	\item $\Phi(\omega, t, u,e) \Bigl |_{H_{Q_{M}}} \Bigr. \in \mathcal{HS} (H_{Q_{M}},G)$.
\end{enumerate}
\item For all $h \in H$, the map $$(\omega, t,u,e) \mapsto \langle \Phi(\omega, t,u,e) \circ Q^{1/2}_M(\omega, t,u)(h),g \rangle,$$ is $\mathcal{P}_{T}\otimes \mathcal{B}(U)\otimes \mathscr{E} / \calB(G)$-measurable. 
\item $\| \Phi  \|_{\Lambda^{1,2}(M,T,E)}:=\displaystyle \int_E \| \Phi(\cdot, \cdot, \cdot, e)\|_{\Lambda^2(M,T)}  \, \varrho(de) < \infty  $.
\end{enumerate}

It is straightforward to verify that $\left(\Lambda^{1,2}(M,T,E), \| \cdot \|_{\Lambda^{1,2}(M,T,E)} \right)$ is a Banach space. We will denote by $\Lambda^{2,2}(M,T,E)$ the subspace of all families of operators $ \Phi(\omega, r,u,e)$  in $\Lambda^{1,2}(M,T,E)$ satisfying
$$
\| \Phi  \|_{\Lambda^{2,2}(M,T,E)}^2:=\displaystyle \int_E \| \Phi(\cdot, \cdot, \cdot, e)\|^2_{\Lambda^2(M,T)}  \, \varrho(de) < \infty.
$$
It is a Hilbert space with the norm $\| \cdot \|_{\Lambda^{2,2}(M,T,E)}$.\end{definition}

		

\begin{lemma}
\label{approximation lemma}
Let $\Phi \in \Lambda^{1,2}(M,T,E)$. There exists a sequence $(\Phi_n)_{n\geq 1}$ in $\Lambda^{2,2}(M,T,E)$ such that $\varrho$-a.e $\|  \Phi_n(\cdot, \cdot, \cdot, e) \|_{\Lambda^{2}(M,T)} \leq \|  \Phi_{n+1}(\cdot, \cdot, \cdot, e) \|_{\Lambda^{2}(M,T)}$, $\forall n \in \N$, and
$$
\lim_{n \to \infty} \| \Phi-\Phi_n\|_{\Lambda^{1,2}(M,T,E)}=0.
$$
\end{lemma}
		
\begin{proof}
From Definition \ref{Fubini integrands}, there exists $E_0 \subset E$ with $\varrho (E \backslash E_0) = 0$ such that for all $e \in E_0$, $\| \Phi(\cdot, \cdot, \cdot, e)  \|_{\Lambda^2(M,T)}<\infty$. Since $(E, \mathscr{E},\varrho)$ is a $\sigma$-finite measure space, we can take an increasing sequence $(E_n)$ on  $\mathscr{E}$ such that $E_0 =\bigcup_{n \in \N}E_n$ and $\varrho(E_n)<\infty$, $\forall n \in \N$.  Let $ \Phi_n(\omega, r,u,e)$ be the family of bounded random variables defined by:
\begin{eqnarray*}
\Phi_n (\omega, r,u,e) &=&\frac{n \Phi(\omega, r,u,e)}{\| \Phi(\cdot, \cdot, \cdot, e)  \|}_{\Lambda^2(M,T)} \mathbbm{1}_{ \left\{e \in E_n: \|  \Phi(\cdot, \cdot, \cdot, e)\|_{\Lambda^2(M,T)}>n \right \} }(e)\\
& & + \Phi(\omega, r,u,e) \mathbbm{1}_{ \left\{e \in E_n: \|  \Phi(\cdot, \cdot, \cdot, e)\|_{\Lambda^2(M,T)}\leq n \right \} }(e).
\end{eqnarray*}
Now, the rest of the proof is very similar to Lemma 4.22 in \cite{FonsecaMora:2018}.
\end{proof}

A proof of the following result can be carried out using similar arguments to those in the proof of Theorem \ref{Density of S}.
		
\begin{lemma}
Let $S(M,T,E)$ denote the collection of all families $\Phi(\omega, r,u,e)$, indexed by  $(\omega,t,u,e)\in \Omega\times [0,T]\times U \times E $, of $\mathcal{HS}(H,G)$-valued maps of the form:
\begin{equation}\label{simple process density lemma}
\Phi(\omega, r,u,e)=\sum_{l=1}^p \sum_{i=1}^n \sum_{j=1}^m \mathbbm{1}_{F_j}(\omega)\mathbbm{1}_{]s_j, t_j]}(r) \mathbbm{1}_{A_i}(u)\mathbbm{1}_{D_l}(e)S_{i,j},
\end{equation}
for all $\omega \in \Omega$, $r \in [0,T]$, $u \in U$, $e \in E$, where $m,n,p \in \N$, and for $l=1, \ldots p$, $i=1, \ldots, n$, $j=1, \ldots, m$, $0 \leq s_j<t_j \leq T$, $F_j \in \mathcal{F}_{s_j}$, $A_i \in \mathcal{A}$, $D_l \in \mathscr{E}$ and $S_{i,j} \in \mathcal{HS}(H,G)$. Then $S(M,T,E)$ is dense en $\Lambda^{2,2}(M,T,E)$.
\end{lemma}

\begin{theorem}[{\bf Stochastic Fubini's Theorem}]
\label{Stochastic Fubini}
Let $\Phi \in \Lambda^{1,2}(M,T,E)$. Then
\begin{enumerate}
	\item For a.e. $(\omega, r,u) \in \Omega \times [0,T]\times U$, the mapping $E \ni e \mapsto \Phi(\omega, r,u,e) \in \mathcal{HS} \left(H_{Q},G\right)$ is Bochner integrable. Moreover, 
	$$
	\int_E \Phi(\cdot, \cdot, \cdot, e) \varrho(de) =\left\{ \int_E \Phi(\omega, r, u,e) \varrho(de): \omega \in \Omega, r \in [0, T], u \in U \right\} \in \Lambda^2(M,T).
	$$
	\item The mapping $E \ni e \mapsto \displaystyle \int_0^{\cdot}\!\! \int_U \Phi(r,u,e)\, M(dr,du) \in \mathcal{M}_T^2(G)$ is Bochner integrable. Furthermore, for all $t \in [0,T]$
	\begin{equation}\label{Martingale equality}
	\left(\int_E \left( \int_0^{\cdot}\!\! \int_U \Phi(r,u,e)\, M(dr,du)   \right) \varrho(de) \right)_t=\int_E \left(  \int_0^t \!\! \int_U \Phi(r,u,e)\, M(dr,du)\right)\varrho(de).
	\end{equation}
	\item The following equality holds $\Prob$-a.e., for all $t \in [0, T]$
    \begin{equation}\label{Fubini equality}
    \int_0^t \!\! \int_U \left( \int_E \Phi(\cdot, \cdot,e) \varrho(de) \right) M(dr, du)=\int_E  \left(\int_0^t \!\! \int_U \Phi(r, u,e) M(dr, du) \right) \varrho (de).
    \end{equation}
\end{enumerate}
\end{theorem}

\begin{proof}
	(i) By the the real Fubini theorem, we have that for each $h\in H$, $g \in G$, the map
		$$
		(\omega, s, u) \mapsto \int_E \left( \Phi(\omega, s, u,e) \circ Q_M^{1/2}(\omega, s,u)(h),g  \right)_{G} \, \varrho (de)
		$$
		is $\mathcal{P}_T \otimes \mathcal{B}(U)$-measurable. Since for all $h \in H$ and $g \in G$,
		\begin{multline*}
		\left(  \left( \int_E \Phi(\omega, s, u,e)\varrho(de)  \right) \circ Q_M^{1/2} (\omega, s,u)(h),g  \right)_{G} \\ 
        = \int_E \left(  \Phi(\omega, s, u,e)  \circ Q_M^{1/2} (\omega, s,u)(h),g  \right)_{G} \varrho(de),
		\end{multline*}
		using the Pettis measurability theorem (e.g. Theorem 1.1.6 of \cite{Hytonen:2016}) we conclude that
		$$
		(\omega, s, u)  \mapsto \left( \int_E \Phi(\omega, s, u, e) \varrho (de)  \right) \circ Q_M^{1/2}(\omega, s, u)(h)
		$$
		is $\mathcal{P}(T) \otimes \mathcal{B}(U)/\mathcal{B}(G)$-measurable for all $h \in H$.
		
		On the other hand, from Minkowski's integral inequality it follows that
		\begin{eqnarray}
		\label{Minkowski inequality}
		& &\int_{\Omega \times [0,T] \times U} \left(\int_E \|\Phi(\omega, s, u,e) \|_{\mathcal{HS} \left(H_{Q_{M}},G \right)} \, \varrho (de)\right)^2 \mu_M(d \omega, ds, du) \nonumber \\
		&=&\int_{\Omega \times [0,T] \times U} \left(\int_E \|\Phi(\omega, s, u,e)\circ Q_M^{1/2}(\omega, s,u) \|_{\mathcal{HS} \left(H,G \right)}\, \varrho (de)\right)^2 \mu_M(d \omega, ds, du)\nonumber\\\
		&\leq & \left(\int_E \left(\int_{\Omega \times [0,T] \times U}\| \Phi(\omega,s,u,e) \circ Q_M^{1/2}(\omega, s, u) \|^2_{\mathcal{HS}(H,G)} \mu_M(d \omega, ds, du)  \right)^{1/2}\, 
		\varrho (de) \right)^2\nonumber \\
		&=& \left(\int_E \|\Phi(\cdot, \cdot, \cdot, e) \|_{\Lambda^2 (M,T)}\, \varrho(de) \right)^2<\infty.
		\end{eqnarray}
		By hypothesis, for all $(\omega, s, u)$, $e \mapsto \Phi(\omega, s,u,e)$ is $\mathscr{E}/\mathcal{B}\left(\mathcal{HS}(H_{Q}/G) \right)$-measurable. Moreover, from (\ref{Minkowski inequality}), 
	$$
	\mu_M-a.e. \quad  \int_E \|\Phi(\omega, s, u,e) \|_{\mathcal{HS} \left(H_{Q_{M}},G \right)}\, \varrho(de)<\infty.
	$$
		Hence, the Bochner integral $\displaystyle \int_E \Phi(\omega, s,u,e)\, \varrho(de)$ belongs to $\mathcal{HS} \left(H_{Q_{M}}, G\right)$, $\mu_M$-a.e.

  (ii) First,  note that for fixed $e$ the stochastic integral $\displaystyle \int_0^{\cdot}\!\! \int_U \Phi(s,u,e)M(ds,du)$ is an element of $\mathcal{M}^2_T(G)$. Moreover, with the help of Lemma \ref{approximation lemma}, we see  that there exists a sequence $(\Psi_k)_{k \in \N}$ of simple processes such that
  \begin{equation}
  \label{approximation result}
  \lim_{k \to \infty}  \| \Phi-\Psi_k\|_{\Lambda^{1,2}(M,T,E)}=0.
  \end{equation}
  Every $\Phi_k$ is of the form (\ref{simple process density lemma}), i.e.
  $$
  \Psi_k(\omega, r,u,e)=\sum_{l=1}^p \sum_{i=1}^n \sum_{j=1}^m \mathbbm{1}_{F_j}(\omega)\mathbbm{1}_{]s_j, t_j]}(r) \mathbbm{1}_{A_i}(u)\mathbbm{1}_{D_l}(e)S_{i,j},
  $$
  where we have omitted the dependence on $k$ of the components
of (\ref{simple process density lemma}) (in order to simplify the notation). In a similar way as in (\ref{NewDefIntSimpleIntegrand}) and (\ref{eqDefRadonProcessY}), its integral is given by
\begin{eqnarray*}
\displaystyle
\int_0^t \int_U\Psi_k(r,u,e)M(dr,du)
& = &   \sum_{l=1}^p\sum_{i=1}^n \sum_{j=1}^m \mathbbm{1}_{D_l}(e)\mathbbm{1}_{F_i}(\omega) Y_{i,j}(t) \\
& = & \sum_{l=1}^p \mathbbm{1}_{D_l}(e) \left[\sum_{i=1}^n \sum_{j=1}^m \mathbbm{1}_{F_i}(\omega) Y_{i,j}(t)\right],
\end{eqnarray*}
where the expression inside square brackets is an element of $\mathcal{M}_T^2(G)$. Hence, we can see this integral as a simple function from $(E,\mathscr{E})$ to $\Bigl(\mathcal{M}_T^2(G), \mathcal{B}(\mathcal{M}_T^2(G))\Bigr)$.
Moreover, by the linearity of the stochastic integral and with the help of the Doob's inequality and the It\^{o} isometry, we can conclude that
\begin{eqnarray}
\label{Doob Ito application}
&& \lim_{k \to \infty} \displaystyle \int_E  \left\| \int_0^{\cdot} \int_U \Bigl(\Phi(r,u,e)-\Psi_k(r,u,e)\Bigr)M(dr,du) \right \|_{\mathcal{M}^2_T(G)
}\varrho(de) \nonumber \\
& \leq& 2 \sqrt{T}\lim_{k \to \infty} \displaystyle \int_E \left\|\Phi(\cdot, \cdot, e)-\Psi_k(\cdot, \cdot, e) \right\|_{\Lambda^2(M,T)} \varrho(de) \nonumber\\
&=&2 \sqrt{T} \lim_{k \to \infty} \| \Phi-\Psi_k  \|_{\Lambda^{1,2}(M,T,E)}.
\end{eqnarray}
Then, it follows from (\ref{approximation result}), (\ref{Doob Ito application}), the Chebyshev inequality and a routine use of Borel-Cantelli lemma that there exists an set $E_2 \subseteq E$ with $\varrho(E-E_2)=0$ and a subsequence $(\Psi_{k_q})_q$ such that
$$\lim_{q \to \infty} \left\| \int_0^{\cdot} \int_U \Bigl( \Phi(r,u,e)-\Psi_{k_q}(r,u,e) \Bigr)M(dr,du)  \right\|_{\mathcal{M}_T^2(G)}=0, \quad \forall\, e \in E_2.$$
In particular, this implies that the map
$\displaystyle e \mapsto \int_0^{\cdot}\!\! \int_U \Phi(r,u,e)M(dr,du)$
is strongly measurable. Moreover, following a similar computation to that in (\ref{Doob Ito application}), we have 
$$
\int_E \left\| \int_0^{\cdot} \int_U \Phi(r,u,e) M(dr, du)  \right\|_{\mathcal{M}^2_T(G)} \varrho(de) \leq 2 \sqrt{T} \| \Phi\|_{\Lambda^2(M,T,E)}<\infty,
$$
which implies that the map $e \mapsto \displaystyle \int_0^{\cdot} \int_U \Phi(r,u,e)M(dr,du)$ is Bochner integrable (in the sense of $\mathcal{M}^2_T (G)$). Furthermore, it is straightforward to conclude that 
$$\left\| \int_E \left( \int_0^{\cdot} \int_U (\Phi(r,u,e)-\Psi_{k_q}(r,u,e))M(dr,du)  \right) \varrho (de) \right\|_{\mathcal{M}_T^2(G)} \to 0,$$
as $q \to \infty$ and since $\Psi_{k_q}$ satisfies in a natural way (\ref{Martingale equality}), we can infer that any $\Phi \in \Lambda^{1,2}(M,T,E)$ also holds (\ref{Martingale equality}).

(iii) Let $(\Psi_{k_q})_q$
be the sequence of simple families as defined in the proof
of (ii). At first, note that any $\Psi_{k_q}$ satisfies (\ref{Fubini equality}) by its simple form. Now, by the linearity of both the stochastic integral and the Bochner integral and with the help of the Doob's inequality and the It\^{o} isometry, it follows that
\begin{flalign*}
& \rho_{k_q} = \\ 
& \left \|\int_0^{\cdot} \int_U \left(  \int_E \Phi(\cdot, \cdot, \cdot, e) \varrho(de)\right) M(dr,du)-\int_0^{\cdot} \int_U \left( \int_E \Psi_{k_q} (\cdot, \cdot, \cdot, e) \varrho (de)\right) M(dr,du)   \right\|_{\mathcal{M}_T^2 (G)}    
\end{flalign*}
satisfies
\begin{equation}
\label{Doob Ito application1}
\rho_{k_q} \leq 2 \sqrt{T}  \left \| \int_E (\Phi(\cdot, \cdot, \cdot, e)-\Psi_{k_q} (\cdot, \cdot, \cdot, e))\varrho(de) \right\|_{\Lambda^{2}(M,T)} \leq 2 \sqrt{T} \| \Phi-\Psi_{k_q} \|_{\Lambda^{1,2}(M,T,E)}
\end{equation}
so $\rho_{k_q} \fle 0$ as $q \to \infty$. Now, by (\ref{Doob Ito application}), (\ref{Doob Ito application1}) and the uniqueness of limits, one can conclude that the processes
$$
\int_0^{\cdot} \int_U \left( \int_E \Phi(\cdot, \cdot, \cdot, e) \varrho(de) \right) M(dr,du),\qquad  \int_E\left( \int_0^{\cdot} \int_U \Phi(\cdot, \cdot, \cdot, e)M(dr,du) \right) \varrho(de) 
$$ 
are equal as elements of $\mathcal{M}_T^2(G)$, which finishes the proof of (\ref{Fubini equality}).
\end{proof}

\section{Stochastic partial differential equations}\label{SectSPDEs}

In this section we study the following class stochastic evolution equations:
\begin{equation}\label{EqSPDE}
  dX_t = \left[ AX_t + B(t,X_t)\right]dt + \int_U F(t,u,X_t) M(dt,du).  
\end{equation}
where we will assume the following: 
\begin{enumerate}
    \item $A$ is the infinitesimal  generator for a $G$-valued $C_{0}$-semigroup $(S(t): t \geq 0)$. 
    \item $B: [0,T] \times G \to G$ is $\calB (\R_{+}) \otimes \calB(G)/\calB(G)$-measurable. 
    \item $F$ is a family of operators  $(F(t, u,g): t\geq 0, u \in U, g \in G)$ such that  
    \begin{enumerate}
        \item $\forall \omega \in \Omega, t\geq 0, u \in U, g \in G$, $F(t, u,g) \in \mathcal{HS}(H_{Q_{M}}, G)$.
        \item For every $h \in H$, the mapping $(\omega, t,u,g) \mapsto F(t,u,g)\circ Q_M^{1/2}(\omega, t,u)(h)$ is $\calP_{T} \otimes   \calB(U)  \otimes \calB(G)/\calB(G)$-measurable. 
    \end{enumerate}
    \item $M$ is a cylindrical orthogonal martingale-valued measure on $H$ (Definition \ref{defiCylindricalOrthoMartinValuMeasure}) with a unique quadratic variation $\operQuadraVari{M}$, and satisfying the conditions of Theorem \ref{theoExistenCovariaOperatorQ}.
\end{enumerate}

In this article we will study existence and uniqueness of the following types of solutions:

\begin{definition}
A $G$-valued predictable process $X=(X_{t}: t \in [0,T])$ is a \emph{weak solution} for equation \eqref{EqSPDE} if for all $g\in \mbox{Dom}(A^*)$  and $t \geq 0$, $\Prob$-a.e.
\begin{multline}\label{eqDefiWeakSolution}
    \left(X_t,g\right)_G  = \left(X_0,g\right)_G + \int_0^t \left[ \left(X_s,A^* g\right)_G  
    + \left(B(s,X_s),g\right)_G \right] ds \\
    + \int_0^t\! \! \int_U \left(F(s,u,X_s),g\right)_G M(ds,du),
\end{multline}
where the first integral in the right-hand side of \eqref{eqDefiWeakSolution} is defined $\Prob$-a.e. as a Lebesgue integral, and the second integral in the right-hand side of \eqref{eqDefiWeakSolution} is the (real-valued) stochastic integral of the family $\{ \left(F(s,u,X_{s}(\omega),g\right)_G:  \omega \in \Omega, s \in [0,t], u \in U\}$.
\end{definition}

\begin{definition}
A $G$-valued predictable process $X=(X_{t}: t \geq 0)$ is a  \emph{mild solution} for equation \eqref{EqSPDE}  if for all $t  \geq 0$, $\Prob$-a.e.
\begin{equation}\label{EqStochIntDiffEq}
  X_t = S(t)X_0 + \int_0^t S(t-s)B(s,X_s)ds + \int_0^t\! \! \int_U S(t-s) F(s,u,X_s)M(ds,du),  
\end{equation}
where the first integral in the right-hand side of \eqref{EqStochIntDiffEq} is defined $\Prob$-a.e. as a Bochner integral, and the second integral in the right-hand side of \eqref{EqStochIntDiffEq} is the stochastic integral of the family $\{ \mathbbm{1}_{[0,t]}(s) S(t-s)F(t,u,X_{s}(\omega)):  \omega \in \Omega, s \in [0,t], u \in U\}$.
\end{definition}

\begin{remark}
 It is well known that there are constants $\kappa>0$ and $N\geq 1$ such that 
\begin{equation}\label{EqSemigroupbound}
  \norm{S(t)g} \leq Ne^{\kappa t} \norm{g},\quad g\in G.  
\end{equation}
Also, we have that
\begin{equation}\label{EqDerivativeSemigroup}
\frac{d}{dt}S(t) = A S(t), \qquad
\frac{d}{dt}S(t)^* = S(t)^*A^*.
\end{equation}
\end{remark}

The stochastic integral in the right-hand side of \eqref{EqStochIntDiffEq} is often known as the \emph{stochastic convolution}. In the following result we show the stochastic convolution possesses a predictable version.

\begin{theorem}\label{theoPredictableStochConvolu}
    For $\Phi \in \Lambda^2(M,T)$, the process $\hat{\Phi}$ defined as
    $$
    \hat{\Phi}_t := \int_0^t\! \! \int_U S(t-s) \Phi(s,u)M(ds,du)
    $$ 
    is $L^2$-continuous, so stochastically continuous. In particular, it has a predictable version.
\end{theorem}

\begin{proof}
For $r<t$ we write $\hat{\Phi}_t - \hat{\Phi}_s$ as
$$
\int_0^r\! \! \int_U [S(t-s)-S(r-s)]\Phi(s,u)M(ds,du) + \int_r^t \!\! \int_U S(t-s) \Phi(s,u)M(ds,du).
$$
This allows us to get a bound $\Exp \norm{\hat{\Phi}_t - \hat{\Phi}_r}_G^2 \leq \sigma_1 + \sigma_2$, where
\begin{align*}
\sigma_1 & = 2 \Exp \norm{\int_0^r\! \! \int_U [S(t-s)-S(r-s)]\Phi(s,u)M(ds,du)}_G^2 \\
  & = 2\Exp \int_0^r\! \! \int_U \norm{ [S(t-s)-S(r-s)]\Phi(s,u)\circ Q_M^{1/2}(s,u) }_{\mathcal{HS}(H,G)}^2 \operQuadraVari{M}(ds,du); \\
\sigma_2 & = 2 \Exp \norm{\int_r^t \!\! \int_U S(t-s)\Phi(s,u)M(ds,du)}_G^2 \\
  & = 2\Exp \int_r^t \!\!  \int_U \norm{S(t-s)\Phi(s,u)\circ Q_M^{1/2}(s,u) }_{\mathcal{HS}(H,G)}^2 \operQuadraVari{M}(ds,du).
\end{align*}
The integrand on the second expression of $\sigma_1$ goes to $0$ as $r\uparrow t$ and if we use (\ref{EqSemigroupbound}) we can see that it is bounded by 
$$
4N^2e^{2\kappa T} \norm{\Phi(\omega,s,u)\circ Q_M^{1/2}(\omega,s,u)}_{\mathcal{HS}(H,G)}^2, 
$$
    which is $\mu_M$-integrable because $\Phi \in \Lambda^2(M,T)$, so by dominated convergence $\sigma_1 \fle 0$ as $r\uparrow t$. Similarly, using (\ref{EqSemigroupbound}), it is straightforward that
$$
\sigma_2 \leq 2N^2 e^{2\kappa T} \Exp \int_r^t \!\! \int_U \norm{\Phi(s,u) \circ Q_M^{1/2}(s,u) }_{\mathcal{HS}(H,G)}^2 \operQuadraVari{M}(ds,du)
\fle 0
$$
as $r\uparrow t$. For $r\downarrow t$, there are minor changes. Finally, since $(\hat{\Phi}_t)_{0 \leq t \leq T}$ is adapted and stochastically continuous it has a predictable version  (see Proposition 3.21 of \cite{PeszatZabczykSPDE}).
\end{proof}

\subsection{Equivalence between weak and mild solutions}\label{subsectWeakMildSoluc}

In the following result we introduce sufficient conditions for the existence and equivalence of weak and mild solutions to \eqref{EqSPDE}. 

\begin{theorem}\label{theoEquivaWeakMildSol}
Let $X=(X_{t}: t \geq 0)$ be a $G$-valued predictable process, and assume that for every $T>0$, $X$, $B$ and $F$ satisfy: 
\begin{align}
& \Prob \left( \int_{0}^{T} \, \norm{X_{s}}_{G} +  \norm{B(s,X_{s})}_{G} \, ds<\infty \right)=1. \label{RealIntegrability}   \\
& \int_{\Omega\times [0,T]\times U} \| F(s,u,X_{s}(\omega))\circ Q_M^{1/2}(\omega,s,u)\|^2_{\mathcal{HS}(H,G)} \, d\mu_M < \infty. \label{StochIntegrability}  
\end{align}
Then, $X$ is a weak solution to \eqref{EqSPDE} if and only if it is a mild solution to \eqref{EqSPDE}.
\end{theorem}

\begin{proof} Let $X$ be a weak solution. To prove that it is a mild solution, it is enough to show that for each $g\in G$, we have
$$\left(X_t,g\right)_G = 
\left( S(t)X_0 + \int_0^t S(t-s)B(s,X_s)ds + \int_0^t\! \! \int_U S(t-s) F(s,u,X_s)M(ds,du),g\right)_G.
$$
By the definition of weak solution for $X_r$, applied to $S(t-r)^*A^*g \in G$, and integrating in $r$ from $0$ to $t$, we obtain 
\begin{align*}
    \MoveEqLeft[3] \int_0^t\!\!\int_0^r\! \! \int_U \left( F(s,u,X_s),S(t-r)^*A^*g \right)_G M(ds,du) \, dr  \\
    =  {}&      \int_0^t (X_r - X_0,S(t-r)^*A^*g)_G \, ds \, dr - \int_0^t\!\! \int_0^r (X_s,A^*S(t-r)^*A^*g)_G \, ds \, dr \\
    &  - \int_0^t\!\! \int_0^r (B(s,X_s),S(t-r)^*A^*g)_G \, ds \, dr. \\  ={}&      \int_0^t (X_r,S(t-r)^*A^*g)_G \, dr - \int_0^t (X_0,S(t-r)^*A^*g)_G \, dr \\
    & - \int_0^t\!\! \int_s^t (X_s,A^*S(t-r)^*A^*g)_G \, dr \, ds - \int_0^t\!\! \int_s^t (B(s,X_s),S(t-r)^*A^*g)_G \, dr \, ds.
\end{align*}
In the last equality, condition \eqref{RealIntegrability}  and the exponential bound \eqref{EqSemigroupbound} allow us to use real valued Fubini.  By \eqref{EqDerivativeSemigroup} on the last three integrals, we obtain that the right hand side is equal to
\begin{align*}
&  \int_0^t (X_r,S(t-r)^*A^*g)_G \, dr + (X_0, g)_G - (X_0,S(t)^*g)_G  \\
    & + \int_0^t  (X_s,A^*g)_G \, ds - \int_0^t  (X_s,S(t-s)^*A^*g)_G \, ds \\
    &  + \int_0^t (B(s,X_s),g)_G \, ds - \int_0^t (B(s,X_s),S(t-s)^*g)_G  \, ds \\
    ={}& (X_0, g)_G - (X_0,S(t)^*g)_G + \int_0^t  (X_s,A^*g)_G \, ds \\
    &  + \int_0^t (B(s,X_s),g)_G \, ds - \int_0^t (B(s,X_s),S(t-s)^*g)_G  \, ds.
\end{align*}
By condition \eqref{StochIntegrability} and the exponential bound \eqref{EqSemigroupbound}, we can use stochastic Fubini on the left hand side, then another application of \eqref{EqDerivativeSemigroup} gives us
\begin{align*}
    \MoveEqLeft[3] \int_0^t\!\!\int_0^r\! \! \int_U \left( F(s,u,X_s),S(t-r)^*A^*g \right)_G M(ds,du) \, dr  \\
    ={}& \int_0^t\! \! \int_U \int_s^t\left( F(s,u,X_s),S(t-r)^*A^*g \right)_G \, dr \, M(ds,du) \\
    ={}& \int_0^t\! \! \int_U \left( F(s,u,X_s),S(t-s)^*g \right)_G \, M(ds,du) - \int_0^t\! \! \int_U \left( F(s,u,X_s),g \right)_G \, M(ds,du) \\
    ={}&  \int_0^t\! \! \int_U \left( F(s,u,X_s),S(t-s)^*g \right)_G \, M(ds,du) - \left(X_t,g\right)_G  + \left(X_0,g\right)_G \\ 
    &+ \int_0^t  \left(X_s,A^* g\right)_G \, ds     +  \int_0^t\left(B(s,X_s), g\right)_G  \, ds,
\end{align*}
where in the last equality, we use again the fact that $X$ is a weak solution.  Finally, comparing both sides and solving for $\left(X_t,g\right)_G$ we obtain the desired result.

Suppose now that $X$ is a mild solution.  Then, for $0\leq r \leq T$ we have
\begin{eqnarray*}
    (X_r, A^*g)_G & = & (X_0, S(r)^*A^*g)_G + \int_0^r (B(s,X_s), S(r-s)^*A^*g)_G \, ds \\
    & & + \int_0^r\! \! \int_U (F(s,u, X_s), S(r-s)^*A^*g)_G \, M(ds, du).
\end{eqnarray*}
We now integrate in $s$ from $0$ to $t \leq T$.  Similar to the previous computations, we use stochastic Fubini and \eqref{EqDerivativeSemigroup} to obtain
\begin{eqnarray*}
    \int_0^t (X_r, A^*g)_G \,dr & = & \int_0^t (X_0, S(r)^*A^*g)_G \,dr + \int_0^t \!\! \int_0^r (B(s,X_s), S(r-s)^*A^*g)_G \, ds \,dr \\
    & & + \int_0^t \!\! \int_0^r\! \! \int_U (F(s,u, X_s), S(r-s)^*A^*g)_G \, M(ds, du) \,dr \\
    & = &  (X_0, S(t)^*g)_G - (X_0, g)_G + \int_0^t (B(s,X_s), S(t-s)^*g)_G \, ds \\
    & & - \int_0^t (B(s,X_s), g)_G \, ds + \int_0^t\! \! \int_U (F(s,u, X_s), S(t-s)^*g)_G \, M(ds, du) \,dr \\
    & & - \int_0^t\! \! \int_U (F(s,u, X_s), g)_G \, M(ds, du) \,dr
    \\
    & = & (X_t, g)_G - (X_0, g)_G - \int_0^t (B(s,X_s), g)_G \, ds \\
    & & - \int_0^t\! \! \int_U (F(s,u, X_s), g)_G \, M(ds, du), 
\end{eqnarray*}
where in the last equality we use again the fact that $X$ is a mild solution.  Solving for $(X_t, g)_G$ shows that $X$ is a weak solution. 
\end{proof}

\begin{example}\label{Ex:LevynoiseSPDE}
Let $d \in \N$ and $\mathcal{O}$ be a non-empty open bounded subset of $\R^{d}$. Consider the following stochastic heat equation with additive L\'evy noise and Dirichlet boudary conditions:
$$
\begin{cases}
dy(t,x)=\Delta y(t,x) dt + \sigma(x) dL_{t}, & (t,x) \in [0,\infty) \times \mathcal{O}, \\ 
y(t,x)=0, & (t,x) \in [0,\infty) \times \partial \mathcal{O}, \\
y(0,x)=y_{0}(x), & x \in \mathcal{O}. 
\end{cases}
$$
where $y_{0}, \sigma \in L^{2}(\mathcal{O})$, and $L=(L_{t}: t \geq 0)$ is a real-valued mean-zero square integrable L\'{e}vy process. 

It is well-known that with the above boundary conditions $A=\Delta$ is the generator for a strongly continuous $C_{0}$-semigroup $(S(t): t \geq 0)$ on $L^{2}(\mathcal{O})$ with domain $\mbox{Dom}(\Delta)=W^{1,2}_{0}(\mathcal{O})$. 

Following ideas taken from Example 7.3 in \cite{RiedlevanGaans:2009}, we can reformulate the above stochastic heat equation in an infinite dimensional context for which we can apply our recently introduced theory. We do this as follows: let $L=(L_{t}: t \geq 0)$  be a cylindrical mean-zero square integrable L\'evy process in $\ell^{2}$. For each $i\in \N$, let $\sigma_{i} \in L^{2}(\mathcal{O})$ and consider a sequence of real-valued numbers $\alpha=(\alpha_{i})_{i \in \N}$ such that $(\alpha_{i} \norm{\sigma_{i}}_{L^{2}(\mathcal{O})}) \in \ell^{2}$. 

Define $F(\alpha): \ell^{2} \rightarrow L^{2}(\mathcal{O})$ by $F(\alpha)(h)=\sum_{i=1}^{\infty} \alpha_{i}h_{i} \sigma_{i}$ for $h=(h_{i})_{i\in \N} \in \ell^{2}$, which is Borel-measurable. Observe that if $(e^{k}:k\in \N)$ is the canonical orthonormal basis in $\ell^{2}$, then 
$$\norm{F(\alpha)}^{2}_{\mathcal{HS}(\ell^{2},L^{2}(\mathcal{O}))} =\sum_{k=1}^{\infty} \norm{F(\alpha)(e^{k})}_{L^{2}(\mathcal{O})}^{2}  
=\sum_{k=1}^{\infty} \norm{\sum_{i=1}^{\infty} \alpha_{i}e^{k}_{i} \sigma_{i}}_{L^{2}(\mathcal{O})}^{2}
= \sum_{i=1}^{\infty} \abs{\alpha_{i}}^{2} \norm{\sigma_{i}}_{L^{2}(\mathcal{O})}^{2}< \infty.$$
Therefore $F(\alpha)\in \mathcal{HS}(\ell^{2},L^{2}(\mathcal{O}))$. 

Taking $U=\{\alpha\}$, $\mathcal{A}=\{ \emptyset, U\}$ and $M=(M(t,A): t \geq 0, A \in \mathcal{A})$ given by $M(t,A)(h)=L_{t}(h) \delta_{\alpha}(A)$, we have by Example \ref{exampleCylinMartingaleDeltaLocBounCova} that $M$ is a  cylindrical martingale-valued measure with quadratic variation
$\operQuadraVari{M}(\omega)(ds,du)=  \norm{Q} ds \delta_{\alpha}(du)$, where $Q$ is the covariance operator of $L$, i.e. $Q:\ell^{2} \rightarrow \ell^{2}$ is such that $\left(Qh,g\right)_{\ell^{2}}=\Exp \left[ (L_{1}(h), L_{1}(g))_{\ell^{2}} \right] $ $\forall h,g \in \ell^{2}$. Moreover, 
$$ Q_{M}(\omega, r, u) = \frac{Q}{\norm{Q}} \mathbbm{1}_{\{\alpha\}}(u). $$
Given the above, we can reinterpret the stochastic heat equation with Dirichlet boundary conditions as the following stochastic evolution equation with values in $L^{2}(\mathcal{O})$:
\begin{equation}\label{exampleStocHeatEqua}
dY_{t}=AY_{t}dt+\int_{U} \, F(u) \, M(dt,du), \quad Y_{0}=y_{0}.
\end{equation}
Since 
$$\Exp \int_{0}^{T} \!\! \int_{U}  \, \norm{F(\alpha)\circ Q_{M}^{1/2}}^{2}_{\mathcal{HS}(\ell^{2},L^{2}(\mathcal{O}))} \, \operQuadraVari{M}(dt,du) \leq T \norm{Q}\sum_{i=1}^{\infty} \abs{\alpha_{i}}^{2} \norm{\sigma_{i}}_{L^{2}(\mathcal{O})}^{2}< \infty,$$
then it is clear that $F(\alpha) \in \Lambda^{2}(M,T)$. Hence, a mild solution for \eqref{exampleStocHeatEqua} exists:
\begin{equation}\label{eqMildSolutionExaStochHeatEqua}
Y_{t} = S(t)y_{0} + \int_{0}^{t} \!\!  \int_{U} \, S(t-s) F(u) \, M(ds,du).    
\end{equation}
Furthermore, we have
\begin{multline*}
 \Exp \left[ \int_{0}^{T} \,  \norm{Y_{t}}_{L^{2}(\mathcal{O})}^{2} \, dt \right]\\
 \leq N^{2} e^{2kT} \left[ \norm{y_{0}}_{L^{2}(\mathcal{O})}^{2} + \Exp \int_{0}^{T} \!\!  \int_{U}  \, \norm{F(\alpha)\circ Q_{M}^{1/2}}^{2}_{\mathcal{HS}(\ell^{2},L^{2}(\mathcal{O}))} \, \operQuadraVari{M}(dt,du)  \right] < \infty.
\end{multline*}
Therefore, by Theorem \ref{theoEquivaWeakMildSol}, the mild solution \eqref{eqMildSolutionExaStochHeatEqua} is also a weak solution for \eqref{exampleStocHeatEqua}. 
\end{example}


\subsection{Existence and uniqueness of solutions}\label{subsecExistUniqueSPDEs}

We assume the following growth and Lipschitz type conditions over $B$ and $F$:

\begin{assumption} \hfill
\begin{enumerate}
\item There is a constant $C_B > 0$ such that for all $t \in [0,T]$ and $g, h \in G$ we have
\begin{align}
\label{eqBoundBB}
    \| B(t, g)  \|_G & \leq C_B(1 + \|  g \|_G) \\
    \label{eqBoundBL}
    \| B(t, g) - B(t ,h)  \|_G & \leq C_B \|  g  - h \|_G
\end{align}

\item There is a constant $C_F > 0$ such that for all $t \in [0,T]$, $\omega \in \Omega$ and $g, h:[0,T] \to G$ c\`adl\`ag we have
\begin{align}
\label{eqBoundIntFQ}
    & \int_0^t\! \! \int_U \| F(s, u, g(s)) \circ Q^{1/2}_{M} \|_{\mathcal{HS}(H,G)}^2 \operQuadraVari{M}(ds, du)  \leq C_F\left(1 + \int_0^t \|  g(s) \|^2 \, ds \right) \\
    \begin{split}
\label{eqLipIntFQ}
        \int_0^t\! \! \int_U \| ( F(s, u, g(s)) - F(s, u, h(s)) )  \circ Q^{1/2}_{M} \|_{\mathcal{HS}(H,G)}^2 \operQuadraVari{M}(ds, du) \\ \leq C_F \int_0^t \|  g(s) -h(s)\|^2 \, ds
    \end{split}
\end{align}

\end{enumerate}
   
\end{assumption}

The following is the main result of this section:

\begin{theorem}\label{theoExisteUniqueness}
For a fixed $G$-valued, $\mathcal{F}_0$ measurable square integrable random variable $X_0$, the stochastic evolution equation \eqref{EqSPDE} has a unique mild solution $X=(X_{t}: t \geq 0)$ with initial value $X_0$. Moreover, for every $T>0 $ we have $\displaystyle \Exp \left[\int_{0}^{T} \norm{X_{t}}^{2}_{G} dt \right]<0.$ 
Furthermore, $X=(X_{t}: t \geq 0)$  is the unique weak solution to \eqref{EqSPDE}.  
\end{theorem}



\begin{proof}
We show first the existence of a unique mild solution to  \eqref{EqSPDE}. It suffices to prove the existence of a unique mild solution $X^{T}$ on the bounded time interval $[0,T]$ for $T>0$. Indeed, in such a case one can select a sequence $(T_{n}:n \in \N)$ of positive real numbers such that $T_{n} \nearrow \infty$ and define $X=(X_{t}: t \geq 0)$ by $X_{t}=X^{T_{n}}_{t}$ for $T_{n-1} \leq t < T_{n}$, where $T_{0}=0$.

Let $T>0$. Consider the space
$$
V := L^2\left( \Omega\times [0,T],\mathcal{P}_T,\Prob\otimes Leb;G\right).
$$
That is, the $G$-valued predictable processes $X=(X_{t}: 0\leq t\leq T)$ such that
$$
\|  X \|_{V}^2 : = \Exp\left[\int_0^T \norm{X_t}_G^2 dt \right] < \infty.
$$
To construct the mild solution, we consider $R:V\fle V$ defined as
\begin{equation}
 \label{eqDefOperContractR}
  R(X)_t = S(t)X_0 + \int_0^t S(t-s)B(s,X_s)ds + \int_0^t\! \! \int_U S(t-s) F(s,u,X_s)M(ds,du).  
\end{equation}

First of all, we must verify that $R$ is well defined. Consider $X\in V$ and observe that
$$
\Exp \int_0^T \norm{S(t)X_0}^2_G \leq N^2Te^{2\kappa T}\Exp \norm{X_0}_G^2 < \infty,
$$
where we have used (\ref{EqSemigroupbound}). As for the second term, let us call it
$$
R_1(X)_t := \int_0^t S(t-s)B(s,X_s)ds.
$$
With the help of (\ref{EqSemigroupbound}) and (\ref{eqBoundBB}), it is straightforward to verify that
$$
\norm{R_1(X)_t}_G^2 \leq N^2 e^{2\kappa T} C_B^2 \int_0^T (1+\norm{X_s}_G)^2 ds \leq 2N^2 e^{2\kappa T} C_B^2 \int_0^T (1+\norm{X_s}_G^2) ds
$$
and therefore
$$
\Exp\int_0^T \norm{R_1(X)_t}^2_G dt \leq 
2TN^2 e^{2\kappa T} C_B^2  \left( T + \Exp \int_0^T \norm{X_s}_G^2 ds \right) < \infty.
$$
Now consider the third term
$$
R_2(X)_t := \int_0^t\! \! \int_U S(t-s) F(s,u,X_s)M(ds,du).
$$
In this case, if we use (\ref{EqSemigroupbound}) and (\ref{eqBoundIntFQ}) we have 
\begin{eqnarray*}
   \Exp \norm{R_2(X)_t}^2_G & = & \Exp \int_0^t\! \! \int_U \| S(t-s)F(s, u, X_s) \circ Q_M^{1/2}(s, u) \|_{\mathcal{HS}(H,G)}^2 \operQuadraVari{M}(ds,du)  \\
    &\leq & N^2 \Exp \int_0^t\! \! \int_U e^{2\kappa(t-s)} \| F(s, u, X_s) \circ Q_M^{1/2}(s, u) \|_{\mathcal{HS}(H,G)}^2 \operQuadraVari{M}(ds,du) \\
    &\leq& N^2 e^{2\kappa t} C_F\left(1+\Exp \int_0^T \| X_s \|_G^2 \, ds\right)
\end{eqnarray*}
It follows that
$$
\Exp\int_0^T \norm{R_2(X)_t}^2_G dt \leq
TN^2 e^{2\kappa t} C_F\left(1+\Exp \int_0^T \| X_s \|_G^2 \, ds\right) < \infty.
$$
It remains to be shown that $R_1(X)$ and $R_2(X)$ each have a predictable version. For $R_1(X)$ the result is classic. In fact, since $S$ is continuous in $t$ and $B$ is Bochner-integrable for $\Prob$-a.e. $\omega\in\Omega$, we have $R_1(X)$ is adapted and has $\Prob$-a.e. continuous paths, therefore it is predictable. On the other hand, applying Theorem \ref{theoPredictableStochConvolu} to $\Phi(\omega,s,u) = F(s,u,X_s)$ we get the desired predictability of $R_2(X)$. Hence, we conclude that $R$ is well defined from $V$ to $V$. \medskip


Now to show that \eqref{EqSPDE} has a unique mild solution on the time interval $[0,T]$, we must show that $R$ has a unique fixed point. If we can show that $R$ is contractive, the result will be a  direct application of Banach's fixed point theorem. 

Given the above, we need to address the contractivity of the operator $R$. For $\beta > 0$ (to be chosen later) consider the weighted $L^2$ space
$$
V_{\beta} := L^2\left( \Omega\times [0,T], \mathcal{P}_T, w_{\beta}\cdot  \Prob\otimes Leb;G\right),
$$
where the weight is given by $w_{\beta}(t) = e^{-\beta t}$. That is, $V_{\beta}$ is the space of predictable processes $X = (X_{t}: 0 \leq t \leq T)$ that satisfy
$$
\| X \|_{V_{\beta}} : = \Exp \left[\int_{0}^T e^{-\beta t} \| X_t \|_G^2 \, dt \right]< \infty.
$$
Note that since $e^{-\beta T} \leq w_{\beta} \leq 1$, the norms $\| \cdot \|_V$ and $\| \cdot \|_{V_{\beta}}$ are equivalent. In particular, $V$ and $V_{\beta}$ have the same elements ($V_{\beta} = V$ as sets).

\textbf{Claim:} \emph{
    The parameter $\beta$ can be chosen so that the operator $R$ is a contraction from $V_{\beta}$ to $V_{\beta}$, that is, there is a constant $0 < C < 1$ such that, for all $X,Y \in {V_{\beta}}$
    $$
    \| R(X) - R(Y) \|_{V_{\beta}} \leq C \| X - Y \|_{V_{\beta}}.
    $$}

To prove the Claim we collect the estimates for $R_{1}$ and $R_{2}$.  For $R_1$, by \eqref{EqSemigroupbound}, (\ref{eqBoundBL}) and some standard computations, we have
\begin{align*}
    \| R_1(X) - R_1(Y) \|_{V_{\beta}}^2 & = \left\| \int_0^t S(t-s)B(s,X_s)ds - \int_0^t S(t-s)B(s,Y_s)ds \right\|_{V_{\beta}}^2 \\
    &= \left\| \int_0^t S(t-s) (B(s,X_s) - B(s,Y_s) )ds \right\|_{V_{\beta}}^2 \\
    &= \Exp \left[\int_0^T e^{-\beta t} \left\| \int_0^t S(t-s) (B(s,X_s) - B(s,Y_s) )\,ds \right\|_G^2 \, dt\right] \\
    & \leq \int_0^T e^{-\beta t} \Exp \left[ \int_0^t \| S(t-s) (B(s,X_s) - B(s,Y_s) ) \|_G \, ds\right]^2 \, dt \\
    & \leq N^2C_B^2\int_0^T e^{-\beta t}\Exp \left[ \int_0^t e^{\kappa t} \|  X_s - Y_s  \|_G \, ds\right]^2 \, dt \\
    & \leq N^2C_B^2 \int_0^T e^{-\beta t} \Exp \left[ \left( \int_0^t e^{2\kappa t} \, ds \right) \left( \int_0^t \| X_s - Y_s \|_G^2 \, ds \right) \right] \, dt \\
    & \leq N^2 C_B^2 T  e^{2\kappa T} \int_0^T e^{-\beta t} \Exp \left[ \int_0^t \| X_s - Y_s \|_G^2 \, ds\right] \, dt \\ 
    & = N^2 C_B^2 T  e^{2\kappa T} \Exp \left[\int_0^T e^{-\beta s} \| X_s - Y_s \|_G^2 \left( \int_s^T e^{\beta(s-t)} \, dt \right)\, ds \right]\\ 
    & \leq \frac{N^2 C_B^2 T  e^{2\kappa T}}{\beta} \Exp \left[\int_0^T e^{-\beta s} \| X_s - Y_s \|_G^2 \, ds \right]\\
    &=  \frac{N^2 C_B^2 T  e^{2\kappa T}}{\beta}  \| X - Y \|_{V_{\beta}}^2.
\end{align*}

We now estimate $\| R_2(X) - R_2(Y) \|_{V_{\beta}}^2$, using the isometry (\ref{eqIsometryGeneralPhi}), \eqref{EqSemigroupbound} and condition (\ref{eqLipIntFQ}),
\begin{align*}
    &\left\| \int_0^t\! \! \int_U  S(t-s)F(s, u, X_s)\, M(ds, du) -  \int_0^t\! \! \int_U  S(t-s)F(s, u, Y_s)\, M(ds, du) \right\| _{V_{\beta}}^2 \\
    &= \left\| \int_0^t\! \! \int_U  S(t-s)(F(s, u, X_s)-F(s,u, Y_s) )\, M(ds, du)\right\| _{V_{\beta}}^2 \\
    &= \Exp \left[\int_0^T e^{-\beta t} \left\| \int_0^t\! \! \int_U  S(t-s)(F(s, u, X_s)-F(s,u, Y_s) )\, M(ds, du)\right\| _{G}^2 \, dt \right] \\
    &= \int_0^T e^{-\beta t}  \Exp \left[\int_0^t\! \! \int_U  \left\| S(t-s)(F(s, u, X_s)-F(s,u, Y_s) ) \circ Q_M^{1/2} \right\|_{\mathcal{HS}(H,G)}^2 \, \operQuadraVari{M}(ds, du)\,\right]dt \\
    &\leq C_F N^2 \int_0^T e^{2\kappa t} e^{-\beta t} \Exp \left[\int_0^t \| X_s - Y_s \|_G^2 \, ds \, \right]dt \\
    &\leq C_F N^2 e^{2\kappa T} \int_0^T  \Exp \left[\int_0^t e^{-\beta s} e^{\beta(s-t)}\| X_s - Y_s \|_G^2 \, ds \, \right]dt \leq \frac{C_F N^2 e^{2\kappa T}}{\beta} \|  X - Y \|_{V_\beta}^2.
\end{align*}
By choosing $\beta$ large enough in both estimates, we obtain the result in our claim. 

Therefore by our Claim we conclude that $R$ is contractive and hence has a unique fixed point in $V$, thus  \eqref{EqSPDE} has a unique mild solution on the time interval $[0,T]$. By our explanation at the beginning of the proof and our definition of $V$ we have 
\eqref{EqSPDE} has a unique mild solution $X=(X_{t}: t \geq 0)$ satisfying $\displaystyle \Exp \left[\int_{0}^{T} \norm{X_{t}}^{2}_{G} dt \right] <0$  for every $T>0$. 

Finally, that $X=(X_{t}: t \geq 0)$ is the unique weak solution to \eqref{EqSPDE} is obtained as a direct consequence of Theorems \ref{theoEquivaWeakMildSol} and \ref{theoExisteUniqueness}. 
\end{proof}

\begin{example}\label{examplExisteUniqSPDE}
We now consider the equation
\begin{equation}\label{eqExampleExisUniqSPDE}
dY_t = A Y_t \, dt + \int_U F(Y_t) \, M(dt, du),
\end{equation}
where $F: G \to \mathcal{HS}(H, G)$ satisfies the following Lipschitz-type condition: There exists a constant $C > 0$ such that for every $g_1, g_2 : [0, T] \to G$, and $s \in [0, T]$ we have
\begin{equation}\label{eqExamLipschitzF}
\| F(g_1(s)) - F(g_2(s))\|_{\mathcal{HS}(H, G)} \leq C \| g_{1}(s) - g_{2}(s) \|_G^2.
\end{equation}
We further assume that the quadratic variation $\operQuadraVari{M}$ is bounded by a product measure $\textup{Leb} \otimes \lambda$, where $\lambda$ is a $\sigma$-finite random measure on $U$. Assume moreover that $Q_{M}$ is constant in the time variable, i.e. $Q_{M}(\omega,t,u)=Q_{M}(\omega, u)$ for every $t \geq 0$, and that there is a constant $K_{M} > 0$ such that $ \int_{U} \, \|Q^{1/2}_{M}(\omega, u)\|^{2} \, \lambda(du) \leq K_M$ for all $\omega \in \Omega$. These assumptions hold true, for example, for cylindrical martingale-valued measures determined by the cylindrical and $H$-valued L\'{e}vy processes (see Examples
\ref{exampleCylinMartingaleDeltaLocBounCova} and \ref{exampleLevyHValuedMVM}).  

 Now, by \eqref{eqExamLipschitzF} we have the linear growth condition
 $$ \| F(g(s)) \|_{\mathcal{HS}(H,G)}  \leq \| F(0)\|_{\mathcal{HS}(H, G)} + C \| g(s) \|_G. $$
 Then for every $t \in [0,T]$, $g:[0,T] \rightarrow G$ c\`{a}dl\`{a}g we have
\begin{flalign*}
& \int_0^t\! \! \int_U \| F(g(s))   \circ Q^{1/2}_{M}(u) \|_{\mathcal{HS}(H,G)}^2 \operQuadraVari{M}(ds, du) \\
& \leq  \left(\int_{0}^{t} \| F(g(s)) \|^{2}_{\mathcal{HS}(H, G)} dt \right) \left( \int_{U} \, \norm{Q^{1/2}_{M}(u)}^{2} \, \lambda(du)  \right)   \leq  C^{1}_F \left( 1+ \int_0^t  \| g(s) \|_G^2 \, ds \right),    
\end{flalign*}
for $C^{1}_F= K_{M} \cdot \max\{ T \| F(0)\|_{\mathcal{HS}(H, G)}^{2}, C^{2} \} $. 

Likewise, for every $t \in [0,T]$, $g_{1},g_{2}:[0,T] \rightarrow G$ c\`{a}dl\`{a}g we have by \eqref{eqExamLipschitzF} that 
\begin{multline*}
\int_0^t\! \! \int_U \| ( F(g_{1}(s)) - F(g_{2}(s)) )  \circ Q^{1/2}_{M}(u) \|_{\mathcal{HS}(H,G)}^2 \operQuadraVari{M}(ds, du) 
 \leq  C^{2}_F \int_0^t  \| g_{1}(s) - g_{2}(s) \|_G^2 \, ds,    
\end{multline*}
for $ C^{2}_F = K_{M} C^{2}$. Then, by Theorem \ref{theoExisteUniqueness} the stochastic evolution equation \eqref{eqExampleExisUniqSPDE} has a unique weak solution given by:
$$   X_t = S(t)X_0 + \int_0^t\! \! \int_U S(t-s) F(X_s) M(ds,du).$$
\end{example}

\medskip

\appendix

\section{Proof of theorem \ref{theoExistMeasMuHValuedMVM}}\label{appendix}

For the reader's convenience we divide the proof in five steps:
\begin{enumerate}
\item[Step 1.] Since $(U,\calB(U))$ is Lusin, there exists $h:U\fle \hat{U}$ invertible, where $\hat{U}$ is a Borel set on the line, such that both $h$ and $h^{-1}$ are measurable (relative to $\calB(U)$ and $\calB(\hat{U})$). All the elements that define $M$ can be copied down on $(\hat{U},\calB(\hat{U}))$, so it is enough to prove the theorem for $\hat{U}$.

\item[Step 2.] According to the previous step, let us assume that $U$ is a Borel set on the line. Notice that $\mu$ is a finite measure on each $(U_n,\calB(U_n))$. Suppose we can prove the theorem for $U=U_n$ and $\calA = \calB(U_n)$. That means there is a predictable family $(\nu^{(n)}_{t}: 0\leq t \leq T)$ of random measures on $(U_n,\calB(U_n))$, satisfying (i)-(iii) for $A\in \calB(U_n)$. Suppose, moreover, that $\nu^n_t(C)\leq \nu^{n+1}_t(C)$ for each $C\in \calB(U_n)$. We extend each $\nu^n_t$ to a measure $\hat{\nu}^n_t$ on $(U,\calB(U))$ by 
$$
\hat{\nu}^n_t(C) = \nu^n_t(C\cap U_n)
$$
and then define
$$
\nu_t := \sup_{n\geq 1} \hat{\nu}^n_t.
$$
According to Corollary \ref{CoroSupPredRandMeas}, 
$(\nu_{t}: 0\leq t \leq T)$ is a predictable family of random measures, and for each $C\in \calB(U)$ we have 
$$
\nu_t(C) = \lim_{n\fle\infty} \hat{\nu}^n_t(C) = \lim_{n\fle\infty} \nu^n_t(C\cap U_n).
$$
Since $\nu^n_t$ satisfies (iii) in $\calB(U_n)$, given $A\in \calA$, $\norm{M_t(A\cap U_n)}^2-\nu^n_t(A\cap U_n)$ is a martingale. Since $M_t(A\cap U_n) \fle M_t(A)$ in $L^2$ and $\nu^n_t(A\cap U_n) \fle \nu_t(A)$ in $L^1$ (monotone convergence) we conclude that 
$$
\norm{M_t(A\cap U_n)}^2-\nu^n_t(A\cap U_n) \fle \norm{M_t(A)}^2-\nu_t(A)
$$
in $L^1$. This shows that $\norm{M_t(A)}^2-\nu_t(A)$ is a martingale and then 
$$
\Prob (\nu_t(A) = \quadraVari{M(A)}_t ) = 1.
$$

\item[Step 3.] By the previous step, it is enough to assume that $\calA = \calB(U)$ and $\mu$ is a finite measure. Because of the regularity of $\mu$, there is an increasing family of compact sets $K_j \subseteq U$ such that $ K = \cup K_j$ satisfies $\mu( U\backslash K ) = \mathbb{E} \quadraVari{M(U\backslash K)}_T = 0$. Since $\quadraVari{M(U\backslash K)}_T$ is a non-negative random variable, we conclude $\quadraVari{(M(U\backslash K)}_T = 0$ a.e. If we prove the theorem for $U=K_j$, we can apply the construction of the second step to obtain a family of random measures $\nu_t$ defined on $(U,\calB(U))$ and concentrated on $K$. Since $U \backslash K$ has $\mu$-measure zero,
$$
\Prob (\nu_t(A) = \quadraVari{M(A)}_t) = \Prob(\nu_t(A \cap K)=\quadraVari{M(A\cap K)}_t)=1
$$
for all $A\in \calB(U)$ and all $t\in [0,T]$.

\item[Step 4.] Following the reduction in Step 3, we will now prove the theorem for $U$ compact on the real line,  ${\mathcal A} = {\mathcal B}(U)$ and $\mu(U)<\infty$. We separate this step in several sub-steps. 
\begin{enumerate} 
   \item[Step 4.1.] We extend $M$ to $(\R,{\mathcal B}_\R)$ so that $M_t(A) = M_t(A\cap U)$. In particular, $M_t(A) = 0$ for $A\cap U = \emptyset$. For each $t\in [0,T]$ define the random function $F_t:\R\rightarrow \R$ by
   $$
   F_t(x) = \quadraVari{M(-\infty,x]}_t.
   $$
   By Lemma \ref{lemmaAdditiveVariation}, each $F_t$ is increasing and, for $a < b$ fixed, $F_{t}(b) - F_{t}(a)$ is increasing in $t$. Also,  for fixed $t$ and $x\in (a,b]$ we have $\Prob$-a.e.
      $$
   F_t(x)-F_t(a) \leq \quadraVari{M((a,b])}_T,
   $$
    and by right-continuity this holds $\Prob$-a.e. simultaneously for all $t\in [0,T]$ and $x\in (a,b]\cap\Q$. Taking the expectation, we get
   \begin{equation}
   \label{eq_expect_sup}
   E \left[ \sup_{t\in [0,T],x\in (a,b]\cap \Q}\left| F_t(x)-F_t(a)  \right|  \right] \leq \mu((a,b]).
   \end{equation}

   \item[Step 4.2.] Define
   $$
   \overline{F}_t(x) := \inf\{ F_s(y): x<y\in\Q,t<s\in\Q \}.
   $$
   This function is increasing and right-continuous in both variables. The first assertion is immediate; for the second one, consider rational sequences $t_n\downarrow t$ and $x_n\downarrow x$ and note that
   $$
   \overline{F}_t(x) \leq \overline{F}_{t_n}(x_n) \leq F_{t_{n-1}}(x_{n-1}).
   $$
   We also observe that, for fixed $x$
   $$
   \Prob \left( \overline{F}_t(x) = F_t(x) \text{ for all } t\in [0,T] \right) = 1.
   $$
In fact, for $t_n\downarrow t$ and $x_n\downarrow x$ we have
$$
F_t(x) \leq \overline{F}_t(x) \leq F_{t_n}(x_n) = F_{t_n}(x) + [F_{t_n}(x_n) - F_{t_n}(x) ] 
$$
where $F_{t_n}(x) \fle F_{t}(x)$ by right-continuity and $F_{t_n}(x_n) - F_{t_n}(x) \fle 0$ in probability by (\ref{eq_expect_sup}). 

\item[Step 4.3.] Let $\nu_t$ be the distribution on $\R$ generated by $\overline{F}_t$. Note that $\nu_t(\R \backslash U) =0$, for given $(\alpha,\beta] \subseteq \R\backslash U$, there exists $\varepsilon >0$ such that $(\alpha,\beta+\varepsilon] \subseteq  \R\backslash U$. This implies that $\overline{F}_t$ is constant on $(\alpha,\beta]$ and therefore $\nu_t((\alpha,\beta])=0$. The predictability is inherited from $F_t$ to $\overline{F}_t$ and then to $\nu_t$, as we can always work through rational values of $x$.

\item[Step 4.4.] For fixed $t<t'$, $\overline{F}_{t'}-\overline{F}_{t}$ is increasing. In fact, for $a<b$ the inequality
$$
(\overline{F}_{t'}-\overline{F}_{t})(a) < (\overline{F}_{t'}-\overline{F}_{t})(b)
$$
is equivalent to
$$
\overline{F}_{t}(b)-\overline{F}_{t}(a) < \overline{F}_{t'}(b)-\overline{F}_{t'}(a)
$$
and this is inherited from $F$. It follows that $\nu_{t}((a,b])$ is increasing and right-continuous in $t$. A monotone class argument allows us to deduce that $\nu_{t}(A)$ is increasing and right-continuous in $t$ for any Borel set $A$.

\item[Step 4.5.] We now prove $(iii)$. Let $\mathcal{G}$ be the class of those $A\in \calB(U)$ for which $\norm{M_t(A)}^2 - \nu_t(A)$ is a martingale. It is clear that $(-\infty,x] \in \mathcal{G}$, because by Step 2, $\nu_t((-\infty,x]) = F_t(x) = \quadraVari{M(-\infty,x]}_t$ a.e. With $x =\sup U$ we obtain $\R\in \mathcal{G}$. It follows that $\mathcal{G}$ contains $\R$ and all finite unions of intervals of the form $(a,b]$. $\mathcal{G}$ is closed under complementation because we can write $\norm{M(A^c)}^2_t - \nu_t(A^c)$ as the sum of three martingales:
$$
[\norm{M(\R)}^2_t - \nu_t(\R)] + [\norm{M(A)}^2_t - \nu_t(A)] - 2 \left( M_t(A),M_t(A^c) \right)_H.
$$
Finally, $\mathcal{G}$ is closed under monotone convergence, for if $A_n\uparrow A$ we have in $L^1$
$$
\norm{M(A_n)}^2_t - \nu_t(A_n) \fle \norm{M(A)}^2_t - \nu_t(A).
$$
We conclude that $\mathcal{G}=\calB(U)$. 
\end{enumerate}
\end{enumerate}
To complete the proof, in order for Steps 2 and 3 to work out correctly, we need to observe that if $U_1$ and $U_2$ are compact and $U_1 \subseteq U_2$, then the constructed families $\nu^1_t$ and $\nu^2_t$ satisfy $\nu^1_t \leq \nu^2_T$ in $U_1$, in the sense that 
$$
\mathbb{P}\left(\nu^1_t(C) \leq \nu^2_t(C) \text{ for all } C\in \calB(U_1) \right) = 1.
$$
It is enough to verify it for any interval $(a,b]$, with $a,b$ rational, and that is straightforward following the construction of Step 4.

\noindent \textbf{Acknowledgments}  This work was partially supported by The University of Costa Rica through the grant ``C3019- C\'{a}lculo Estoc\'{a}stico en Dimensi\'{o}n Infinita''. The authors thank an anonymous referee for valuable comments and
suggestions that contributed greatly to improve the presentation of this article

\smallskip

\noindent \textbf{Data Availability.} Data sharing not applicable to this article as no data sets were generated or analyzed during the current study.

\section*{Declarations}

\noindent \textbf{Conflict of interest} The authors have no conflicts of interest to declare that are relevant to the content of this article.


\begin{thebibliography}{HD}




\bibitem{AlvaradoFonseca:2021} Alvarado-Solano, A. E.; Fonseca-Mora, C. A. \emph{Stochastic integration in Hilbert spaces with respect to cylindrical martingale-valued measures}, ALEA Lat. Am. J. Probab. Math. Stat. 18, no. 2, 1267--1295 (2021).

\bibitem{Applebaum:2006} Applebaum, D.: \emph{Martingale-valued measures, Ornstein–-Uhlenbeck processes with jumps and operator self-decomposability in Hilbert spaces}. In Memoriam Paul-Andr\'{e} Meyer. S\'{e}minaire de Probabilit\'{e}s  39, Lecture Notes in Mathematics 1874, 171--197 (2006).

\bibitem{ApplebaumLPSC} Applebaum, D.: Lévy processes and stochastic calculus. Second edition. Cambridge Studies in Advanced Mathematics, 116. Cambridge University Press, Cambridge (2009).

\bibitem{ApplebaumRiedle:2010} Applebaum, D.; Riedle, M.: \emph{Cylindrical Lévy processes in Banach spaces}, 
Proc. Lond. Math. Soc. (3) 101, no. 3, 697--726 (2010).


\bibitem{BennettSarpley:1988} Bennett, C.; Sharpley, R.C.: \emph{Interpolation of operators}, Pure and Applied Mathematics, 129, Academic Press, Boston, MA, (1988).

\bibitem{BrzezniakLiuZhu:2014} Brze\'{z}niak, Z.; Liu, W.;  Zhu, J.: \emph{Strong solutions for SPDE with locally monotone coefficients driven by L\'{e}vy noise}, Nonlinear Anal. Real World Appl, 17, 283--310 (2014).


\bibitem{CCFM:SUP} Cambronero, S.; Campos, D.; Fonseca-Mora, C.A.; Mena, D: \emph{On the supremum of a family of random signed measures}.   Unpublished manuscript.  ArXiv e-prints (2023), available at \href{https://arxiv.org/abs/2308.10914}{2308.10914}. 



\bibitem{CCFM:Ito} Cambronero, S.; Campos, D.; Fonseca-Mora, C.A.; Mena, D: \emph{It\^o's Formula for It\^o processes defined with respect to a cylindrical-martingale valued measure}.  Unpublished manuscript. ArXiv e-prints (2024), available at \href{https://arxiv.org/abs/2407.16086}{2407.16086}. 


\bibitem{ChongKevei:2019} Chong, C. and Kevei, P.: \emph{Intermittency for the stochastic heat equation with Lévy noise}, Ann. Probab., 47 (4), 1911–1948 (2019).

\bibitem{CohenElliott:2015} Cohen, S. N.; Elliott, R. J.: \emph{Stochastic calculus and applications}. Second edition. Probability and its Applications. Springer, Cham (2015).

\bibitem{ConusDalang:2008} Conus, D. and Dalang, R. C.: \emph{The non-linear stochastic wave equation in high dimensions}, Electron.
J. Probab., 13, no. 22, 629–670 (2008).


\bibitem{Dalang:1999} Dalang, R. C.: \emph{Extending the martingale measure stochastic integral with applications to spatially homogeneous s.p.d.e.’s}, Electron. J. Probab., 4, no. 6, 29 (1999).

\bibitem{DalangMueller:2003} Dalang, R. C.; Mueller, C.: \emph{Some non-linear S.P.D.E.’s that are second order in time}, Electron. J. Probab. 8, no. 1, 21 pp (2003). 


\bibitem{DiestelUhl:1977} Diestel, J.; Uhl, J. J., Jr: Vector measures. Mathematical Surveys, No. 15. American Mathematical Society, Providence, R.I.  (1977).


\bibitem{DiGirolamiFabbriRusso:2014} Di Girolami, C.; Fabbri, G.; Russo, F.: \emph{The covariation for Banach space valued processes and applications}. Metrika, 77, 51--104 (2014).

\bibitem{Dinculeanu:2000} Dinculeanu, N.: \emph{Vector integration and stochastic integration in Banach spaces}. Pure and Applied Mathematics, Wiley-Interscience, New York, (2000).

\bibitem{Dirksen:2014} Dirksen, S.: \emph{It\^o{} isomorphisms for {$L^p$}-valued {P}oisson stochastic integrals}. The Annals of Probability 42. no.6, 2595–2643., (2014).


\bibitem{FerrarioOlivera:2019} Ferrario, B.; Olivera, C.: \emph{2D Navier–Stokes equation with cylindrical fractional Brownian noise}. Annali di Matematica 198, 1041--1067 (2019).

\bibitem{Folland:1999} Folland, G.B. \emph{Real Analysis. Modern techniques and their applications.} Pure and Applied Mathematics, Wiley-Interscience, New York, (1999).

\bibitem{FonsecaMora:2018} Fonseca-Mora, C. A.: \emph{Stochastic Integration and Stochastic PDEs Driven by Jumps on the Dual of a Nuclear Space}. Stoch PDE: Anal Comp,   6, no.4,  618--689 (2018).

\bibitem{FoondunKhoshnevisan:2009} Foondun, M.; Khoshnevisan, D.: \emph{Intermittence and nonlinear parabolic stochastic partial differential equations}. Electron. J. Probab., 14, no. 21, 548–568 (2009)


\bibitem{Hytonen:2016} Hyt\"{o}nen, T.; van Neerven, J.; Veraar, M.; Weis, L.: \emph{Analysis in Banach spaces. Vol. I. Martingales and Littlewood-Paley theory}. Ergebnisse der Mathematik und ihrer Grenzgebiete. 3. Folge. A
Series of Modern Surveys in Mathematics, 63, Springer, Cham (2016).

              
\bibitem{Kallenberg:2021} Kallenberg, O.: \emph{Foundations of modern probability}. Third edition. Probability Theory and Stochastic Modelling, 99. Springer, Cham, (2021).

\bibitem{KallianpurXiong} Kallianpur, G.; Xiong, J.: \emph{Stochastic differential equations in infinite-dimensional spaces}, IMS Lecture Notes Monogr. Ser., 26
Institute of Mathematical Statistics, Hayward CA (1995)

\bibitem{KarandikarRao:2018} Karandikar, R. L.; Rao, B. V. \emph{Introduction to stochastic calculus}. Indian Statistical Institute Series, Springer, Singapore (2018).

\bibitem{KarouiMeleard:1990} Karoui, N.E.; M\'{e}l\'{e}ard, S.: \emph{Martingale measures and stochastic calculus}. Probab. Th. Rel. Fields 84, 83–101 (1990).

\bibitem{Kluvanek:1972} Kluv\'{a}nek, I.: \emph{The extension and closure of vector measure}. Vector and operator valued measures and applications (Proc. Sympos., Alta, Utah, 1972), pp. 175–190. Academic Press, New York (1973).

\bibitem{KumarRiedle:2020} Kumar, U.; Riedle, M.: \emph{The stochastic Cauchy problem driven by a cylindrical L\'{e}vy process}. Electron. J. Probab., Volume 25, paper no. 10 (2020).  

\bibitem{Liu-Rockner:2015} Liu, W.; R\"{o}ckner, M.:  \emph{Stochastic partial differential equations: An introduction}. Springer Universitext. Springer international publishing Switzerland (2015).

\bibitem{LiuZhai:2016} Liu, Y.; Zhai, J.: \emph{Time regularity of generalized Ornstein-Uhlenbeck processes with L\'{e}vy noises in Hilbert spaces}. J. Theoret. Probab.,  {29}, No.3, 843--866 (2016).

\bibitem{MandrekarRudiger} Mandrekar, V. ; R\"{u}diger, B.: 
\emph{Stochastic integration in Banach spaces. Theory and applications}, 
Probab. Theory Stoch. Model. 73, Springer, Cham (2015).

\bibitem{MetivierPellaumail} M\'{e}tivier, M; Pellaumail, J.:  \emph{Stochastic integration}. Probability and Mathematical Statistics, Academic Press, New York (1980).

\bibitem{Metivier} M\'{e}tivier, M.: \emph{Semimartingales. A course on stochastic processes}. de Gruyter Studies in Mathematics, 2. Walter de Gruyter \& Co., Berlin-New York (1982).

\bibitem{MikuleviciusRozovskii:1998} Mikulevicius, R.; Rozovskii, B. L.: \emph{Normalized stochastic integrals in topological vector spaces}.  S\'{e}minaire de Probabilit\'{e}s XXXII, Lecture Notes in Mathematics, vol 1686, Springer, Berlin (1998).

\bibitem{Ondrejat:2005} Ondrejat, M.: \emph{Brownian Representations of Cylindrical Local Martingales, Martingale Problem and Strong Markov Property of Weak Solutions of SPDEs in Banach Spaces}. Czech Math J 55, 1003–1039 (2005).


\bibitem{Parthasarathy:1978} Parthasarathy, K.R.: \emph{Introduction to Probability and Measure}. The Macmillan Company
of India Ltd., Delhi (1977). Springer-Verlag, New York (1978).

\bibitem{PeszatZabczykSPDE} Peszat, S.; Zabczyk, J.: \emph{Stochastic partial differential equations with Lévy noise. An evolution equation approach}. Encyclopedia of Mathematics and its Applications, 113, Cambridge University Press, Cambridge (2007).


\bibitem{PriolaZabczyk:2011} Priola, E.; Zabczyk, J.:  \emph{Structural properties of semilinear SPDEs driven by cylindrical stable processes}. Probab. Theory Relat. Fields, {149}, no.1-2, 97--137 (2011).

\bibitem{Radchenko:1989} Radchenko, V. N.: \emph{The Radon-Nikodým theorem for random measures}. (Russian) Ukrain. Mat. Zh.41 (1989), no.1, 63–67, 135; translation in Ukrainian Math. J.41, no.1, 57–61  (1989).


\bibitem{Riedle:2015}
Riedle, M.: \emph{Ornstein-Uhlenbeck processes driven by cylindrical L\'{e}vy processes}. Potential Anal., {42}, no.4, 809--838, (2015).

\bibitem{Riedle:2018}
Riedle, M.: \emph{Stable cylindrical L\'{e}vy processes and the stochastic Cauchy problem}. Electron. Commun. Probab., 23, Paper No. 36, 12, (2018).

\bibitem{RiedlevanGaans:2009} Riedle, M., van Gaans, O.: \emph{Stochastic integration for Levy processes with values in Banach spaces}, Stoch. Process. Their Appl., {119}, no.6,  1952--1974 (2009). 


\bibitem{RozovskyLototsky:2018} Rozovsky, B. L., Lototsky, S. V.: \emph{Stochastic evolution systems. Linear theory and applications to non-linear filtering}, Second edition, Probability Theory and Stochastic Modelling, 89, Springer, Cham (2018).

\bibitem{SunXieXie:2020} Sun, X.; Xie, L.; Xie, Y.: \emph{Pathwise Uniqueness for a Class of SPDEs Driven by Cylindrical $\alpha$-Stable Processes}. Potential Anal., {53}, no. 2, 659--675 (2020).


\bibitem{VakhaniaTarieladzeChobanyan} Vakhania, N. N.; Tarieladze, V. I.; Chobanyan, S. A.: \emph{Probability distributions on Banach spaces}.  Mathematics and its Applications (Soviet Series), 14. D. Reidel Publishing Co., Dordrecht (1987).

\bibitem{VeraarYaroslavtsev:2016} Veraar, M.; Yaroslavtsev, I.: \emph{Cylindrical continuous martingales and stochastic integration in infinite dimensions}. 
Electron. J. Probab. 21, Paper No. 59, 53 pp.  (2016).

\bibitem{Walsh:1986} Walsh, John B.: \emph{An introduction to stochastic partial differential equations}. \'{E}cole d'\'{e}t\'{e} de probabilit\'{e}s de Saint-Flour, XIV—1984, 265–439, Lecture Notes in Math., 1180, Springer, Berlin (1986).

\bibitem{WangYangZhaiZhang:2024} Wang, J.; Yang, H.; Zhai, J.; Zhang, T.:
\emph{Accessibility of SPDEs driven by pure jump noise and its applications}, Proc. Amer. Math. Soc., {152}, no.4, 1755–-1767 (2024).

\bibitem{WuYangWangSong:2021} Wu, P.; Yang, Z.; Wang, H.; Song, R.: \emph{Time fractional stochastic differential equations driven by pure jump L\'{e}vy noise},
J. Math. Anal. Appl.,
{504}, no. 2, 125412–32 (2021). 


\end{thebibliography}
\end{document}